\documentclass[11pt,a4paper]{article}

\usepackage[utf8]{inputenc}
\usepackage[T1]{fontenc}
\usepackage{lmodern}
\usepackage{amsmath,amsfonts,amssymb,amsthm}
\usepackage{mathtools}
\usepackage{enumitem}
\usepackage{accents}
\usepackage[protrusion=true,expansion=true]{microtype}
\usepackage[margin=1in]{geometry}
\usepackage[normalem]{ulem}
\usepackage{hyperref}
\hypersetup{
    colorlinks=true,
    linkcolor=blue,
    citecolor=red,
    urlcolor=blue,
    pdfborder={0 0 0}
}            
\usepackage{cleveref}
  \crefname{section}{Section}{Sections}
  \crefname{figure}{Figure}{Figures}
  \crefname{theorem}{Theorem}{Theorems}
  \crefname{lemma}{Lemma}{Lemmas}
  \crefname{proposition}{Proposition}{Propositions}  
  \crefname{corollary}{Corollary}{Corollaries}
  \crefname{definition}{Definition}{Definitions}
  \crefname{example}{Example}{Examples}
  \crefname{remark}{Remark}{Remarks}
  \crefname{equation}{}{}

\newtheorem{theorem}{Theorem}[section]
\newtheorem{lemma}[theorem]{Lemma}
\newtheorem{proposition}[theorem]{Proposition}
\newtheorem{corollary}[theorem]{Corollary}

\theoremstyle{definition}

\newtheorem{remark}[theorem]{Remark}
\newtheorem{example}[theorem]{Example}

\DeclarePairedDelimiter\abs{\lvert}{\rvert}
\DeclarePairedDelimiter\norm{\lVert}{\rVert}

\newcommand*{\defeq}{\mathrel{\mathop:}=}

\newcommand*{\dbtilde}[1]{\tilde{\raisebox{0pt}[0.9\height]{$\tilde{#1}$}}}
\newcommand*{\mmiddle}[1]{\mathrel{}\middle#1\mathrel{}}

\renewcommand*{\Re}{\operatorname{Re}}
\renewcommand*{\Im}{\operatorname{Im}}

\newcommand*{\diam}{\operatorname{diam}}
\newcommand*{\dist}{\operatorname{dist}}
\newcommand*{\crad}{\operatorname{crad}}
\newcommand*{\hcap}{\operatorname{hcap}}

\renewcommand*{\fill}{\operatorname{fill}}

\newcommand*{\Cont}{\operatorname{Cont}}

\newcommand*{\bbC}{\mathbf{C}}
\newcommand*{\bbD}{\mathbf{D}}
\newcommand*{\bbE}{\mathbf{E}}
\newcommand*{\bbH}{\mathbf{H}}
\newcommand*{\bbN}{\mathbf{N}}
\newcommand*{\bbP}{\mathbf{P}}
\newcommand*{\bbR}{\mathbf{R}}

\newcommand*{\bbZ}{\mathbf{Z}}

\newcommand*{\hatC}{\widehat{\bbC}}

\newcommand*{\barD}{\overline{\bbD}}

\newcommand*{\sle}[1]{SLE$_{#1}$}
\newcommand*{\slek}{\sle{\kappa}}

\newcommand*{\logp}{\log^*}

\newcommand{\C}{\mathbf{C}}
\newcommand{\D}{\mathbf{D}}

\newcommand{\R}{\mathbf{R}}
\renewcommand{\P}{\mathbf{P}}

\newcommand{\eps}{\varepsilon}

\newcommand{\cG}{\mathcal{G}}
\newcommand{\cF}{\mathcal{F}}

\newcommand{\eqb}{\begin{equation}}
	\newcommand{\eqe}{\end{equation}}
\newcommand{\eqbn}{\begin{equation*}}
	\newcommand{\eqen}{\end{equation*}}

\newcommand{\BB}{\mathbf}
\newcommand{\ol}{\overline}
\newcommand{\ul}{\underline}
\newcommand{\op}{\operatorname}

\newcommand{\rta}{\rightarrow}

\newcommand{\wt}{\widetilde}
\newcommand{\wh}{\widehat} 
\newcommand{\mcl}{\mathcal}

\newcommand{\bdy}{\partial}

\renewcommand{\u}{{\boldsymbol{u}}}

\title{Regularity of the Schramm-Loewner evolution:\\ Up-to-constant variation and modulus of continuity}

\author{
Nina Holden\thanks{Courant Institute of Mathematical Sciences, New York University. Email: \texttt{nina.holden@nyu.edu}}
\and
Yizheng Yuan\thanks{TU Berlin, Germany \& U Cambridge, United Kingdom.
Email: \texttt{yy547@cam.ac.uk}}
}

\begin{document}

\maketitle

\begin{abstract}
We find optimal (up to constant) bounds for the following measures for the regularity of the Schramm-Loewner evolution (SLE): variation regularity, modulus of continuity, and law of the iterated logarithm. For the latter two we consider the SLE with its natural parametrisation. More precisely, denoting by $d\in(0,2]$ the dimension of the curve, we show the following.

1. The optimal $\psi$-variation is $\psi(x)=x^d(\log\log x^{-1})^{-(d-1)}$ in the sense that $\eta$ is a.s.\ of finite $\psi$-variation for this $\psi$ and not for any function decaying more slowly as $x \downarrow 0$.

2. The optimal modulus of continuity is $\omega(s) = c\,s^{1/d}(\log s^{-1})^{1-1/d}$, i.e.\ for some random $c>0$ we have $|\eta(t)-\eta(s)| \le \omega(t-s)$ a.s., while this does not hold for any function $\omega$ decaying faster as $s \downarrow 0$.

3. $\limsup_{t\downarrow 0} |\eta(t)|\,\big(t^{1/d}(\log\log t^{-1})^{1-1/d}\big)^{-1}$ is a.s.\ equal to a deterministic constant in $(0,\infty)$.

We also show that the natural parametrisation of SLE is given by the fine mesh limit of the $\psi$-variation. As part of our proof, we show that every stochastic process whose increments satisfy a particular moment condition attains a certain variation regularity.
\end{abstract}

\section{Introduction}
\label{sec:intro}

The Schramm-Loewner evolution (SLE) was first introduced by Schramm \cite{sch-sle} in 1999 as a candidate for the scaling limit of curves in statistical physics models at criticality. Soon afterwards it was proven that the SLE indeed describes the limiting behavior of a range of statistical physics models, including the uniform spanning tree, 
the loop-erased random walk \cite{lsw-lerw-ust}, 
percolation \cite{smi-cardy}, 
the Ising model \cite{smi-ising,cs-ising}, and 
the discrete Gaussian free field \cite{ss-dgff}. Schramm argued in his work that SLE is the unique one-parameter family of processes satisfying two natural properties called conformal invariance and the domain Markov property, and he denoted the parameter by $\kappa>0$.

In this paper we will study the regularity of SLE. We first present some measures of regularity for general fractal curves in Section \ref{sec:quantify} along with previous results for SLE. Then we state our main results of Section \ref{se:results}, we discuss consequences for discrete models in Section \ref{se:discrete}, and finally we give an outline of the paper in Section \ref{se:outline}.

\subsection{How to quantify the regularity of fractal curves, and previous SLE results}
\label{sec:quantify}

Let $I\subset\R$ be an interval and let $\eta\colon I\to\bbC$ be a fractal curve in the plane. A natural question is how one can quantify the regularity or fractality of $\eta$. 

One approach is to view the curve as a subset of $\bbC$ by considering the set $\eta(I)$. One can study the dimension of this set, e.g.\ the Hausdorff or Minkowski dimension. It follows from \cite{bef-sle-dimension,rs-sle} that in the case of SLE, the two latter dimensions agree a.s.\ and are given by $(1+\kappa/8)\wedge 2$. One can also ask about the exact gauge function that gives a non-trivial Hausdorff measure or Minkowski content for SLE; this is known to be an exact power-law for the Minkowski content  \cite{lr-minkowski-content} while it is unknown for the Hausdorff measure \cite{rez-hausdorff}. The $(1+\kappa/8)\wedge 2$-dimensional Minkowski content of SLE has been proven to define a parametrization of the curve known as the natural parametrization \cite{lr-minkowski-content,lv-lerw-nat,ls-natural-parametrisation}.

For conformally invariant curves in $\bbC$ like SLE it is also natural to study the regularity of a uniformizing conformal map from the complement of the curve to some reference domain. See e.g.\ \cite{bs-aims-sle,gms-sle-multifractal,vl-tip-multifractal,abv-boundary-collisions,sch-mf-boundary,ghm-dim-conformal,kms-sle48x} for results on this for SLE. 

One can also quantify the regularity of $\eta$ by viewing it as a parametrized curve rather than a subset of $\bbC$. The \emph{modulus of continuity} is natural in this regard. We say that $\eta\colon I \to\bbC$ admits $\omega\colon[0,\infty]\to[0,\infty]$ as a modulus of continuity if $|\eta(t)-\eta(s)|\leq \omega(|t-s|)$ for any $s,t\in I$. If this holds for $\omega(t)=c\,t^\alpha$ for some $c>0$ and $\alpha>0$, then we say that the curve is $\alpha$-H\"older continuous.

Note that the modulus of continuity of a curve depends strongly on the parametrization of the curve. For SLE there are two commonly considered parametrizations: the capacity parametrization (see e.g.\ \cite{law-conformal-book}) and the natural parametrization. 
The optimal H\"older  exponent of SLE was computed by Lawler and Viklund \cite{vl-sle-hoelder} for the capacity parametrization, and logarithmic refinements were studied in \cite{yuan-psivarx,kms-sle48x}. For SLE with its natural parametrization the optimal H\"older exponent is equal to one over the dimension of the curve (see \eqref{eq:d} below). 
This was proven by Zhan for $\kappa\leq 4$ \cite{zhan-hoelder} and by Gwynne, Holden, and Miller for space-filling SLE with $\kappa>4$ \cite{ghm-kpz-sle}. For non-space-filling SLE$_\kappa$ for $\kappa\in(4,8)$ it was known that the curve is H\"older continuous for any exponent strictly smaller than one over the dimension, but the matching upper bound on the optimal exponent was unknown; see \cite[Remark 4.3]{zhan-hoelder}.

The regularity of $\eta$ at a fixed time $t_0\in I$ can be quantified via a \emph{law of the iterated logarithm} which describes the magnitude of the fluctuations of $\eta$ as it approaches $t_0$. One can also consider the set of exceptional times where the fluctuations are different from a typical point, e.g. so-called fast or slow points.

Finally, another important notion of regularity for a curve is the \emph{variation}. For an increasing homeomorphism $\psi\colon(0,\infty)\to(0,\infty)$ we say that $\eta\colon I\to\bbC$ has finite $\psi$-variation if
\[
\sup_{t_0<\dots<t_r} \sum_i \psi(|\eta(t_{i+1})-\eta(t_i)|) < \infty,
\]
where we take the supremum over all finite sequences $t_0<t_1<\dots<t_r$, $t_i\in I$ for $i=1,\dots,r$.
Equivalently, a continuous curve $\eta$ has finite $\psi$-variation if and only if there exists a reparametrization $\wt\eta$ of $\eta$ that admits modulus of continuity $\psi^{-1}$, i.e.,
\[ \abs{\wt\eta(t)-\wt\eta(s)} \le \psi^{-1}(\abs{t-s}) . \]
This measure of regularity is invariant under reparametrizations of the curve, and is therefore particularly natural in the setting of SLE where people use multiple different parametrizations or simply consider the curve to be defined only modulo reparametrization of time. The case $\psi(x) = x^p$ (usually called $p$-variation) plays an important role in the theories of Young integration and rough paths. A curve is of finite $p$-variation if and only if it admits a $1/p$-Hölder continuous reparametrisation. When $\eta$ has finite $\psi$-variation, the following quantity is finite and, for convex $\psi$, can be proven to define a semi-norm
\[ [\eta]_{\psi\text{-var},I} = \inf\left\{ M>0 \mmiddle| \sup_{\substack{t_0<\dots<t_r\\ t_i \in I}} \sum_i \psi\left(\frac{\abs{\eta(t_{i+1})-\eta(t_i)}}{M}\right) \le 1 \right\} . \]

The optimal $p$-variation exponent of SLE was computed to be equal to its dimension $(1+\kappa/8)\wedge 2$ in \cite{bef-sle-dimension,ft-sle-regularity,wer-sle-pvar}, and a non-optimal logarithmic refinement of the upper bound was established by the second author \cite{yuan-psivarx}. Recall the general result that the $p$-variation exponent cannot be smaller than the Hausdorff dimension of the curve.

\subsection{Main results}
\label{se:results}

Unless otherwise mentioned, we assume throughout the section that $\eta$ is either
\begin{enumerate}[label=(\roman*)]
    \item \label{it:non-spfill} a two-sided whole-plane SLE$_\kappa$, $\kappa\le 8$, from $\infty$ to $\infty$ passing through $0$ with its natural parametrization, 
    or
    \item \label{it:spfill} a whole-plane space-filling SLE$_\kappa$, $\kappa>4$, from $\infty$ to $\infty$ with its natural parametrization (i.e., parametrized by Lebesgue area measure).
\end{enumerate}
Furthermore, we let $d$ denote the dimension of the curve, namely
\eqb
d = 1+\frac{\kappa}{8} \quad\text{in case \ref{it:non-spfill}\qquad and\qquad }
d=2 \quad\text{in case \ref{it:spfill}.}
\label{eq:d}
\eqe
The SLE variants considered in \ref{it:non-spfill} and \ref{it:spfill} are particularly natural since it can be argued that they describe the local limit in law of an arbitrary variant of SLE with its natural parametrization zoomed in at a typical time. The reason we consider these two SLE variants is that they are self-similar processes of index $1/d$ with stationary increments, in the sense that for every $t_0 \in \bbR$ and $\lambda > 0$, the process $t \mapsto \eta(t_0+\lambda t)-\eta(t_0)$ has the same law as $t \mapsto \lambda^{1/d}\eta(t)$ (see \cite[Corollary 4.7 and Remark 4.9]{zhan-loop} for case \ref{it:non-spfill}, and \cite[Lemma 2.3]{hs-mating-eucl} for case \ref{it:spfill}). We give a more thorough introduction to these curves in Section \ref{se:prelim}.

Throughout the paper, we write $\logp(x) = \log(x) \vee 1$.

\begin{theorem}[variation regularity]\label{thm:main_var}
Let $\psi(x) = x^d(\logp\logp\frac{1}{x})^{-(d-1)}$. There exists a deterministic constant $c_0 \in {(0,\infty)}$ such that almost surely
\eqb 
\lim_{\delta\downarrow 0} \sup_{\abs{t_{i+1}-t_i}<\delta} \sum_i \psi(\abs{\eta(t_{i+1})-\eta(t_i)}) = c_0\abs{I} 
\label{eq:main_var}
\eqe
for any bounded interval $I \subseteq \bbR$, where the supremum is taken over finite sequences $t_0 < ... < t_r$ with $t_i \in I$ and $\abs{t_{i+1}-t_i}<\delta$.

Moreover, for any bounded interval $I \subseteq \bbR$, there exists $c>0$ depending on the length of $I$ such that
\[
\bbP( [\eta]_{\psi\text{-var};I} > u ) \le c^{-1}\exp(-c u^{d/(d-1)}) 
.
\]
\end{theorem}

Recall that the previous works \cite{bef-sle-dimension,ft-sle-regularity} have identified $d$ as the optimal $p$-variation exponent. Our result gives the optimal function $\psi$ up to a non-explicit deterministic factor.
In other words, the best modulus of continuity among all parametrisations of $\eta$ is $\omega(t) = c\,t^{1/d}(\logp\logp t^{-1})^{1-1/d}$, and in particular, there is no reparametrisation that is $1/d$-Hölder continuous. The latter result has been stated as an open problem in e.g.\ \cite[Remark~1.6]{zhan-hoelder}.

\begin{theorem}[modulus of continuity]\label{thm:main_moc}
There exists a deterministic constant $c_1 \in {(0,\infty)}$ such that almost surely
\eqb 
\lim_{\delta\downarrow 0} \sup_{s,t \in I, \abs{t-s}<\delta} \frac{\abs{\eta(t)-\eta(s)}}{\abs{t-s}^{1/d}(\log \abs{t-s}^{-1})^{1-1/d}} = c_1  
\label{eq:main-moc}
\eqe
for any non-trivial bounded interval $I \subseteq \bbR$.

Moreover, for any bounded interval $I \subseteq \bbR$ there exists $c>0$ depending on the length of $I$ such that
\[
\bbP \left( \sup_{s,t \in I} \frac{\abs{\eta(t)-\eta(s)}}{\abs{t-s}^{1/d}(\logp \abs{t-s}^{-1})^{1-1/d}} > u \right) \le c^{-1}\exp(-c u^{d/(d-1)}) .
\]
\end{theorem}
The optimal H\"{o}lder exponent $1/d$ has been identified previously in \cite{zhan-hoelder,ghm-kpz-sle} (except in case \ref{it:non-spfill} for $\kappa\in(4,8)$, where only the upper bound was established). We prove that $\omega(t) = t^{1/d}(\log t^{-1})^{1-1/d}$ is the optimal (up to constant) modulus of continuity.

\begin{theorem}[law of the iterated logarithm and maximal growth rate]\label{thm:main_lil}
There exist deterministic constants $c_2,c_3 \in {(0,\infty)}$ such that for any $t_0 \in \bbR$, almost surely
\begin{align*}
\limsup_{t \downarrow 0} \frac{\abs{\eta(t_0+t)-\eta(t_0)}}{t^{1/d}(\log\log t^{-1})^{1-1/d}} = c_2,\\
\limsup_{t \to \infty} \frac{\abs{\eta(t)}}{t^{1/d}(\log\log t)^{1-1/d}} = c_3.
\end{align*}
\end{theorem}

\begin{remark}
We expect that with some extra care one can show that the moment bounds in \cref{thm:main_moc,thm:main_var} are uniform in $\kappa$ as long as we stay away from the degenerate values, i.e.\ away from $4$ in case \ref{it:spfill} and possibly away from $0$ in case \ref{it:non-spfill}. A uniformity statement of this type is proven in \cite{am-sle8x} and used to construct \sle{8} as a continuous curve. 
\end{remark}

\begin{remark}
In view of \cref{thm:main_lil,thm:main_moc}, it is natural to study exceptional times where the law of the iterated logarithm fails. For $a>0$ and a time $t_0$ call the time $a$-fast if
\[ \limsup_{t \downarrow 0} \frac{\abs{\eta(t_0+t)-\eta(t_0)}}{t^{1/d}(\log t^{-1})^{1-1/d}} \ge a . \]
With the methods of this paper, one can show that the Hausdorff dimension of the set of $a$-fast times for SLE is bounded between $1/d-c_1 a^{d/(d-1)}$ and $1-c_2 a^{d/(d-1)}$ with (non-explicit) deterministic constants $c_1,c_2>0$. The reason we get $1/d$ instead of 1 in our lower bound is that in our argument we only consider the radial direction of the curve instead of $d$ dimensions. We conjecture that there is a deterministic constant $c_0 >0$ such that the dimension is exactly $1-c_0 a^{d/(d-1)}$. For comparison, the Hausdorff dimension of the set of $a$-fast times for Brownian motion is proven in \cite{ot-bm-fast} to be $1-a^2/2$.
\end{remark}

In the case of space-filling SLE, we show a stronger formulation of the upper bound in \cref{thm:main_moc}.
\begin{theorem}\label{thm:ball_filling}
Consider space-filling \slek{} as in case \ref{it:spfill}. There exist $\delta > 0$ and $u_0 > 0$ such that the following is true. For $r,u>0$ and $I$ a bounded interval let $E_{r,u,I}$ denote the event that for any $s,t \in I$ with $\abs{\eta(s)-\eta(t)} \le ur$, the set $\eta[s,t]$ contains $\delta u^2 \log(u\abs{\eta(s)-\eta(t)}^{-1})$ disjoint balls of radius $u^{-2} \abs{\eta(s)-\eta(t)} /\log(u\abs{\eta(s)-\eta(t)}^{-1})$. Then for any bounded interval $I \subseteq \bbR$ there exist $c_1,c_2 > 0$ such that
\[ \bbP(E_{r,u,I}^c) \le c_1 r^{c_2 u^2} \]
for any $u \ge u_0$ and $r \in {(0,1)}$.
\end{theorem}

A key input to the proofs is the following precise estimate for the lower tail of the Minkowski content of SLE segments:
\eqb 
\bbP(\Cont(\eta[0,\tau_r]) < t) \approx \exp\left(-c\, r^{d/(d-1)} t^{-1/(d-1)} \right),
\label{eq:mink-tail}
\eqe
where 
$\Cont$ denotes Minkowski content of dimension $d$ (with $d$ as \eqref{eq:d}), $\tau_r=\inf\{t\geq 0\,:\, |\eta(t)|=r \}$ denotes the hitting time of radius $r$, we write $\eta[0,\tau_r]$ instead of $\eta([0,\tau_r])$ to simplify notation, and we use $\approx$ to indicate that the left side of \eqref{eq:mink-tail} is bounded above and below by the right side of \eqref{eq:mink-tail} for different choices of $c$. Furthermore, building on the domain Markov property of SLE we prove ``conditional'' variants of \eqref{eq:mink-tail} where we condition on part of the past curve. The conditional variant of the upper bound holds for all possible realizations of the past curve segment while the conditional variant of the lower bound requires that the tip of the past curve is sufficiently nice. See \cref{pr:diam_tail_upper,pr:cont_tail_lower,pr:area_tail_lower} for precise statements, and see \cref{eq:cont_utail_cond,pr:cont_ltail_cond,pr:area_ltail_cond} for conditional variants. Note that since we parametrise $\eta$ by its Minkowski content, \eqref{eq:mink-tail} can be equivalently formulated as
\[ \bbP(\diam(\eta[0,t]) > r) \approx \exp\left(-c\, r^{d/(d-1)} t^{-1/(d-1)} \right) . \]

The upper bounds in \cref{thm:main_lil,thm:main_moc,thm:main_var} hold for arbitrary stochastic processes whose increments satisfy a suitable moment condition. Several related results are available in the existing literature. More precisely, fairly general results have been worked out for the modulus of continuity (see e.g.\ \cite{Bed07}), while less has been done for the variation regularity, but one instance is \cite[Appendix A]{fv-rp-book}. In \cref{sec:variation-upper-general} we review these results and generalise \cite[Appendix A]{fv-rp-book}. We prove that a general stochastic process satisfying the condition \eqref{eq:Xinc} attains a certain variation regularity. We then prove that the SLE variants \ref{it:non-spfill} and \ref{it:spfill} do satisfy the required condition. In the latter step we use only the scaling property of the SLE along with a Markov-type property satisfied by the increments, namely a conditional variant of the upper bound in \eqref{eq:mink-tail}.

To prove the lower bounds, we need to argue that the increments of the process in disjoint time intervals are sufficiently decorrelated. The decorrelation results we prove differ from those obtained in earlier works in a fundamental way. The earlier works prove independence of SLE segments within \emph{disjoint regions}. In our approach, we need to consider segments in \emph{neighbouring intervals}, and this adds a fundamental challenge as conformal maps are in general irregular near the boundary. We obtain bounds on the conditional probabilities given \emph{every instance} of an initial segment. We introduce the concept of a $(p_{\mathrm{N}},r_{\mathrm{N}},c_{\mathrm{N}})$-nice domain to deal with this issue.

Given sufficient decorrelation, our proof follows a similar strategy as for Markov processes with uniform bounds on the transition probabilities. We give the proof for Markov processes in \cref{se:markov} since this is helpful for the reader to have in mind when looking at the SLE case later. For SLE, we rely on the conditional variant of the lower bound in \eqref{eq:mink-tail}, which is based on the domain Markov property. Here extra care is needed due to the fact that this estimate only holds when the past curve is nice.

The theorems so far concern the whole-plane SLE variants \ref{it:non-spfill} and \ref{it:spfill}. They transfer to other SLE variants by conformal invariance and absolute continuity as long as we stay away from the domain boundary and potential force points, i.e., the statements of \cref{thm:main_lil,thm:main_moc,thm:main_var} hold true on closed curve segments that do not touch force points or domain boundaries. We expect that all the statements of \cref{thm:main_lil,thm:main_moc,thm:main_var} hold for entire SLE curves in bounded domains whose boundaries are not too fractal. For chordal SLE, we prove in this paper that it does have finite $\psi$-variation on boundary-intersecting segments away from the initial and terminal points when $\psi$ is as in \cref{thm:main_var}.

\begin{theorem}\label{thm:var_chordal}
Let $\psi(x) = x^d(\logp\logp\frac{1}{x})^{-(d-1)}$. Let $\eta$ be a chordal \slek{} in $(\bbH,0,\infty)$ for $\kappa > 0$, where we consider either the space-filling or non-space-filling variant for $\kappa\in(4,8)$. Then almost surely the $\psi$-variation of $\eta\big|_{[s,t]}$ is finite for every $0<s<t<\infty$.
\end{theorem}

Furthermore, we prove that (when restricted to certain subsets) the $\psi$-variation constant has a fast decaying tail, see \cref{pr:bdy-fin-var} for detailed statements.

The main challenge in proving the  previous theorem for chordal SLE is that the arguments for two-sided whole-plane SLE (in particular translation invariance) are not valid for chordal SLE. To obtain \cref{thm:var_chordal} we combine whole-plane SLE arguments and arguments from the chordal Loewner equation, which a priori are very different, and we have not seen these arguments combined in this way in earlier literature. The regularity of chordal SLE has been studied previously in e.g.\ \cite{vl-sle-hoelder,ft-sle-regularity,yuan-psivarx,wer-sle-pvar,kms-sle48x}, but the techniques in these papers do not seem to be sufficient to obtain our sharp logarithmic results.

Finally, we prove a new characterisation of the natural parametrisation of SLE in terms of $\psi$-variation.

\begin{theorem}\label{thm:var_cont}
Let $\psi(x) = x^d(\logp\logp\frac{1}{x})^{-(d-1)}$. Let $\eta$ be either one of the whole-plane SLE variants \ref{it:non-spfill}, \ref{it:spfill} or a chordal (either space-filling or non-space-filling) \slek{} in a domain with analytic boundary. Then almost surely
\[ c_0\Cont_d(\eta[s,t]) = \lim_{\delta\downarrow 0} \sup_{\substack{s \le t_0 < ... < t_r \le t\\ \abs{t_{i+1}-t_i}<\delta}} \sum_i \psi(\abs{\eta(t_{i+1})-\eta(t_i)}) \]
for all $0<s<t<\infty$, where $c_0>0$ is the constant in \cref{thm:main_var}.
\end{theorem}
The last result for the SLE variants \ref{it:non-spfill} and \ref{it:spfill} is a restatement of (the first part of) \cref{thm:main_var}. We will prove the theorem for chordal SLE in \cref{se:chordal}.

\subsection{Comment on discrete models}
\label{se:discrete}

There are a number of discrete models (e.g.\ the loop-erased random walk, Ising model interfaces, Gaussian free field contours, percolation interfaces, and the uniform spanning tree) that have been proven to converge to SLE in the scaling limit. Our results shed light on the regularity of these curves as well.

Lower bounds on the regularity of the discrete curves follow immediately by the observation that if $\eta^n \to \eta$ pointwise on $[0,1]$, then
\[
[\eta]_{\psi\text{-var},[0,1]} \le \liminf_n [\eta^n]_{\psi\text{-var},[0,1]} 
\]
and
\[
\sup_{s,t\in[0,1]} \frac{\abs{\eta(s)-\eta(t)}}{\omega(\abs{s-t})} \le \liminf_n \sup_{s,t\in[0,1]} \frac{\abs{\eta^n(s)-\eta^n(t)}}{\omega(\abs{s-t})}.
\]
As a consequence, if the random curves $\eta^n$ converge in law to $\eta$ with respect to the uniform topology on $[0,1]$, then
\eqb
\lim_{n\to\infty} \bbP\left( \sup_{\abs{t_{i+1}-t_i}<\delta} \sum_i \psi(\abs{\eta^n(t_{i+1})-\eta^n(t_i)}) \le c_0-\varepsilon \right) = 0
\label{eq:var-discrete}
\eqe
and
\eqb
\lim_{n\to\infty} \bbP\left( \sup_{\abs{s-t}<\delta} \frac{\abs{\eta^n(s)-\eta^n(t)}}{\omega(\abs{s-t})} \le c_0'-\varepsilon \right) = 0
\label{eq:moc-discrete}
\eqe
for any $\delta,\varepsilon > 0$ where $\psi(x) = x^d(\logp\logp\frac{1}{x})^{-(d-1)}$, $\omega(t) = t^{1/d}(\logp\frac{1}{t})^{1-1/d}$, and $c_0,c_0' > 0$ are as in Theorems \ref{thm:main_var} and \ref{thm:main_moc}, respectively. In fact, since variation regularity is invariant under reparametrization of the curve, in order to get \eqref{eq:var-discrete} it is sufficient that $\eta^n$ converge in law to $\eta$ as a curve modulo reparametrization of time. Furthermore, a convergence result in this topology also gives a non-trivial (but most likely non-optimal) lower bound of $\psi^{-1}$ on the modulus of continuity of the discrete curve, i.e., \eqref{eq:moc-discrete} holds for some constant $c'_0$ with $\psi^{-1}$ instead of $\omega$; see Section \ref{sec:quantify}. 

Let us now comment on which discrete models that can be treated via the observations in the previous paragraph. First recall that our lower bounds on regularity hold for chordal variants of the SLE$_\kappa$, not only the bi-infinite variants \ref{it:non-spfill} and \ref{it:spfill}; therefore regular chordal SLE$_\kappa$ (space-filling or non-space-filling in case $\kappa\in(4,8)$) can play the role of $\eta$ in the previous paragraph. Convergence to chordal SLE$_\kappa$ has been established  for all the discrete models mentioned in the first paragraph of this subsection if we view them as curves modulo reparametrization of time, and therefore the discrete curve will satisfy \eqref{eq:var-discrete}. Convergence in law for the uniform topology where each edge is traced in $n^{-d}$ units of time has been established for the loop-erased random walk \cite{lv-lerw-nat}, percolation \cite{li-et-al-sharp-arm-percolation,hls-sle6}, and the uniform spanning tree \cite{hs-mating-eucl}, and therefore \eqref{eq:moc-discrete} holds for these discrete models.

A natural question is whether the matching upper bounds hold. While convergence of a discrete curve to a particular continuum curve in the uniform topology does not imply upper bounds on the regularity of the discrete curve, it turns out that our argument for proving the regularity upper bound for SLE can be adapted to prove regularity upper bounds of discrete models. This is the case since our upper bound proofs use rather few 
inputs that can be proven to be satisfied also by the discrete models. We need to prove an analogue of \eqref{eq:mink-tail} that holds uniformly in $n$. By \cref{pr:ss_tail} below, it suffices to show for every $t$ that if $\tau_{t,r} = \inf\{ s>t \mid \abs{\eta^n(s)-\eta^n(t)} \ge r \}$, then
\[
\bbP\left( \operatorname{length}(\eta^n[\tau_{t,r},\tau_{t,r+\lambda}]) \le \ell\lambda^d \mmiddle| \eta^n[t,\tau_{t,r}] \right) < p
\]
for some $p<1$ and $\ell>0$ that do not depend on $n$. This can be obtained for certain discrete models, in particular the two-sided infinite LERW constructed in \cite{law-twosided-lerw}. 
We plan to carry this out in future work.

We note that a uniform upper bound for the regularity of $\eta^n$ is directly related to the strongest topology in which the convergence $\eta^n \to \eta$ holds. Namely, suppose that $\eta^n \to \eta$ in law with respect to the uniform topology on $[0,1]$ and that additionally
\[
\lim_{M\to\infty} \sup_n \bbP( [\eta^n]_{\psi\text{-var}} > M ) = 0.
\]
Then the convergence of $\eta^n \to \eta$ in law holds if we view the paths as random variables in the topological space induced by the variation norm $\norm{\cdot}_{\wt\psi\text{-var}} = \norm{\cdot}_\infty + [\cdot]_{\wt\psi\text{-var}}$ (cf.\ \cref{sec:quantify}) for every $\wt\psi$ with $\lim_{x\searrow 0}\wt\psi(x)/\psi(x) = 0$. 
Similarly, if
\[
\lim_{M\to\infty} \sup_n \bbP( [\eta^n]_{\omega\text{-Höl}} > M ) = 0 ,
\]
then the convergence of $\eta^n \to \eta$ in law holds with respect to $\norm{\cdot}_{\wt\omega\text{-Höl}}$ for every $\wt\omega$ with $\lim_{t\searrow 0}\omega(t)/\wt\omega(t) = 0$.

\subsection{Outline}
\label{se:outline}

We give in \cref{se:prelim} some basic definitions and results on conformal maps and SLE, including the precise definition of the SLE variants that we work with.
In \cref{se:upper} we prove our main theorems except for the lower bounds which will be proved in \cref{sec:sle-lower-bounds}. To illustrate the basic idea of the proof, we show in \cref{se:markov} the analogous results for Markov processes. In \cref{se:chordal} we use the whole-plane results to prove the results on $\psi$-variation of chordal SLE.

\medskip
\textbf{Notation:} We will frequently use the notation $a \lesssim b$ indicating $a \le cb$ for some finite constant $c$ that may vary from line to line, and the constant $c$ may depend on the given parameters unless specified otherwise.  Furthermore, we write $a \asymp b$ meaning $a \lesssim b$ and $b \lesssim a$.

\medskip

{\bf Acknowledgements}.
We thank an anonymous referee for comments which helped us to improve the manuscript. We also thank Peter Friz, Ewain Gwynne, and Terry Lyons for helpful discussions and comments.  N.H.\ was supported by grant 175505 of the Swiss National
Science Foundation (PI: Wendelin Werner) and was part of SwissMAP during the initial stage of this work. Later she was supported by grant DMS-2246820 of the National Science Foundation. Y.Y.\@ acknowledges partial support from European Research Council through Starting Grant 804166 (SPRS; PI: Jason Miller), and Consolidator Grant 683164 (PI: Peter Friz) during the initial stage of this work at TU Berlin. 

\section{Preliminaries}
\label{se:prelim}

\subsection{Conformal maps}
\label{se:prelim_conformal}

We will always denote by $\bbH$ the upper complex half-plane $\{ z \in \bbC \mid \Im z > 0 \}$, and by $\bbD$ the unit disk $\{ z \in \bbC \mid \abs{z} < 1\}$. For $z_0 \in \bbC$ and $r > 0$, we denote by $B(z_0,r)$ the open ball of radius $r$ about $z$, i.e. $\{ z \in \bbC \mid \abs{z-z_0} < r\}$.

For a bounded, relatively closed set $A \subseteq \bbH$, we define its half-plane capacity to be $\lim_{y \to \infty} y\,\bbE[\Im B^{iy}_{\tau_{A \cup \bbR}}]$ where $B^{iy}$ denotes Brownian motion started at $iy$ and $\tau_{A \cup \bbR}$ the hitting time of $A \cup \bbR$.

For a simply connected domain $D \subseteq \hatC$ and a prime end $x \in \partial D$, fix a conformal map $f\colon D \to \bbH$ with $f(x) = \infty$. For a relatively closed set $A \subseteq D$ with positive distance to $x$, we define the {\it{capacity of $A$ in $D$ relative to $x$}} to be the half-plane capacity of $f(A)$.\footnote{In case $\partial D$ has an analytic segment in the neighbourhood of $x$, there is a more intrinsic definition given in \cite{dv-relcap}. Their definition differs to ours by a fixed factor depending on the normalisation of $f$. In particular, we can pick $f$ such that both definitions agree. For our purposes, the choice of normalisation will not matter.}

Standard results for conformal maps include Koebe's distortion and $1/4$-theorem. See e.g. \cite[Theorems 14.7.8 and 14.7.9]{con-complex2-book} for proofs.

\begin{lemma}[Koebe's distortion theorem]
Let $R>0$ and $f\colon B(z,R) \to \bbC$ be a univalent function. Then
\[ \frac{1-r}{(1+r)^3} \le \frac{\abs{f'(w)}}{\abs{f'(z)}} \le \frac{1+r}{(1-r)^3} \]
for all $w \in B(z,R)$ where $r = \frac{\abs{w-z}}{R} < 1$.
\end{lemma}

The most common application in our work is that for any  $w\in B(z,r)$, $r < R$, we have the bounds
\[ c\abs{f'(z)} \le \abs{f'(w)} \le c^{-1}\abs{f'(z)} \]
with $c>0$ depending on $r$.

Another useful application is the following estimate on the upper half-plane.
\begin{lemma}
There exists $c>1$ such that for every univalent function $f\colon \bbH \to \bbC$ the following estimate holds
\begin{equation}\label{eq:koebe_hp}
    \abs{f'(x+iy)} \le c(1+\abs{x}/y)^4 \abs{f'(iy)} .
\end{equation}
\end{lemma}

\begin{proof}
This is an immediate consequence of Koebe's distortion theorem. Indeed, pick a conformal map $\varphi\colon \bbH \to \bbD$ such that $\varphi(iy) = 0$. A concrete choice is $\varphi(z) = \frac{z-iy}{z+iy}$. Then one can verify (either by a direct calculation or by considering the hyperbolic distance between $iy$ and $x+iy$) that $\dist(\varphi(x+iy),\partial\bbD) \asymp (x/y)^{-2}$ as $x/y \to \infty$. The statement then follows from Koebe's distortion theorem and the chain rule.
\end{proof}

\begin{lemma}[Koebe's $1/4$ theorem]
Let $R>0$ and $f\colon B(z,R) \to D \subseteq \bbC$ be a conformal map. Then
\[ \dist(f(z),\partial D) \ge \frac{R}{4}\abs{f'(z)} . \]
\end{lemma}

For a simply connected domain $D \subseteq \bbC$ and $z \in D$, the \textit{conformal radius of $z$ in $D$} is defined as $\abs{f'(0)}$ where $f\colon \bbD \to D$ is a (unique up to rotations of $\bbD$) conformal map with $f(0) = z$. We have the standard estimates
\[ \dist(z,\partial D) \le \crad(z,D) \le 4\dist(z,\partial D) \]
which follow from the Schwarz lemma and Koebe's $1/4$ theorem.

Throughout the paper, we will often consider domains of the following type. Let $(D,a)$ be such that
\eqb\begin{split}
&D\subset\wh{\C} \text{ is a simply connected domain with } \infty\in D,\,
0\not\in D,\\
\text{and}\quad
&a\in\partial D \text{ with }
\abs{a} = \sup\{ \abs{z}\mid z \in \wh\bbC\setminus D\} > 0.
\label{eq:Da}
\end{split}\eqe
Typical examples of such domains are $D = \hatC \setminus \fill(\eta[0,\tau_r])$, $a = \eta(\tau_r)$ where $\eta$ is some continuous path starting from the origin,  
$\tau_r=\inf\{t\geq 0\,:\, |\eta(t)|=r \}$ denotes the hitting time of $\partial B(0,r)$,
and fill($\eta[0,\tau_r]$) denotes the union of $\eta[0,\tau_r]$ and the set of points disconnected from $\infty$ by this set.

To $(D,a)$ as in \eqref{eq:Da}, we associate a conformal map $f\colon D \to \bbD$ as follows. Let $z_D \in \partial B(0,\abs{a}+2)$ be the point closest to $a$, and $f\colon D \to \bbD$ the conformal map with $f(z_D)=0$ and $f(a)=1$. The following property is shown within the proof of \cite[Lemma 3.1]{ghm-kpz-sle}.

\begin{lemma}\label{le:ghm}
There exists $\varepsilon_0 > 0$ such that the following is true. Let $(D,a)$ be as in \eqref{eq:Da} with $\abs{a} \ge 1$, and $f\colon D \to \bbD$ the associated conformal map. Let $V$ be the union of $B(z_D,3) \cap D$ with all points that it separates from $\infty$ in $D$. There exists a path $\alpha$ from $0$ to $1$ in $\bbD$ whose $\varepsilon_0$-neighbourhood is contained in $f(V)$. Moreover, $\alpha$ can be picked as a simple nearest-neighbour path in $\varepsilon_0 \bbZ^2$.
\end{lemma}

\subsection{(Ordinary) SLE}
\label{se:prelim_sle}

In this and the next section we discuss the SLE variants that we use in the paper. All SLE variants are probability measures on curves (modulo reparametrisation) either in a simply connected domain $D \subseteq \hatC$ or in the full plane $\hatC$.

Fix $\kappa > 0$. Let $D \subseteq \hatC$ be a simply connected domain, and $a,b \in \partial D$ two distinct prime ends. Moreover, we may possibly have additional force points $u^1,...,u^n \in \overline{D}$ with weights $\rho_1,...,\rho_n \in \bbR$. The chordal \slek{}$(\rho_1,...,\rho_n)$ in $D$ from $a$ to $b$ with these force points is a probability measure on curves in $\overline{D}$ stating at $a$ with the following domain Markov property: For any stopping time $\tau$, conditionally on $\eta\big|_{[0,\tau]}$, the law of $\eta\big|_{[\tau,\infty)}$ in an \slek{}$(\rho_1,...,\rho_n)$ in the connected component of $D \setminus \eta[0,\tau]$ containing $b$ from $\eta(\tau)$ to $b$ with the same force points. (There is a subtlety when $\eta$ has swallowed some force points, but in this paper we will not encounter that scenario.)

Moreover, the SLE measures are conformally invariant in the sense that if $f\colon D \to f(D)$ is a conformal map, then the push-forward of \slek{}$(\rho_1,...,\rho_n)$ in $D$ is \slek{}$(\rho_1,...,\rho_n)$ in $f(D)$ from $f(a)$ to $f(b)$ with force points $f(u^1),...,f(u^n)$.

Similarly, radial SLE is characterised by the same properties except that $b \in D$ instead of $\partial D$. For both chordal and radial SLE, it is sometimes convenient to consider $b$ as an additional force point with weight $\rho_{n+1} = \kappa-6-\rho_1-...-\rho_n$. By doing so, we have the following simple transformation rule (see \cite[Theorem 3]{sw-sle-coordinate-change}): For any conformal map $f\colon D \to f(D)$, the pushforward of \slek{}$(\rho_1,...,\rho_{n+1})$ in $D$ starting from $a$ (stopped before swallowing any force point) is \slek{}$(\rho_1,...,\rho_{n+1})$ in $f(D)$ starting from $f(a)$. (Note that in this rule, the target point $b$ is not necessarily $u^{n+1}$ but can be any $u^j$. The law does not depend on the choice of a target point.)

In case $D = \bbH$ or $D = \bbD$, the laws of SLE can be spelled out explicitly (cf.\@ \cite{sw-sle-coordinate-change}). Namely, we can describe the law of the conformal maps $g_t\colon \bbH \setminus \fill(\eta[0,t]) \to \bbH$ with $g_t(z) = z + O(1/z)$, $z \to \infty$. If $\eta$ is parametrised by half-plane capacity, i.e. $\hcap(\fill(\eta[0,t])) = 2t$, then the families $(g_t(z))_{t \ge 0}$ satisfy
\[ \partial_t g_t(z) = \frac{2}{g_t(z)-W_t}, \quad g_0(z) = z, \]
with processes $(W_t,U^1_t,...,U^n_t)$ satisfying the following system of SDEs
\[ \begin{alignedat}{2}
dW_t &= \sqrt{\kappa}\,dB_t + \sum_j \Re\frac{\rho_j}{W_t-U^j_t}\,dt , \quad && W_0 = a,\\
dU^j_t &= \frac{2}{U^j_t-W_t}\,dt , \quad && U^j_0 = u^j .
\end{alignedat} \]
By Girsanov's theorem, for fixed $\kappa \ge 0$, all such \slek{} variants (with different values of $\rho_j$) are absolutely continuous with respect to each other (before any force point is swallowed). The Radon-Nikodym derivatives can be spelled out explicitly, see \cite[Theorem 6]{sw-sle-coordinate-change}.
One particular consequence is the following.
\begin{lemma}\label{le:sle_abs_cont}
Let $D \subseteq \bbC$ be a simply connected domain, $U \subseteq D$ a bounded subdomain, and $a \in \partial D \cap \overline{U}$. Let $\varepsilon > 0$ and $\rho_1,...,\rho_n \in \bbR$. Then there exists $c>0$ such that the following is true:

Let $u^1,...,u^k \in U$ and let $u^{k+1},...,u^n$, and $b$ be outside the $\varepsilon$-neighbourhood of $U$. Consider the law $\nu$ of \slek{} in $D$ from $a$ to $b$ with force points $(u^1,...,u^k)$ and weights $(\rho_1,...,\rho_k)$, and the law $\wt\nu$ of \slek{} in $D$ from $a$ to $b$ with force points $(u^1,...,u^n)$ and weights $(\rho_1,...,\rho_n)$. Then, as laws on curves stopped upon exiting $U$, these two \slek{} measures are absolutely continuous, and their Radon-Nikodym derivative is bounded within $[c,c^{-1}]$.
\end{lemma}

Whole-plane \slek{} and whole-plane \slek{}$(\rho)$ are probability measures on curves in $\hatC$ running from $a$ to $b$ where $a,b$ are two distinct points on $\hatC$. They are characterised by an analogous domain Markov property. For any non-trivial stopping time $\tau$, conditionally on $\eta\big|_{(-\infty,\tau]}$, the law of $\eta\big|_{[\tau,\infty)}$ is a radial \slek{} (resp. \slek{}$(\rho)$) in the connected component of $\hatC \setminus \eta((-\infty,\tau])$ containing $b$ from $\eta(\tau)$ to $b$ (with a force point at $a$).

Two-sided whole-plane \slek{}, $\kappa \le 8$, is a probability measure on closed curves in $\hatC$ from $a$ to $a$ passing through some $b \in \hatC$. It is defined as follows:
\begin{itemize}
    \item The segment $\eta\big|_{(-\infty,0]}$ is a whole-plane \slek{}$(2)$ from $a$ to $b$ with force point at $a$.
    \item Conditionally on $\eta\big|_{(-\infty,0]}$, the segment $\eta\big|_{[0,\infty)}$ is a chordal \slek{} in $\hatC \setminus \eta(-\infty,0]$ from $b$ to $a$.
\end{itemize}

We will use the following facts about two-sided whole-plane \slek{} from $\infty$ to $\infty$ passing through $0$ (cf.\@ \cite{zhan-loop}). Both whole-plane \slek{}$(2)$ and two-sided whole-plane \slek{} are reversible. In particular, the restriction $\eta\big|_{[0,\infty)}$ is a whole-plane \slek{}$(2)$ from $0$ to $\infty$ with force point at $0$. Moreover, if $\eta$ is parametrised by Minkowski content, then $\eta$ is a self-similar process of index $1/d$ with stationary increments, in the sense that for every $t_0 \in \bbR$ and $\lambda > 0$, the process $t \mapsto \eta(t_0+\lambda t)-\eta(t_0)$ has the same law as $t \mapsto \lambda^{1/d}\eta(t)$.

In the remainder of this paper, we denote by $\nu_{D;a\to b}$ (chordal or radial) \slek{} in $D$ from $a$ to $b$, and by $\nu_{D;a\to b;u}$ \slek{}$(2)$ with a force point at $u$. We denote by $\nu_{\hatC;a\to b;a}$ whole-plane \slek{}$(2)$.

The Minkowski content measures the size of fractal sets. It has been shown in \cite{lr-minkowski-content} that \slek{} curves\footnote{Strictly speaking, their result applies to \slek{} in $\bbH$ without force points, but it transfers to other \slek{} variants by conformal invariance and absolute continuity, at least on segments away from force points and fractal boundaries. The result is also true for two-sided whole-plane \slek{}, see \cite[Lemma 2.12]{zhan-loop}.} possess non-trivial $d$-dimensional Minkowski content where $d = (1+\frac{\kappa}{8})\wedge 2$. Moreover, the Minkowski content of $\eta[0,t]$ is continuous and strictly increasing in $t$, hence $\eta$ can be parametrised by its Minkowski content. 
This is called the \emph{natural parametrization} of the curve.
Finally, it is shown there that the Minkowski content is additive over SLE curve segments, i.e. $\Cont(\eta[s,u])=\Cont(\eta[s,t])+\Cont(\eta[t,u])$ for $s\le t\le u$.

By \cite[Lemma 2.6]{zhan-loop}, the Minkowski content satisfies the following transformation rule. If $f\colon U \to f(U) \subseteq \bbC$ is a conformal map and $\mu$ is the Minkowski content measure of some $S \subseteq U$, i.e. $\mu(K) = \Cont(K \cap S)$ for every compact $K \subseteq U$, then
\begin{equation}\label{eq:cont_transf}
\Cont(f(K \cap S)) = \int_K \abs{f'(z)}^d \,\mu(dz)
\end{equation}
for every compact $K \subseteq U$.

\subsection{Space-filling SLE}
\label{se:prelim_spf}

In this section we introduce the whole-plane space-filling SLE$_\kappa$ which is defined via the theory of imaginary geometry. Whole-plane space-filling SLE$_{\kappa}$ from $\infty$ to $\infty$ for $\kappa > 4$ was originally defined in~\cite[Section~1.4.1]{dms-mating-trees}, building on the chordal definition in~\cite[Sections 1.2.3 and 4.3]{ms-ig4}. See also \cite[Section 3.6.3]{ghs-mot-survey} for a survey. For $\kappa \geq 8$, whole-plane space-filling SLE$_{\kappa}$ from $\infty$ to $\infty$ is a curve from $\infty$ to $\infty$ obtained via the local limit of a regular chordal SLE$_{\kappa}$ at a typical point. For $\kappa \in (4,8)$, space-filling SLE$_{\kappa}$ from $\infty$ to $\infty$ traces points in the same order as a curve that locally looks like (non-space-filling) SLE$_{\kappa}$, but fills in the ``bubbles'' that it disconnects from its target point with a continuous space-filling loop.

To construct whole-plane space-filling SLE$_{\kappa}$ from $\infty$ to $\infty$, fix a deterministic countable dense subset $\mcl C\subset \BB C$. Let $h$ be a whole-plane GFF, viewed modulo a global additive multiple of $2\pi \chi$ 
where $\chi \defeq \frac{2}{\sqrt\kappa} - \frac{\sqrt\kappa}{2}$; see \cite[Section 2.2]{ms-ig4} for the definition of this variant of the GFF. It is shown in~\cite{ms-ig4} that for each $z\in\mcl C$, one can make sense of the flow lines $\eta_z^{\mathrm{L}}$ and $\eta_z^{\mathrm{R}}$ of $e^{ih/\chi}$ of angles $\pi/2$ and $-\pi/2$, respectively, started from $z$. These flow lines are SLE$_{16/\kappa}(2-16/\kappa)$ curves~\cite[Theorem~1.1]{ms-ig4}. The flow lines $\eta_z^{\mathrm{L}}$ and $\eta_w^{\mathrm{L}}$ (resp.\ $\eta_z^{\mathrm{R}}$ and $\eta_w^{\mathrm{R}}$) started at distinct $z,w\in\mcl C$ eventually merge, such that the collection of flow lines $\eta_z^{\mathrm{L}}$ (resp.\ $\eta_z^{\mathrm{R}}$) for $z\in\mcl C$ form the branches of a tree rooted at $\infty$. 

We define a total ordering on $\mcl C$ by saying that $w\in \mcl C$ comes after $z\in \mcl C$ if and only if $\eta^{\mathrm{L}}_w$ merges into $\eta^{\mathrm{L}}_z$ on its right side (equivalently, $\eta^{\mathrm{R}}_w$ merges into $\eta^{\mathrm{R}}_z$ on its left side). It can be argued that there is a unique space-filling curve $\eta \colon \BB R\rta\BB C$ that visits the points in $\mcl C$ in order, is continuous when parameterized by Lebesgue measure (i.e. $\eta[0,t]$ and $\eta[-t,0]$ both have Lebesgue measure $t$ for any $t>0$), satisfies $\eta(0)=0$, and is such that $(\eta)^{-1}(\mcl C)$ is a dense set of times (see \cite[Section 4.3]{ms-ig4} and \cite{dms-mating-trees}). The curve $\eta$ does not depend on the choice of $\mcl C$ and is defined to be whole-plane space-filling SLE$_{\kappa}$ from $\infty$ to $\infty$. For each fixed $z\in\BB C$, it is a.s.\ the case that the left and right outer boundaries of $\eta$ stopped at the first time it hits $z$ are given by the flow lines $\eta_z^{\mathrm{L}}$ and $\eta_z^{\mathrm{R}}$. Since the flow lines have zero Lebesgue measure and $(\eta)^{-1}(\mcl C)$ is dense, it follows that
almost surely for all $s<t$ the Lebesgue measure of $\eta[s,t]$ is exactly $\abs{t-s}$. 

We remark that for $\kappa=8$, the whole-plane space-filling SLE$_8$ as defined here is equal in law to the two-sided whole-plane SLE$_8$ defined in the previous subsection.

In our proofs in the next subsection we will consider $\eta|_{[0,\infty)}$ and we will now describe this curve slightly more explicitly. The two flow lines $\eta^{\mathrm{L}}_0$ and $\eta^{\mathrm{R}}_0$ divide $\bbC$ into two (for $\kappa\geq 8$) or countably infinite (for $\kappa\in(4,8)$) connected components, such that the boundary of each connected component can be written as the union of a segment of $\eta^{\mathrm{L}}_0$ and a segment of $\eta^{\mathrm{R}}_0$. The curve $\eta|_{[0,\infty)}$ will visit precisely the connected components that lie to the right of $\eta^{\mathrm{L}}_0$ (i.e. $\eta^{\mathrm{L}}_0$ traces its boundary in clockwise direction), and $\eta|_{[0,\infty)}$ restricted to each such component has the law of a chordal space-filling SLE$_\kappa$ connecting the two points of intersection of $\eta^{\mathrm{L}}_0$ and $\eta^{\mathrm{R}}_0$ on its boundary. For $\kappa\geq 8$ the chordal space-filling SLE$_\kappa$ is just a regular chordal SLE$_\kappa$, while for $\kappa\in(4,8)$ the curve can be constructed by starting with a regular chordal (non-space-filling) SLE$_\kappa$ and filling in the components disconnected from the target point by a space-filling SLE$_\kappa$-type loop. The SLE$_\kappa$-type loop can be obtained via an iterative construction where one first samples a regular chordal (non-space-filling) SLE$_\kappa(\kappa-6)$ and then samples curves with this law iteratively in each complementary component of the curve.

The boundary data of $h$ along an angle $\theta\in\{-\pi/2,\pi/2 \}$ flow line is given by $\chi$ times the winding of the curve plus $-\pi/\sqrt{\kappa}-\theta\chi$ (resp.\ $\pi/\sqrt{\kappa}-\theta\chi$) on the left (resp.\ right) side, where the winding is relative to a segment of the curve going straight upwards. We refer to \cite[Section 1]{ms-ig4} for the precise description of this conditional law and in particular to \cite[Figures 1.9 and 1.10]{ms-ig4} for the boundary data and the concept of winding.

For $t\geq 0$, let $\fill(\eta[0,t])$ be the hull generated by $\eta[0,t]$, i.e.\ the union of $\eta[0,t]$ and the set of points which it disconnects from $\infty$. For $\kappa\geq 8$ we have  $\fill(\eta[0,t])=\eta[0,t]$ while for $\kappa\in(4,8)$ we have that $\eta[0,t]$ is strictly contained in $\fill(\eta[0,t])$. However, in the latter case it still holds a.s.\ that $\eta(\tau_r)$ lies on the boundary of $\fill(\eta[0,\tau_r])$ for $\tau_r=\inf\{t\geq 0\,:\,|\eta(t)|\geq r \}$ and fixed $r>0$, and that $\eta|_{[\tau_r,\infty)}$ stays in $\hatC \setminus \fill(\eta[0,\tau_r])$. See Figure \ref{fig:spfill-sle}.

\begin{figure}
    \centering
    \includegraphics{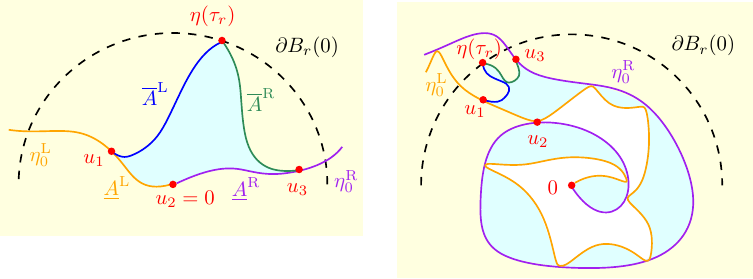}
    \caption{An illustration of $\eta[0,\tau_r]$ (light blue) for $\kappa\geq 8$ (left) and $\kappa\in(4,8)$ (right), where $\eta$ is a whole-plane space-filling SLE$_\kappa$. The complement of $\fill(\eta[0,\tau_r])$ is shown in light yellow. The arcs $\ul A^{\op{L}},\ul A^{\op{R}},\ol A^{\op{L}},\ol A^{\op{R}}$ are the subarcs of $\partial \fill(\eta[0,\tau_r])$ shown in orange, purple, blue and green, respectively, separated by the points $\eta(\tau_r),u_1,u_2,u_3$.  (Note that $\ul A^{\op{L}}$ (resp.\ $\ul A^{\op{R}}$) is only a segment of the full orange (resp.\ purple) curve, namely the segment on the boundary of $\fill(\eta[0,\tau_r])$.)}
    \label{fig:spfill-sle}
\end{figure}

Fix $r>0$. The set $\bdy \fill(\eta[0,\tau_r])$ can be divided into four distinguished arcs, which we denote as follows.
\begin{itemize}
	\item $\ul A^{\mathrm{L}}$ (resp.\ $\ul A^{\mathrm{R}}$) is the arc of $\bdy \fill(\eta[0,\tau_r])$ traced by $\eta_0^{\mathrm{L}}$ (resp.\ $\eta_0^{\mathrm{R}}$). 
	\item $\ol A^{\mathrm{L}}$ (resp.\ $\ol A^{\mathrm{R}}$) is the arc of $\bdy \fill(\eta[0,\tau_r]) $ not traced by $\eta_0^{\mathrm{L}}$ or $\eta_0^{\mathrm{R}}$ which is adjacent to $\ul A^{\mathrm{L}}$ (resp.\ $\ul A^{\mathrm{R}}$). 
\end{itemize}  
Define the $\sigma$-algebra $\mcl F_r$ by
$
\mcl F_r \defeq \sigma\left( \eta|_{[0,\tau_r]} ,\, h|_{\eta[0,\tau_r]} \right) .
$
The following is \cite[Lemma 3.2]{ghm-kpz-sle}.
\begin{lemma}\label{le:localset}
	The set $\eta[0,\tau_r]$ is a local set for $h$ in the sense of~\cite[Lemma 3.9]{ss-gff-contour}. 
	In particular, the boundary data for the conditional law of $h|_{\BB C\setminus \fill(\eta[0,\tau_r]) }$ given $\mcl F_r$ on each of the arcs $\ul A^{\mathrm{L}}$, $\ul A^{\mathrm{R}}$, $\ol A^{\mathrm{L}}$, and $\ol A^{\mathrm{R}}$ coincides with the boundary data of the corresponding flow line of $h$, and the conditional law of $h|_{\BB C\setminus \fill(\eta[0,\tau_r]) }$ is that of a Dirichlet GFF in $\BB C\setminus \fill(\eta[0,\tau_r])$ with the given boundary data. 
\end{lemma}

Let $(D,a)$ be as in \eqref{eq:Da}, and let $\u=(u_1,u_2,u_3)$ be distinct points on $\partial D\setminus \{a \}$ such that $a,u_1,u_2,u_3$ are ordered counterclockwise. Let $\wt h$ be a GFF in $D$ with the law in Lemma \ref{le:localset} if $D=\BB C\setminus \fill(\eta[0,\tau_r])$ and we let $a,u_1,u_2,u_3$ describe the points of intersection of the boundary arcs $\ul A^{\mathrm{L}}$, $\ul A^{\mathrm{R}}$, $\ol A^{\mathrm{L}}$, and $\ol A^{\mathrm{R}}$. Then since the conditional law of $\eta$ given $\cF_r$ depends only on $(D,a,\u)$, we can define a measure
$\wh\nu_{D;a\to\infty;\u}$ on curves from $a$ to $\infty$ in $D$ that describes this conditional law. 

Consider the pair $(\BB C\setminus \fill(\eta[0,\tau_r]),\eta(\tau_r))$ and let $f\colon\BB C\setminus \fill(\eta[0,\tau_r])\to\D$ be the conformal map described right below \eqref{eq:Da}, which in particular satisfies $f(\eta(\tau_r))=1$. Define $\wh h$ by
\eqb
\wh h \defeq h \circ f^{-1} - \chi \op{arg} (f^{-1})',
\label{eq:ig-coord-ch}
\eqe
Then $\wh h$ is a GFF on $\BB D$ with Dirichlet boundary data determined by $f(u_1),f(u_2),f(u_3)$ plus an $\arg$ singularity at $f(\infty)$, where we pick an arbitrary choice of branch cut for $\arg$ function; picking a different branch cut has the effect of adding a multiple of $4\pi\chi$ in the region between the two branch cuts.
In particular, the law of $\wh h$ (modulo $2\pi\chi$) depends only on the location of the points $f(u_1),f(u_2),f(u_3),f(\infty)$. On each of the four arcs $f(\ol A^{\mathrm{L}}),f(\ol A^{\mathrm{R}}),f(\ul A^{\mathrm{L}}),f(\ul A^{\mathrm{R}})$ on $\partial\D$ the boundary data will be given by a constant (depending on which arc we consider) plus $\chi$ times the winding of $\partial\D$ viewed as a curve. 
The image of $\eta|_{\R\setminus[0,\tau_r]}$ under $f$ can be constructed via the flow lines of $\wh h$ exactly as above. The boundary data of $\wh h$ along its flow line is given by $\pm \lambda'$ plus $\chi$ times the winding of the flow line, except that there is a jump of $4\pi\chi$ when crossing the branch cut in clockwise direction (i.e. $-2\chi$-flow line boundary conditions in the terminology of \cite{ms-ig4}).

Similarly as in the paragraph right after Lemma \ref{le:localset} we can define a measure $\wh\nu_{\BB D,1\to z_\infty,\wh{\u}}$ on curves  in $\BB D$ from $1$ to $z_\infty\in\BB D$ that describes the conditional law of $f\circ\eta|_{[\tau_r,\infty)}$ given $\cF_r$, where $z_\infty=f(\infty)$ and $\wh{\u}=(f(u_1),f(u_2),f(u_3))$. This conditional law can be explicitly defined in terms of the flow lines of $\wh h$. The flow lines started from $f(u_1)$ and $f(u_3)$ of angle $\pi/2$ and $-\pi/2$, respectively, will end at $f(\infty)$, and these two flow lines divide $\bbD$ into two (for $\kappa\geq 8$) or countably infinite (for $\kappa\in(4,8)$) connected components. The curve $f\circ\eta|_{[\tau_r,\infty)}$ visits precisely the complementary components that lie to the right of the flow line started from $f(u_1)$ and has the law of a chordal space-filling SLE$_\kappa$ in each such component.

The following lemma quantifies the transience of the SLE variants stated at the beginning of \cref{se:results}.

\begin{lemma}\label{le:sle_transience}
Let $\eta$ be a two-sided whole-plane \slek{}, $\kappa\le 8$, or whole-plane space-filling \slek{}, $\kappa>4$, as in \cref{it:non-spfill,it:spfill}. There exists $\beta > 0$ such that
\[
\bbP( \eta[t,\infty) \cap B(0,1) \neq \varnothing ) \lesssim t^{-\beta}
\quad\text{as } t \to \infty .
\]
\end{lemma}

\begin{proof}
The case of space-filling SLE is \cite[Proposition~6.2]{hs-mating-eucl}. The case of two-sided whole-plane \slek{}, $\kappa \le 4$, follows from \cite[Theorem~1.3]{fl-sle-transience} and (the statement for negative moments in) \cite[Lemma~1.7]{zhan-hoelder}.

It remains to prove the claim for two-sided whole-plane \sle{\kappa'}, $\kappa' > 4$. Again, by (the statement for negative moments in) \cite[Lemma~1.7]{zhan-hoelder}, it suffices to show that the probability that the SLE re-enters $B(0,1)$ after exiting $B(0,r)$ is bounded by $O(r^{-\beta})$ as $r\to\infty$.

Define $\eta'_+ = \eta'\big|_{[0,\infty)}$ and let $\eta'_-$ denote the time-reversal of $\eta'\big|_{(-\infty,0]}$. We will now argue that, for any pair of stopping times $\tau_+$ and $\tau_-$ for $\eta'_+$ and $\eta'_-$, respectively, the conditional law of $\eta'_+\big|_{[\tau_+,\infty)}$ given $\eta'_+\big|_{[0,\tau_+]}$ and $\eta'_-\big|_{[0,\tau_-]}$ is a radial SLE$_{\kappa'}(2)$ targeted at $\infty$ with a force point at $\eta'_-(\tau_-)$. By symmetry, the analogous statement will also be true for $\eta'_-\big|_{[\tau_-,\infty)}$. 

First we observe that we can couple $\eta'_+,\eta'_-$ with $h_\alpha = h-\alpha\arg(\cdot)$ where $h$ is a whole-plane GFF modulo $2\pi(\chi+\alpha)$ and $\alpha = -\frac{\sqrt{\kappa'}}{2}$ such that they are the $\alpha$-counterflow lines of $h_\alpha$ with angles $0$ resp.\ $\pi+\alpha\pi/\chi$ in the sense of \cite[Theorem~1.6]{ms-ig4}. This can be seen by observing that the angle $0$ counterflow line has the same marginal law as $\eta'_+$ by \cite[Theorem~1.6]{ms-ig4}, and then the angle $\pi+\alpha\pi/\chi$ counterflow line has the law of a chordal SLE$_{\kappa'}$ in the complement of $\eta'_+$ by \cite[Theorems~1.15 and 1.11]{ms-ig4}. The claim then follows from the fact that for each pair $\tau_+$, $\tau_-$ the set $\eta'_+[0,\tau_+] \cup \eta'_-\big[0,\tau_-]$ is local for $h$, the boundary values of $h_\alpha$ along these sets (in particular they do not exhibit singularities at their intersection points, cf.\ \cite[Section~6.1]{ms-ig1}), and the martingale characterisation of radial SLE (cf.\ \cite[last paragraph in the proof of Proposition~3.1]{ms-ig4}).

For $k\in\bbN$, let $\tau_{+,k}$ resp.\ $\tau_{-,k}$ be the first hitting times of $\partial B(0,8^k)$, and $f_k\colon \hatC\setminus\fill(\eta'_+[0,\tau_{+,k}] \cup \eta'_-\big[0,\tau_{-,k}]) \to \hatC\setminus \overline{B}(0,r_k)$ the conformal map with $f_k(\infty)=\infty$ and $f_k'(\infty)=1$. Note that $r_k \in [8^k/4,8^k]$. By \cite[Lemmas~2.4 and 2.5]{mw-cut-double}, there exists $p>0$ (not depending on $k$ and $\eta'_+\big|_{[0,\tau_{+,k}]}$, $\eta'_-\big|_{[0,\tau_{-,k}]}$) such that with conditional probability at least $p$, the curves $f_k(\eta_+)$, $f_k(\eta_-)$ disconnect $B(0,1.5r_k)$ from $\infty$ before hitting $\partial B(0,2r_k)$. By Koebe's $1/4$-theorem, $f_k^{-1}(\partial B(0,2r_k)) \subseteq \overline{B}(0,8r_k) \subseteq \overline{B}(0,8^{k+1})$. Furthermore, by Koebe's distortion theorem, $f_k^{-1}(\partial B(0,1.5r_k)) \subseteq \hatC\setminus B(0,r_k/6) \subseteq \hatC\setminus B(0,8^k/24)$. We conclude that for every $k$, conditionally on $\eta'_+\big|_{[0,\tau_{+,k}]}$, $\eta'_-\big|_{[0,\tau_{-,k}]}$, the curves $\eta'_+\big|_{[0,\tau_{+,k+1}]}$, $\eta'_-\big|_{[0,\tau_{-,k+1}]}$ will disconnect $B(0,8^k/24) \supseteq B(0,1)$ from $\infty$ with probability at least $p$. If this happens, $\eta'_+$, $\eta'_-$ will not re-enter $B(0,1)$ after exiting $B(0,8^{k+1})$. This proves the claim since the probability that this occurs for some $k=1,...,K$ is $1-O((1-p)^K)$.
\end{proof}

\section{Upper bounds}
\label{se:upper}
In this section we prove the upper bounds in our main results (\cref{thm:main_lil,thm:main_moc,thm:main_var}). The upper bounds hold for general stochastic processes whose increments satisfy a suitable moment condition, and we state these general results in \cref{sec:variation-upper-general}. In \cref{se:ss_upper,se:sle_upper} we will prove that SLE satisfies these conditions, which in particular implies the upper bound in \cref{eq:mink-tail}. We phrase part of the argument in \cref{se:ss_upper} in a more general setting for processes with a suitable scaling property and Markov-type increment condition. 
In \cref{se:01laws} we prove zero-one laws for some quantities related to SLE which will imply that the constants $c_0,c_1$ in \cref{thm:main_lil,thm:main_moc,thm:main_var} are deterministic. Using the earlier results in the section we get that these constants are finite (but we do not know at this point whether they are positive; this will be proved in Section \ref{sec:sle-lower-bounds}). In \cref{se:ball_filling}, we prove \cref{thm:ball_filling}.

\subsection{Regularity upper bounds under increment moment conditions}
\label{sec:variation-upper-general}

In this section, we let $\Phi,\varphi\colon {[0,\infty)} \to {[0,\infty)}$ be convex self-homeomorphisms with $\Phi(1) = \varphi(1) = 1$. 
Suppose that there exist $R>1$ and $n_0 \ge 1$, $n_0 \in \bbN$ such that
\begin{align}
&\Phi(x)\Phi(y) \le \Phi(R^2 xy) \quad\text{for } x,y \ge 1,\label{eq:Phi_mult}\\
&\frac{\varphi(R^k)}{\varphi(R^{k+1})} \le \frac{\varphi(R^{k-1})}{\varphi(R^k)} \quad\text{for } k \ge 1, k \in \bbN ,\nonumber\\
\text{and}\quad &\sum_{k=0}^\infty \frac{\varphi(R^k)}{\Phi(R^{k+n_0})} < \infty .\nonumber
\end{align}
We give examples of appropriate functions $\Phi,\varphi$ in \cref{ex:expo,ex:poly} below.

Let $\alpha \in {(0,1]}$, and let $(X_t)_{t \in [0,1]}$ be a separable process with values in a separable Banach space\footnote{The result in \cite{Bed07} is stated for real-valued processes, but it is not required in their proof.} that satisfies
\begin{equation}\label{eq:Xinc}
\bbE \Phi\left(\frac{\abs{X_s-X_t}}{\abs{s-t}^\alpha}\right) \le 1 \quad\text{for all } s,t \in [0,1] .
\end{equation}
We will give general results on the modulus of continuity, law of the iterated logarithm, and variation regularity for such processes.

\subsubsection{Modulus of continuity}

We review the following general result which is a special case of \cite[Corollary 1]{Bed07}. (Recall that a separable process that is uniformly continuous on a suitable countable dense subset is necessarily continuous.)

\begin{theorem}[{\cite[Corollary 1]{Bed07}}]\label{thm:moc_upper}
There exists a finite constant $K$ (depending on $\Phi$, $\varphi$, $\alpha$) such that every separable process as in \eqref{eq:Xinc} satisfies
\[ \bbE \sup_{s,t \in [0,1]} \Phi\left(\frac{\abs{X_s-X_t}}{2K\tau(\abs{t-s})}\right) \le 1 \]
where
\[ \tau(t) = \int_0^{t^\alpha} \varphi^{-1}\left(\frac{1}{2u^{1/\alpha}}\right) du . \]
\end{theorem}
The above theorem in particular says that $X$ a.s.\ admits a modulus of continuity given by a (possibly random) constant multiple of $\tau$ (as long as $\tau$ is finite).

The general result in \cite{Bed07} allows for stochastic processes indexed by a compact metric space $T$, and they give a modulus of continuity by a function $\tau(s,t)$ that may depend on both variables $s,t$ (also called ``minorising metric'' in the literature).

Note that for any $t_0 > s_0 > 0$, an application of \cref{thm:moc_upper} to the process $t \mapsto (t_0-s_0)^{-\alpha}X_{s_0+(t_0-s_0) t}$ yields
\begin{equation}\label{eq:Xsup_inc}
\bbE \sup_{s,t \in [s_0,t_0]} \Phi\left(\frac{\abs{X_s-X_t}}{2K\tau(1)\abs{s_0-t_0}^{\alpha}} \right) 
\le \bbE \sup_{s,t \in [s_0,t_0]} \Phi\left(\frac{\abs{X_s-X_t}}{2K\tau\bigl(\abs{s_0-t_0}^{-1}\abs{t-s}\bigr)\abs{s_0-t_0}^{\alpha}} \right) 
\le 1 .
\end{equation}

\subsubsection{Law of the iterated logarithm}

We get the following theorem via Theorem \ref{thm:moc_upper} and a union bound. In fact, we only use \eqref{eq:Xsup_inc} and not the full statement of \cref{thm:moc_upper}.
\begin{theorem}\label{thm:lil_upper}
For every non-decreasing function $h\colon {(0,\infty)} \to {(0,\infty)}$ with $\int_1^\infty \frac{du}{h(u)} < \infty$, every separable process as in \eqref{eq:Xinc} satisfies
\[
\limsup_{t \downarrow 0} \frac{\abs{X_t-X_0}}{t^\alpha \Phi^{-1}(h(\log\frac{1}{t}))} \le 2K\tau(1) 
\]
with the same $K$ and $\tau$ as in \cref{thm:moc_upper}.
\end{theorem}

\begin{proof}
Assume without loss of generality that $X_0 = 0$. Fix $q < 1$ and consider the events
\[ A_n = \{ \norm{X}_{\infty;[0,q^n]} \le 2K\tau(1) q^{n\alpha} \Phi^{-1}(h(n)) \} . \]
By \eqref{eq:Xsup_inc}, we have $\bbP(A_n^c) \le \frac{1}{h(n)}$, and therefore $\sum_n \bbP(A_n^c) < \infty$. By the Borel-Cantelli lemma, with probability $1$ all but finitely many $A_n$ occur.

If $A_n$ occur, then for every $t \in [q^{n+1},q^n]$ we have
\[
\abs{X(t)}
\le \norm{X}_{\infty;[0,q^{n}]} 
\le 2K\tau(1) q^{n\alpha} \Phi^{-1}(h(n)) 
\le 2K\tau(1) q^{-\alpha} t^{\alpha}\Phi^{-1}(h(\frac{\log\frac{1}{t}}{\log\frac{1}{q}})) .
\]
The result follows by observing that for every $C>0$ and $h$ with $\int \frac{1}{h} < \infty$ we can find $\tilde h$ with $\int \frac{1}{\tilde h} < \infty$ such that $\limsup_{u \to \infty} \frac{\tilde h(Cu)}{h(u)} < 1$, applying the estimate above with $C = \frac{1}{\log\frac{1}{q}}$ and $\tilde h$, and then sending $q \uparrow 1$.
\end{proof}

\subsubsection{Variation}

We now establish variation regularity of processes satisfying \eqref{eq:Xsup_inc}.
Such results can be found e.g. in \cite{tay-bm-variation} and \cite[Appendix A.4]{fv-rp-book} for some exponential $\Phi$ as described in \cref{ex:expo}. We generalise their result to the setup of \cref{sec:variation-upper-general}. Moreover, we find our proof conceptually clearer since we construct a natural way of parametrising $X$ to obtain optimal modulus of continuity $\sigma$. (Note that this is the best modulus that we can expect, cf.\@ \cref{pr:var_lil}.)

\begin{theorem}\label{thm:psivar_upper}
Suppose that $\Phi$ also satisfies $\int_2^\infty \frac{\log y}{\Phi(y^\alpha)}\,dy < \infty$.
Let $h\colon {(0,\infty)} \to {(0,\infty)}$ be a non-decreasing continuous function with $\int_1^\infty \frac{du}{h(u)} < \infty$, and such that
\[ \sigma(t) = t^\alpha \Phi^{-1}(h(\logp\frac{1}{t})) \]
is increasing with $\sigma(0^+)=0$. Then there exist $c>0$ (depending on $\Phi$, $\alpha$) and $C>0$ (depending on $\Phi$, $\alpha$, $h$) such that every separable process as in \eqref{eq:Xinc} satisfies
\[ \bbP( [X]_{\psi\text{-var}} > M ) \le \frac{C}{\Phi(cM)} \]
with $\psi=\sigma^{-1}$.
\end{theorem}

To prove this, we construct a parametrisation of $X$ that has modulus $\sigma$.
Fix some $M > 0$. Consider the intervals $I_{k,j} = [j2^{-k}, (j+1)2^{-k}]$, and let
\begin{align*}
s_{k,j} &= \sup\{s \le 1 \mid \sup_{u,v \in I_{k,j}} \abs{X_u-X_v} \le M\sigma(2^{-k}/s) \} , \\
s(t) &= \inf\{ s_{k,j} \mid t \in I_{k,j} \} .
\end{align*}
The idea is to slow down the path $X$ at time $t$ to speed $s(t)$. More precisely, define
\[
T(t) = \int_0^t \frac{dt}{s(t)}.
\]
Intuitively, this describes the elapsed time after reparametrisation.

\begin{lemma}
If $T(1) < \infty$, then
\[
\abs{X_{t_1}-X_{t_2}} \le (2^\alpha+1)M \sigma(\abs{T(t_1)-T(t_2)}) .
\]
In other words, $X \circ T^{-1}$ has modulus $(2^\alpha+1)M\sigma$.
\end{lemma}

\begin{proof}
Pick $k$ such that $\abs{t_1-t_2} \in [2^{-k-1},2^{-k}]$. Then $t_1$ and $t_2$ lie in the same or in adjacent $I_{k,j}$'s. In the former case we have
\[ \begin{split}
\abs{X_{t_1}-X_{t_2}} 
&\le M\sigma(2^{-k}/s_{k,j}) \\
&\le M 2^\alpha \sigma(\abs{t_1-t_2}/s_{k,j}) \\
&\le M 2^\alpha \sigma(\abs{T(t_1)-T(t_2)})
\end{split} \]
since $s(t) \le s_{k,j}$ for all $t \in [t_1,t_2]$.
In the latter case (say the point $j2^{-k}$ lies between them) we have
\[
\abs{X_{t_1}-X_{t_2}} \le \abs{X_{t_1}-X_{j2^{-k}}}+\abs{X_{t_2}-X_{j2^{-k}}} 
\]
and proceed similarly. Note that we have $\abs{t_1-j2^{-k}} \le 2^{-k-1}$ or $\abs{t_2-j2^{-k}} \le 2^{-k-1}$, so for that term we get the same estimate without the factor $2^\alpha$.
\end{proof}

We claim that
\begin{equation}\label{eq:psivar_norm}
[X]_{\psi\text{-var}} \le M(2^\alpha+1)T(1)^\alpha .
\end{equation}
Indeed, for any partition $0 = t_0 < t_1 < ... < t_m = 1$ we have, by the lemma above,
\[ \begin{split}
\sum_i \psi\left(\frac{\abs{X_{t_i}-X_{t_{i-1}}}}{\wt M}\right)
&\le \sum_i \psi\left(\frac{(2^\alpha+1)M}{\wt M}\sigma(\abs{T(t_i)-T(t_{i-1})})\right) \\
&\le \sum_i \left(\frac{(2^\alpha+1)M}{\wt M}\right)^{1/\alpha} \abs{T(t_i)-T(t_{i-1})} \\
&= \left(\frac{(2^\alpha+1)M}{\wt M}\right)^{1/\alpha} T(1) .
\end{split} \]

\begin{proof}[Proof of \cref{thm:psivar_upper}]
By \eqref{eq:psivar_norm}, we have
\[ \begin{split}
\bbP( [X]_{\psi\text{-var}} > 6M ) 
&\le \bbP( T(1)^\alpha > 2 ) \\
&\le \bbP\left( \int_0^1 \frac{1}{s(t)} 1_{s(t)<1}\,dt > 1 \right) .
\end{split} \]
We now estimate
\[ \bbE \int_0^1 \frac{1}{s(t)} 1_{s(t)<1}\,dt = \int_0^1 \bbE \frac{1}{s(t)} 1_{s(t)<1}\,dt \]
and then
\[ \begin{split}
\bbE\left[ \frac{1}{s(t)}1_{s(t)<1} \right]
&= \int_0^\infty \bbP\left( \frac{1}{s(t)}1_{s(t)<1} > y \right) dy \\
&= \int_0^\infty \bbP\left( s(t) < \frac{1}{y} \wedge 1 \right) dy .
\end{split} \]
For $y \ge 1$, we get from \eqref{eq:Xsup_inc}
\[ \begin{split}
\bbP( s_{k,j} < 1/y ) 
&= \bbP( \sup_{u,v \in I_{k,j}} \abs{X_u-X_v} > M\sigma(2^{-k}y) ) \\
&= \bbP\left( \sup_{u,v \in I_{k,j}} \frac{\abs{X_u-X_v}}{2^{-k\alpha}} > M y^\alpha \Phi^{-1}(h(\logp(2^{k}/y))) \right) \\
&\le \frac{1}{\Phi( \frac{M y^\alpha \Phi^{-1}(h(\logp(2^{k}/y)))}{2K\tau(1)} )} \\
&\le \frac{1}{\Phi(\frac{M}{2K\tau(1)R^4})\Phi(y^\alpha) h(\logp (2^k/y))}
\end{split} \]
where we have applied \eqref{eq:Phi_mult} in the last step. 
Hence,
\[ \begin{split}
\bbP( s(t) < 1/y ) \le \sum_{k \in \bbN} \bbP( s_{k,j} < 1/y ) 
&\le \frac{1}{\Phi(\frac{M}{2K\tau(1)R^4})\Phi(y^\alpha)} \sum_{k \in \bbN} \frac{1}{h(\logp (2^k/y))} \\
&\lesssim \frac{1}{\Phi(\frac{M}{2K\tau(1)R^4})\Phi(y^\alpha)} \left( \log y + \sum_{2^k > y} \frac{1}{h(k\log 2 - \log y)} \right) \\
&\le \frac{1}{\Phi(\frac{M}{2K\tau(1)R^4})\Phi(y^\alpha)} (\log y + C) ,
\end{split} \]
and finally
\[ 
\int_0^\infty \bbP\left( s(t) < \frac{1}{y} \wedge 1 \right) dy
\le \frac{1}{\Phi(\frac{M}{2K\tau(1)R^4})} \left( C + \int_2^\infty \frac{\log y + C}{\Phi(y^\alpha)}\,dy \right)
\lesssim \frac{1}{\Phi(\frac{M}{2K\tau(1)R^4})}
\]
by the extra assumption on $\Phi$. Thus we have shown
\[ 
\bbP( [X]_{\psi\text{-var}} > 6M ) 
\le \bbE \int_0^1 \frac{1}{s(t)} 1_{s(t)<1}\,dt
\lesssim \frac{1}{\Phi(\frac{M}{2K\tau(1)R^4})} .
\]
\end{proof}

\subsubsection{Examples}

\begin{example}\label{ex:expo}
Suppose that $\beta > 0$ and
\[ \Phi(x) \asymp \exp(cx^\beta) \quad\text{as } x \to \infty . \]
Then the results of \cref{sec:variation-upper-general} apply with $\varphi = \Phi$ and e.g.\ $h(u) = u^2$.
Hence, the modulus of continuity is given by
\[ \tau(t) \asymp t^\alpha (\logp\frac{1}{t})^{1/\beta} , \]
the law of the iterated logarithm reads
\[ \sigma(t) \asymp t^\alpha (\log\log\frac{1}{t})^{1/\beta} , \]
and the variation regularity is
\[ \psi(x) \asymp x^{1/\alpha} (\logp\logp\frac{1}{x})^{-1/(\alpha\beta)} . \]
Hence, we recover the results of \cite[Section A.4]{fv-rp-book} which consider the case $\beta=2$, and generalize these results to arbitrary $\beta>0$. 

Note that Brownian motion satisfies this with $\alpha = 1/2$, $\beta = 2$. For SLE, we will use this result with $\alpha = 1/d$, $\beta = \frac{d}{d-1}$, where $d$ is the dimension of the process as in \eqref{eq:d}.
\end{example}

\begin{example}\label{ex:poly}
Suppose
\[ \Phi(x) = x^p, \quad p>\frac{1}{\alpha} . \]
Define $h$ by e.g.\ $h(u)=u^{1+\varepsilon}$ or $h(u)=u(\logp u)(\logp\logp u)\cdots(\logp\cdots\logp u)^{1+\varepsilon}$, and set $\varphi(x) = \frac{x^p}{h(\logp x)}$. Then the modulus of continuity is given by
\[ \tau(t) \asymp t^{\alpha-1/p}\, h(\logp\frac{1}{t})^{1/p} , \]
the law of the iterated logarithm reads
\[ \sigma(t) = t^\alpha\, h(\log\frac{1}{t})^{1/p} , \]
and the variation regularity is
\[ \psi(x) \asymp x^{1/\alpha}\, h(\logp\frac{1}{x})^{-1/(p\alpha)} . \]
Our result on the modulus of continuity is a sharper version of Kolmogorov's continuity theorem, which yields exponents arbitrarily close to $\alpha-1/p$.
\end{example}

\subsection{Diameter upper bound given Markov-type increments with a scaling property}
\label{se:ss_upper}

In this and the next subsection we show that SLE satisfies the conditions of \cref{ex:expo}. The argument in this subsection concerns general processes whose increments satisfy a suitable scaling and Markov-type property. (We remark here that the argument does \emph{not} apply to stable processes since they violate \eqref{eq:ss_markov} due to large jumps.) 

For this subsection, let $\eta$ denote any right-continuous stochastic process. Denote by $(\cF_t)_{t \ge 0}$ the filtration generated by $\eta$, and $\tau_r = \inf\{ t \ge 0 \mid \abs{\eta(t)} \ge r\}$. Let $d>1$, and suppose there exist $p<1$ and $\ell>0$ such that almost surely
\begin{equation}\label{eq:ss_markov}
    \bbP( \tau_{r+6\lambda} \le \tau_r+\ell\lambda^d \mid \cF_{\tau_r} ) < p
\end{equation}
for any $r>0$ and $\lambda > 0$. For such processes, we show the following statement.

\begin{proposition}\label{pr:ss_tail}
There exist $\ell >0$ and $c_1,c_2>0$ such that
\[ \bbP( \tau_{r+\lambda r'} \le \tau_r+\ell \lambda^d \mid \cF_{\tau_r} ) < c_1\exp(-c_2 (r')^{d/(d-1)}) \]
for all $r,r'>0$ and $\lambda > 0$.
\end{proposition}

The number $6$ in $\tau_{r+6}$ in \eqref{eq:ss_markov} is not significant, and can be replaced by any $s>0$. Of course, in all the statements, the constants may depend on $s$.

\begin{lemma}\label{le:ss_large_inc}
For any $\varepsilon > 0$ there exists $c>0$ such that
\[ \bbP( \tau_{r+c\lambda} \le \tau_r+\ell\lambda^d \mid \cF_{\tau_r} ) < \varepsilon \]
for any $r>0$ and $\lambda > 0$.
\end{lemma}

\begin{proof}
    This follows by applying \eqref{eq:ss_markov} iteratively.
\end{proof}

\begin{lemma}\label{le:ss_iter}
There exist $\ell>0$ and $c_1,c_2>0$ such that
\[ \bbP( \tau_{r+\lambda r'} \le \tau_r+\ell \lambda^d r' \mid \cF_{\tau_r} ) < c_1\exp(-c_2 r') \]
for all $r,r'>0$ and $\lambda > 0$.
\end{lemma}

\begin{proof}
Pick $\varepsilon < 1/4$ and let $c$ be the constant from \cref{le:ss_large_inc}. It suffices to show
\[ \bbP( \tau_{r+c\lambda r'} \le \tau_r+\frac{\ell}{2}\lambda^d r' \mid \cF_{\tau_r} ) < c_1\exp(-c_2 r') \]
for $r' \in 2\bbN$.

On the event that $\tau_{r+c\lambda r'} \le \tau_r+\frac{\ell}{2}\lambda^d r'$ there must exist integers $k_1,...,k_{r'/2}$ such that $\tau_{r+c\lambda k_i} \le \tau_{r+c\lambda (k_i-1)}+\ell\lambda^d$ for $i=1,...,r'/2$. For each such choice of $k_1,...,k_{r'/2}$, we can apply \cref{le:ss_large_inc} iteratively (each time conditionally on $\cF_{\tau_{r+c\lambda (k_i-1)}}$), so that
\[
\bbP( \tau_{r+c\lambda k_i} \le \tau_{r+c\lambda (k_i-1)}+\ell\lambda^d \text{ for all } i \mid \cF_{\tau_r})
\le \varepsilon^{r'/2} .
\]
Since the number of such choices of $k_1,...,k_{r'/2}$ is $\binom{r'}{r'/2} \le 2^{r'}$, we get
\[
\bbP( \tau_{r+c\lambda r'} \le \tau_r+\frac{\ell}{2}\lambda^d r' \mid \cF_{\tau_r})
\le 2^{r'} \varepsilon^{r'/2} .
\]
The claim follows since we picked $\varepsilon < 1/4$.
\end{proof}

\begin{proof}[Proof of \cref{pr:ss_tail}]
This follows by applying \cref{le:ss_iter} with $r' = (r'')^{d/(d-1)}$ and $\lambda = \lambda''(r'')^{-1/(d-1)}$.
\end{proof}

\subsection{Markov-type increment bound for SLE}
\label{se:sle_upper}

In this subsection we verify \eqref{eq:ss_markov} for SLE. Once we have done this, \cref{pr:ss_tail} and the self-similarity of $\eta$ immediately imply the following result. The proposition is a strengthening of \cite[Lemma 1.7]{zhan-hoelder} which proved that $\bbE[\diam(\eta[0,1])^a]<\infty$ for any $a>0$.

\begin{proposition}\label{pr:diam_tail_upper}
Let $\eta$ be a two-sided whole-plane \slek{} or a whole-plane space-filling \slek{} as specified in \cref{se:results}. There exists $c > 0$ such that
\[ \bbE\exp\left(c\left(\frac{\diam(\eta[s,t])}{\abs{s-t}^{1/d}}\right)^{d/(d-1)}\right) < \infty \]
for all $s<t$.
\end{proposition}

In particular, the conditions from \cref{sec:variation-upper-general,ex:expo} are satisfied with $\alpha=1/d$, $\beta=\frac{d}{d-1}$. This proves the upper bounds in \cref{thm:main_lil,thm:main_moc,thm:main_var}.

In fact, we will prove a stronger result than Proposition \ref{pr:diam_tail_upper} below, namely that there exist finite constants $c_1,c_2>0$ such that for any $(D,a)$ as in \eqref{eq:Da} and $u \in \partial D \setminus \{a\}$ we have
\begin{equation}
\nu_{D;a\to\infty;u}(\Cont(\eta[0,\tau_{\abs{a}+r}]) < \ell) \le c_1 \exp\left(-c_2 l^{-1/(d-1)} r^{d/(d-1)}\right)
\label{eq:cont_utail_cond}
\end{equation}
for all $r>0$. The same is true for $\wh\nu_{D;a\to\infty;\u}$ defined in \cref{se:prelim_spf}.

Since we parametrise $\eta$ by its Minkowski content, we can phrase the condition \eqref{eq:ss_markov} as follows. (Recall that the Minkowski content is additive over SLE curve segments.) As before, we write $\tau_r = \inf\{ t \ge 0 \mid \abs{\eta(t)} = r\}$.

\begin{lemma}\label{le:sle_markov_inc}
There exists $l>0$ and $p<1$ such that for any $r>0$ we have
\[ 
\P\left(\Cont(\eta[\tau_r,\tau_{r+6}]) < \ell \mmiddle| \eta\big|_{[0,\tau_r]} \right) < p . 
\]
\end{lemma}

\begin{proof}[Proof of \cref{le:sle_markov_inc} in case of space-filling \slek{}]
In case of $r > 1$, the statement (even with $\tau_{r+1}$ instead of $\tau_{r+6}$) is precisely \cite[Lemma 3.1]{ghm-kpz-sle}. In case $r \le 1$, conditioning on $\eta[0,\tau_2]$ we get
\[ 
\P\left(\Cont(\eta[\tau_2,\tau_3]) < \ell \mmiddle| \eta\big|_{[0,\tau_2]} \right) < p . 
\]
This implies \cref{le:sle_markov_inc} for arbitrary $r > 0$.
\end{proof}

\begin{remark}\label{rm:ball_filling}
In the case of space-filling \slek{}, the proof shows that there exist $\ell_0,\delta>0$ such that 
\[ \wh\nu_{D;a\to\infty;\u}\left( \eta[0,\tau_{\abs{a}+r}] \text{ contains $\delta r$ disjoint balls of radius $\ell_0$} \right) > 1-c_1\exp(-c_2 r) . \]
By scaling, we get
\eqb \wh\nu_{D;a\to\infty;\u}\left( \eta[0,\tau_{\abs{a}+r}] \text{ contains $\delta \ell^{-1} r^2$ disjoint balls of radius $\ell_0 \ell r^{-1}$} \right) > 1-c_1\exp(-c_2 \ell^{-1} r^2) . 
\label{eq:balls}
\eqe
Notice that on this event we have
\[ \Cont(\eta[0,\tau_{\abs{a}+r}]) \gtrsim \ell , \]
so \eqref{eq:balls} is a stronger version of \cref{pr:diam_tail_upper} and \eqref{eq:cont_utail_cond}.
\end{remark}

In the remainder of the section, we prove \cref{le:sle_markov_inc} for two-sided whole-plane \slek{}. Recall from \cref{se:prelim_sle} that the restriction $\eta\big|_{[0,\infty)}$ is a whole-plane \slek{}$(2)$ from $0$ to $\infty$ with force point at $0$. Therefore the statement is equivalent when we consider whole-plane \slek{}$(2)$.

\begin{lemma}
There exists $\ell>0$ and $p<1$ such that for any $(D,a)$ as in \eqref{eq:Da} and $u \in \partial D \setminus \{a\}$ we have
\[ \nu_{D;a\to\infty;u}(\Cont(\eta[0,\tau_{\abs{a}+6}]) < \ell) < p . \]
\end{lemma}

\begin{proof}
We show the statement in case of $\abs{a} \ge 1$ and with $\tau_{\abs{a}+5}$ instead of $\tau_{\abs{a}+6}$. In case $\abs{a} < 1$, considering $\eta\big|_{[\tau_1,\tau_6]}$ conditionally on $\eta[0,\tau_1]$ gives us
\[ \nu_{D;a\to\infty;u}(\Cont(\eta[\tau_1,\tau_6]) < \ell) < p . \]

Let $f\colon D \to \bbD$ be the conformal map described in the paragraph below \eqref{eq:Da}, and $\varepsilon_0 > 0$ and $\alpha$ as in \cref{le:ghm}. Denote by $B(\alpha,\varepsilon_0)$ the $\varepsilon_0$-neighbourhood of $\alpha$. By the definitions above, $f\circ\eta$ leaves $B(\alpha,\varepsilon_0)$ before $\eta$ hits radius $\abs{a}+5$.

Although $\alpha$ depends on $D$, it can be picked among a finite number of nearest-neighbour paths. Therefore it suffices to show that for given $\varepsilon_0$ and $\alpha$, there exists $\ell>0$ and $q>0$ such that the following holds. Let $w \in \bbD \setminus B(\alpha,\varepsilon_0)$ and $\nu_{\bbD;1\to w;u}$ a radial \slek$(2)$ with a force point $u \in \partial\bbD$. Then
\[ \nu_{\bbD;1\to w;u}(\Cont(\eta[0,\sigma_{\alpha,\varepsilon_0}] \cap B(0,\varepsilon_0/2)) > \ell) > q \]
where $\sigma_{\alpha,\varepsilon_0}$ is the exit time of $B(\alpha,\varepsilon_0)$.

Indeed, since $\abs{(f^{-1})'} \asymp 1$ in a neighbourhood of $0$ (due to Koebe's distortion theorem), by the transformation rule for Minkowski content \eqref{eq:cont_transf} this will imply $f^{-1}(\eta[0,\sigma_{\alpha,\varepsilon_0}] \cap B(0,\varepsilon_0/2))$ has Minkowski content at least a constant times $\ell$.

To show the claim, we need to find $\ell$ and $q$ such that the bound holds uniformly over all target points and force points. For concreteness, let us map $(\bbD,1,0)$ to $(\bbH,0,i)$. Then the image is an \slek{} in $\bbH$ with force points $(u,2)$, $u \in \bbR$, and $(w,\kappa-8)$, $w \in \bbH \setminus B(\alpha,\varepsilon_0)$ (cf.\@ \cite[Theorem 3]{sw-sle-coordinate-change}). Since the force point $w$ lies outside $B(\alpha,\varepsilon_0)$, we can disregard it until the exit time of $B(\alpha,\varepsilon_0/2)$ since the density between the corresponding SLE measures is uniformly bounded, regardless of the locations of $u$ and $w$ (cf.\@ \cref{le:sle_abs_cont}). Hence we are reduced to proving the following statement.
\end{proof}

\begin{lemma}
Let $\alpha$ be a simple path in $\bbH$ from $0$ to $i$, and $\varepsilon_0 > 0$. There exists $\ell>0$ and $q>0$ such that the following holds. Let $\nu_{\bbH,0\to\infty;u}$ denote chordal \slek$(2)$ with a force point $u \in \partial\bbH$. Then
\[ \nu_{\bbH,0\to\infty;u}(\Cont(\eta[0,\sigma_{\alpha,\varepsilon_0}] \cap B(i,\varepsilon_0)) > \ell) > q \]
where $\sigma_{\alpha,\varepsilon_0}$ is the exit time of $B(\alpha,\varepsilon_0)$.
\end{lemma}

\begin{proof}
Case 1: Suppose that $u \notin B(0,\varepsilon_0)$. Then the density between the laws of \slek$(2)$ and \slek$(0)$ is uniformly bounded until the exit of $B(\alpha,\varepsilon_0/2)$ (cf.\@ \cref{le:sle_abs_cont}). Therefore it suffices to consider \slek$(0)$. There is a positive probability that $\eta$ follows $\alpha$ within $\varepsilon_0/2$ distance. Moreover, since the Minkowski content on each sub-interval is almost surely positive, there is a positive probability that also $\Cont(\eta[0,\sigma_{\alpha,\varepsilon_0}] \cap B(i,\varepsilon_0)) > \ell$ with sufficiently small $\ell>0$.

Case 2: Suppose $u$ is arbitrary. Let $t>0$ be a small time. Denote by $f_t$ the conformal map from $\bbH \setminus \fill(\eta[0,t])$ to $\bbH$ with $f_t(\eta(t))=0$ and $f_t(z) = z+O(1)$. Then there exist $c_1>0$ and $q>0$ (independent of $u$) such that with probability at least $q$ the following occur:\\
1. $\eta[0,t] \subseteq B(0,\varepsilon_0/4)$,\\
2. $\abs{f_t(u)} \ge c_1$.\\
Indeed, $u_t = f_t(u)$ is a Bessel process of positive index started at $u$ (this follows directly from the definition of \slek$(2)$). By the monotonicity of Bessel processes in the starting point, it suffices to consider $u=0$. The claim follows since the Bessel process stopped at a deterministic time is almost surely positive.

It follows that if $t$ is chosen small enough, we have $\abs{f_t(z)-z} < \varepsilon_0/2$ for every $z \in \partial B(\alpha,\varepsilon_0) \subseteq \bbH$, and therefore $B(\alpha,\varepsilon_0/2) \subseteq f_t(B(\alpha,\varepsilon_0))$. Then, applying Case 1 to $f_t \circ \eta\big|_{[t,\infty)}$ with $\varepsilon_0$ replaced by $\varepsilon_0/2 \wedge c_1$ implies the claim.
\end{proof}

\subsection{Zero-one laws and upper bounds on the regularity of SLE}
\label{se:01laws}

In this subsection, we show the upper bounds in our main results (\cref{thm:moc_upper,thm:lil_upper,thm:psivar_upper}), which is stated as \cref{pr:upperbounds} below. To show that the constants $c_0,c_1$ are deterministic, we prove that they satisfy zero-one laws.

We begin by proving analogues of Blumenthal's and Kolmogorov's zero-one laws for SLE. Define $\cF_t = \sigma(\eta(s), s \in [0,t])$, and $\cF_{t+} = \bigcap_{s>t} \cF_s$. Moreover, denote by $\cG$ the shift-invariant $\sigma$-algebra, i.e. the sub-$\sigma$-algebra of $\bigcap_{t > 0} \sigma(\eta(s), s>t)$ consisting of events $A$ such that $(\eta(t))_{t \ge 0} \in A$ if and only if $(\eta(t_0+t))_{t \ge 0} \in A$ for any $t_0 \ge 0$. Note that we are considering paths restricted to $t \in \bbR^+$.

\begin{proposition}\label{pr:01law}
For two-sided whole-plane \slek{} and whole-plane space-filling \slek{} as in \cref{it:non-spfill,it:spfill}, the $\sigma$-algebras $\cF_{0+}$ and $\cG$ are trivial (in the sense that they contain only events of probability $0$ and $1$).
\end{proposition}

\begin{proof}
Denote by $s(t)$ the logarithmic capacity of $\eta[0,t]$, and write $\hat\cF_u \defeq \cF_{s^{-1}(u)}$. Recall that $e^{s(t)}$ is comparable to $\diam(\eta[0,t])$. Note also that $s^{-1}(u) = \Cont(\hat\eta[0,u])$ where $\hat\eta$ is $\eta$ parametrised by capacity.

We claim that $\hat\cF_{-\infty+} = \cF_{0+}$. Let $A \in \hat\cF_{-\infty+}$. We want to show that $A \in \cF_t$ for any $t>0$. Since $s(t) \downarrow -\infty$ as $t \downarrow 0$, we can write $A = \bigcup_{u \in \bbZ} A \cap \{s^{-1}(u)\le t\}$. By definition, since $A \in \hat\cF_u$ for any $u$, we have $A \cap \{s^{-1}(u)\le t\} \in \cF_t$, and hence $A \in \cF_t$. The other inclusion is analogous.

For whole-plane space-filling \slek{}, it is shown in \cite[Lemma 2.2]{hs-mating-eucl} that for a whole-plane GFF modulo $2\pi\chi$, the $\sigma$-algebra $\bigcap_{r>0} \sigma( h\big|_{B(0,r)})$ is trivial. Since \cref{le:localset} implies $\hat\cF_{-\infty+} \subseteq \bigcap_{r>0} \sigma( h\big|_{B(0,r)})$, the former is also trivial. 

We now prove the proposition for two-sided whole-plane \slek{}, or rather whole-plane \slek$(2)$ since we are restricting to $t \ge 0$. Let $\hat\eta$ denote whole-plane \slek$(2)$ parametrised by capacity. Since initial segments of $\hat\eta$ and $\eta$ determine each other, we have the identity $\hat\cF_{-\infty+} = \bigcap_{u\in\bbR}\sigma(\hat\eta(u'), u' \le u)$. We show that $\hat\cF_{-\infty+}$ is independent of $\sigma(\hat\eta(u), u \in \bbR)$ which in particular implies that $\hat\cF_{-\infty+}$ is independent of itself, and therefore trivial.

Fix arbitrary $u_1<...<u_r$, and let $g$ be some bounded continuous function. From the Markov property of the driving process and \cite[Lemma 4.20]{law-conformal-book}, it follows that
\[ 
\bbE[ g(\hat\eta(u_1),...,\hat\eta(u_r)) \mid \hat\cF_u ] \to \bbE[ g(\hat\eta(u_1),...,\hat\eta(u_r)) ] 
\quad\text{as } u \downarrow -\infty .
\]
On the other hand, by backward martingale convergence, we also have
\[ 
\bbE[ g(\hat\eta(u_1),...,\hat\eta(u_r)) \mid \hat\cF_u ] \to \bbE[ g(\hat\eta(u_1),...,\hat\eta(u_r)) \mid \hat\cF_{-\infty+} ] 
\quad\text{as } u \downarrow -\infty 
\]
which implies that $\sigma(\hat\eta(u_1),...,\hat\eta(u_r))$ is independent of $\cF_{-\infty+}$. Since this is true for any choice of $u_1<...<u_r$, we must have that $\sigma(\hat\eta(u), u \in \bbR)$ is independent of $\hat\cF_{-\infty+}$.

For the triviality of $\cG$ we show the following statement: Denote by $\wt\eta\colon (-\infty,\infty)\to\bbC$ the time-reversal of $\eta$, parametrised by log conformal radius of its complement relative to the origin, and by $\wt\cF_{-\infty+}$ the infinitesimal $\sigma$-algebra of $\wt\eta$. Then we claim
\begin{equation}\label{eq:sigmaalg_reversal}
    \cG = \wt\cF_{-\infty+} \quad\text{(modulo null sets).}
\end{equation}
This will imply the triviality of $\cG$ since for whole-plane \slek{}$(2)$, the reversibility and the previous step imply $\wt\cF_{-\infty+}$ is trivial, and for space-filling \slek{}, we have triviality of $\bigcap_{R>0}\sigma(h\big|_{\bbC\setminus B(0,R)})$ by \cite[Lemma 2.2]{hs-mating-eucl}.

Now we show \cref{eq:sigmaalg_reversal}. The inclusion $\wt\cF_{-\infty+} \subseteq \cG$ is easy to see. Indeed, for any fixed $t_0$ the time reversals (parametrized by log conformal radius as before) of $\eta$ and $\eta(t_0+\cdot)$ agree until hitting some circle $\partial B_r(0)$ (with $r$ random and depending on $t_0$). We are left to show $\cG \subseteq \wt\cF_{-\infty+}$. For any $A \in \cG$ and $R>0$, we need to show $A \in \wt\cF_R$, where $\wt\cF_R$ is the $\sigma$-algebra generated by $\wt\eta$ up to the first hitting of circle $\partial B_R(0)$ (recall that conformal radius and radius are comparable up to a factor). Note that for any two curves $\wt\eta_1,\wt\eta_2$ starting at $\infty$ that agree until hitting circle $\partial B_R(0)$, their time-reversals agree after their last exit of $B_R(0)$. Consequently, when we parametrise their time-reversals by Minkowski content (denoted by $\eta_1,\eta_2$), they will agree up to a time-shift after their last exit of $B_R(0)$. In particular, $1_A(\eta_1) = 1_A(\eta_2)$. This implies $\bbP(A \mid \wt\cF_R)$ takes only the values $0,1$, or equivalently $A \in \wt\cF_R$ (modulo null sets).
\end{proof}

\begin{remark}
We believe that also the tail $\sigma$-algebra $\bigcap_{t > 0} \sigma(\eta(s), s>t)$ is trivial. Proving this requires extra work since the cumulative Minkowski content and hence the parametrisation of $\eta$ at large times \emph{does} depend on its initial part. An interesting consequence of the tail triviality would be that the measure-preserving maps $T_{t_0}\colon \eta \mapsto \eta(t_0+\cdot)-\eta(t_0)$ (now seen as paths on $t \in \bbR$) are ergodic, i.e.\@ any event $A \in \sigma(\eta)$ that is invariant under $T_{t_0}$ has probability $0$ or $1$. 
\end{remark}

The above proposition implies that the limits $c_2,c_3$ in \cref{thm:main_lil} are deterministic. We now show that the limits $c_0,c_1$ in \cref{thm:main_moc,thm:main_var} are also deterministic.

\begin{proposition}\label{pr:01law_mv}
There exist deterministic constants $c_0,c_1$ (possibly $0$) such that almost surely the following identities hold for any non-trivial bounded interval $I\subseteq\bbR$.
\begin{enumerate}[label=(\roman*)]
\item\label{it:moc_det}
$\displaystyle\lim_{\delta\downarrow 0} \sup_{s,t \in I, \abs{t-s}<\delta} \frac{\abs{\eta(t)-\eta(s)}}{\abs{t-s}^{1/d}(\log \abs{t-s}^{-1})^{1-1/d}} = c_1$.
\item\label{it:var_det}
$\displaystyle\lim_{\delta\downarrow 0} \sup_{(t_i) \subseteq I, \abs{t_{i+1}-t_i}<\delta} \sum_i \psi(\abs{\eta(t_{i+1})-\eta(t_i)}) = c_0\abs{I}$ \quad where $\psi(x) = x^d(\log\log\frac{1}{x})^{-(d-1)}$.
\end{enumerate}
\end{proposition}

\begin{proof}
\ref{it:moc_det}
For an interval $I$ define
\[ S_I \defeq \lim_{\delta\downarrow 0}\sup_{t,s\in I , \abs{t-s}<\delta} \frac{\abs{\eta(t)-\eta(s)}}{\abs{t-s}^{1/d}(\log \abs{t-s}^{-1})^{1-1/d}}. \]
The law of $S_I$ is independent of the choice of $I$ by scale-invariance in law of $\eta$ and since for any fixed $a>0$ it holds that $\log ((a\Delta t)^{-1})=(1+o(1))\log((\Delta t)^{-1})$ as $\Delta t\rta 0$.

We claim that $S_I = S_{I'}$ almost surely for any two intervals $I,I'$. Indeed, in case $I \subseteq I'$, we have $S_I \le S_{I'}$ by the definition of $S_I$. But then the two random variables can only have the same law if they are almost surely equal. For general $I,I'$ apply the argument iteratively.

It follows that $S_{[0,t]} = S_{[0,1]}$ almost surely for every $t$. Letting $t \downarrow 0$, we see that $S_{[0,1]}$ is (up to null sets) measurable with respect to $\cF_{0+}$, and therefore deterministic by \cref{pr:01law}.

\ref{it:var_det}
Let
\[ V_I \defeq \lim_{\delta\downarrow 0}\sup_{(t_i) \subseteq I, \abs{t_{i+1}-t_i}<\delta}\sum_i \psi(\abs{\eta(t_{i+1})-\eta(t_i)}) . \]
We claim that $V_I = \abs{I}V_{[0,1]}$ almost surely. This will imply $V_{[0,t]} = tV_{[0,1]}$ for all rational $t$, and by continuity for all $t$. As before, we conclude that $V_{[0,1]}$ is measurable with respect to $\cF_{0+}$, and therefore deterministic.

Note that $V$ is additive, i.e.
\[ V_{[t_1,t_2]}+V_{[t_2,t_3]} = V_{[t_1,t_3]} \quad\text{for } t_1 < t_2 < t_3 . \]
Moreover, by scaling and translation-invariance of $\eta$, the random variable $V_I$ has the same law as $\abs{I}V_{[0,1]}$. Therefore the claim follows from the lemma below. (Note that we have shown in the previous subsections that $V$ has exponential moments.)
\end{proof}

\begin{lemma}
Let $X,Y,Z$ be random variables with the same law and finite second moments. If $\lambda X+ (1-\lambda)Y = Z$ for some $\lambda \neq 0,1$, then $X=Y=Z$ a.s.
\end{lemma}

\begin{proof}
We have
\[ \bbE[Z^2] = \lambda^2\bbE[X^2]+(1-\lambda)^2\bbE[Y^2]+2\lambda(1-\lambda)\bbE[XY] \]
and hence (using $\bbE[X^2]=\bbE[Y^2]=\bbE[Z^2] < \infty$)
\[ \bbE[Z^2] = \bbE[XY] \le (\bbE[X^2]\bbE[Y^2])^{1/2} = \bbE[Z^2] , \]
i.e. the Cauchy-Schwarz inequality holds with equality. This means $X,Y$ are linearly dependent. The claim follows.
\end{proof}

\begin{proposition}\label{pr:upperbounds}
The assertions of
\cref{thm:main_var,thm:main_moc,thm:main_lil} hold, except that the constants $c_0,c_1$ may take their value in $[0,\infty)$ (instead of $(0,\infty)$).
\end{proposition}

\begin{proof}
By \cref{pr:diam_tail_upper}, the condition \eqref{eq:Xinc} is satisfied with a function $\Phi$ as given in Example \ref{ex:expo} with $\alpha=1/d$ and $\beta=\frac{d}{d-1}$. \cref{thm:psivar_upper} then yields the last assertion of \cref{thm:main_var}, and \cref{pr:01law_mv}\ref{it:var_det} shows that \eqref{eq:main_var} holds for some deterministic $c_0\in[0,\infty)$. Similarly, \cref{thm:moc_upper} yields the second assertion of \cref{thm:main_moc}, and \cref{pr:01law_mv}\ref{it:moc_det} shows that \eqref{eq:main-moc} holds for some deterministic constant $c_1\in[0,\infty)$. Finally, the statement of \cref{thm:main_lil} for $t\to 0$ (with $c_2 \ge 0$) follows from \cref{thm:lil_upper} and \cref{pr:01law}. The statement of \cref{thm:main_lil} for $t\to\infty$ (with $c_3 \ge 0$) follows from the same argument by setting $\wt{A}_n = \{ \norm{\eta}_{\infty;[0,\wt{q}^n]} \le \wt{c}\wt{q}^{n\alpha} \Phi^{-1}(n^2) \}$ with a fixed constant $\wt{q} > 1$. Then \eqref{eq:Xinc} implies $\bbP(\wt A_n^c) \lesssim n^{-2}$, and the result follows by the exact same argument as in \cref{thm:lil_upper} together with the triviality of the $\sigma$-algebra $\mathcal G$ in \cref{pr:01law}.
\end{proof}

\subsection{Proof of \cref{thm:ball_filling}}
\label{se:ball_filling}

In essence, we follow the proof of \cref{thm:moc_upper} using the stronger input given in \cref{rm:ball_filling}. By stationarity, it suffices to prove the result on the interval $I=[0,1]$.

Denote by $A_{s,r}^{\ell,k}$ the event that $\eta[s,\tau_{s,r}]$ contains $k$ disjoint balls of radius $\ell$ where $\tau_{s,r} = \inf\{ t>s \mid \abs{\eta(t)-\eta(s)} \ge r\}$. For $n \in \bbN$, by \cref{rm:ball_filling},
\[ \bbP\left( \left(A_{s,\, u 2^{-n/2} n^{1/2}}^{u^{-1} 2^{-n/2} n^{-1/2},\, \delta u^2 n}\right)^c \right) \lesssim \exp(-c_2 u^2 n) . \]
Summing over $s = j\pi\delta 2^{-n}$, $j=0,...,(\pi\delta)^{-1} 2^n-1$, yields
\[ \bbP\left( \bigcup_{s = j\pi\delta 2^{-n}} \left(A_{s,\, u 2^{-n/2} n^{1/2}}^{u^{-1} 2^{-n/2} n^{-1/2},\, \delta u^2 n}\right)^c \right) \lesssim 2^n \exp(-c_2 u^2 n) \lesssim \exp(-\tilde c_2 u^2 n) \]
for sufficiently large $u$. Summing over $n \ge n_0$ yields
\[ \bbP\left( \bigcup_{n \ge n_0} \bigcup_{s = j\pi\delta 2^{-n}} \left(A_{s,\, u 2^{-n/2} n^{1/2}}^{u^{-1} 2^{-n/2} n^{-1/2},\, \delta u^2 n}\right)^c \right) \lesssim \exp(-\tilde c_2 u^2 n_0) \]
for sufficiently large $u$.

Let $r \in {(0,1)}$, and pick $n_0 \in \bbN$ such that $r \asymp 2^{-n_0/2} n_0^{1/2}$. Then the estimate above reads
\[ \bbP( \bigcup \dots ) \lesssim r^{\tilde c_2 u^2} . \]
We claim that
\[ \bigcap_{n \ge n_0} \bigcap_{s = j\pi\delta 2^{-n}} A_{s,\, u 2^{-n/2} n^{1/2}}^{u^{-1} 2^{-n/2} n^{-1/2},\, \delta u^2 n} \subseteq E_{r,u,[0,1]} . \]
Suppose $\abs{\eta(t)-\eta(s)} \le ur$. Find $n \ge n_0$ such that $\frac{\abs{\eta(t)-\eta(s)}}{u 2^{-n/2} n^{1/2}} \in [4,8]$.
Note that on the event $A_{s,\, u 2^{-n/2} n^{1/2}}^{u^{-1} 2^{-n/2} n^{-1/2},\, \delta u^2 n}$, we have $\tau_{s,u 2^{-n/2}n^{1/2}} > \pi\delta 2^{-n}$ and hence $\diam(\eta[s,s+\pi\delta 2^{-n}]) \le 2u 2^{-n/2}n^{1/2}$ since $\eta$ is parametrised by area. Therefore, by our choice of $n$ we must have $s \le j\pi\delta 2^{-n} < (j+1)\pi\delta 2^{-n} \le t$ for some $j$. In particular, $\eta[s,t] \supseteq \eta[j\pi\delta 2^{-n} , (j+1)\pi\delta 2^{-n}]$ contains $\delta u^2 n \asymp \delta u^2 \log(u\abs{\eta(t)-\eta(s)}^{-1})$ disjoint balls of radius $u^{-1} 2^{-n/2} n^{-1/2} \asymp u^{-2}\abs{\eta(t)-\eta(s)} /\log(u\abs{\eta(t)-\eta(s)}^{-1})$.

\section{Lower bounds for Markov processes}
\label{se:markov}

We prove lower bounds on the regularity (corresponding to the lower bounds in \cref{thm:main_lil,thm:main_moc,thm:main_var}) for Markov processes satisfying a uniform ellipticity condition. The arguments are elementary but they illustrate well the general idea on how to obtain lower bounds, and we have not seen them written out in earlier literature, except that (even functional versions of) laws of the iterated logarithms and rates of escape of Markov processes have been proved in \cite{bk-markov-lil}. Our arguments on SLE follow the same idea, but will be more technical since SLE is not exactly a Markov process, and we need to work with its domain Markov property.

In the following, let $X =(X_t)_{t\geq 0}$ be a Markov process on a metric space $(E,d)$, and let $\bbP^x$ denote the law of the Markov process started at $X_0=x$. In particular, we assume the Markov property $\bbP^x(X_{t+s} \in A \mid \cF_t) = \bbP^{X_t}(X_s \in A)$ for every $x$. We suppose the following uniform bounds on the transition probabilities: There exist constants $d_{\mathrm{w}} > 1$, $T>0$, $r_0 > 0$ such that
\begin{equation}\label{eq:markov_tail_upper}
    \bbP^x(d(X_t,x) > r) \le c_1 \exp\left(-c_2 \left(\frac{r}{t^{1/d_{\mathrm{w}}}}\right)^{d_{\mathrm{w}}/(d_{\mathrm{w}}-1)} \right)
\end{equation}
for all $r>0$, $0<t\le T$, and
\begin{equation}\label{eq:markov_tail_lower}
    \bbP^x(d(X_t,x) > r) \ge c_3 \exp\left(-c_4 \left(\frac{r}{t^{1/d_{\mathrm{w}}}}\right)^{d_{\mathrm{w}}/(d_{\mathrm{w}}-1)} \right)
\end{equation}
for $r \le r_0$, $0<t<r^{d_{\mathrm{w}}}$. The exponent $d_{\mathrm{w}}$ is usually called the walk dimension.
For instance, this is satisfied for diffusions on $\bbR^d$ with uniformly elliptic generator (for which $d_{\mathrm{w}}=2$). Other typical examples are Brownian motions on fractals (cf.\@ \cite{bk-markov-lil} and references therein) or Liouville Brownian motion (cf.\@ \cite{akm-lbm83}).

For these Markov processes, the analogues of \cref{thm:main_lil,thm:main_moc,thm:main_var} hold with $d = d_{\mathrm{w}}$, only that we do not prove 0-1 laws for the limits but only deterministic upper and lower bounds for them (but see e.g.\@ \cite{bk-markov-lil} for a type of 0-1 law). We only need to prove the lower bounds since the matching upper bounds follow already from the results in \cref{sec:variation-upper-general}. The upper bounds hold for general stochastic processes and the Markov property is not needed.

\begin{proposition}
Under assumption \eqref{eq:markov_tail_lower} there exist positive deterministic constants $a_1,a_2,a_3 > 0$ such that the following is true.
\begin{enumerate}[label=(\roman*)]
    \item\label{it:markov_var} 
    Variation: For any bounded interval $I \subseteq \bbR^+$, almost surely
    \[ \inf_{\delta > 0} \sup_{\abs{t_{i+1}-t_i}<\delta} \sum_i \psi(d(X_{t_{i+1}},X_{t_i})) \ge a_3\abs{I} \]
    with $\psi(x) = x^{d_{\mathrm{w}}}(\logp\logp\frac{1}{x})^{-(d_{\mathrm{w}}-1)}$, and the supremum is taken over finite sequences $t_0 < ... < t_r$ with $t_i \in I$ and $\abs{t_{i+1}-t_i}<\delta$.
    \item\label{it:markov_moc} 
    Modulus of continuity:
    For any non-trivial interval $I \subseteq \bbR^+$, almost surely
    \[ \inf_{\delta > 0} \sup_{s,t \in I, \abs{t-s}<\delta} \frac{d(X_t,X_s)}{\abs{t-s}^{1/d_{\mathrm{w}}}(\log \abs{t-s}^{-1})^{1-1/d_{\mathrm{w}}}} \ge a_2 . \]
    \item\label{it:markov_lil} 
    Law of the iterated logarithm: 
    For any $t \ge 0$, almost surely
    \[ \limsup_{t \downarrow 0} \frac{d(X_{t_0+t},X_{t_0})}{t^{1/d_{\mathrm{w}}}(\log\log t^{-1})^{1-1/d_{\mathrm{w}}}} \ge a_1 . \]
\end{enumerate}
\end{proposition}

\begin{proof}
\ref{it:markov_moc}: By the Markov property, there is no loss of generality assuming $I = [0,1]$. For $\varepsilon > 0$ and $k = 1,...,\lfloor\varepsilon^{-1}\rfloor$, we define the event
\[ A_{\varepsilon,k} = \{ d(X_{k\varepsilon},X_{(k-1)\varepsilon}) \ge a_0\varepsilon^{1/d_{\mathrm{w}}}(\log\varepsilon^{-1})^{1-1/d_{\mathrm{w}}} \} \in \cF_{k\varepsilon} \]
with $a_0 > 0$ a constant whose value will be decided upon later. By \eqref{eq:markov_tail_lower} and the Markov property, we have
\[ \bbP( A_{\varepsilon,k} \mid \cF_{(k-1)\varepsilon} ) \ge c_3 \exp\left( -c_4 a_0^{d_{\mathrm{w}}/(d_{\mathrm{w}}-1)} (\log\varepsilon^{-1}) \right) = c_3\varepsilon^{1/2} \]
for a suitable choice of $a_0$. Applying this estimate iteratively yields
\[ \bbP( A_{\varepsilon,1}^c \cap ... \cap A_{\varepsilon,\varepsilon^{-1}}^c ) \le \left( 1-c_3\varepsilon^{1/2} \right)^{\varepsilon^{-1}} \le \exp\left( -c_3\varepsilon^{-1/2} \right) \to 0 \quad \text{as }\varepsilon \downarrow 0 . \]
This shows that for any fixed $\delta > 0$, the event
\[ \sup_{s,t \in I, \abs{t-s}<\delta} \frac{d(X_t,X_s)}{\abs{t-s}^{1/d_{\mathrm{w}}}(\log \abs{t-s}^{-1})^{1-1/d_{\mathrm{w}}}} \ge a_0 \]
must occur with probability $1$. The claim follows.

\ref{it:markov_lil}: By the Markov property, there is no loss of generality assuming $t_0 = 0$. Define a sequence of events
\[ A_k = \{ d(X_{e^{-k}},X_{e^{-(k+1)}}) \ge a_0 e^{-k/d_{\mathrm{w}}}(\log k)^{1-1/d_{\mathrm{w}}} \} \in \cF_{e^{-k}} \]
with a constant $a_0 > 0$ whose value will be decided upon later. We show that almost surely $A_k$ occur infinitely often. This implies the claim since on the event $A_k$ we have
\[ d(X_{e^{-k}},X_0)+d(X_{e^{-(k+1)}},X_0) \ge a_0 e^{-k/d_{\mathrm{w}}}(\log k)^{1-1/d_{\mathrm{w}}} \]
and hence for either $t=e^{-k}$ or $t=e^{-(k+1)}$ we have
\[ d(X_t,X_0) \ge \frac{a_0}{2} t^{1/d_{\mathrm{w}}}(\log\log t^{-1})^{1-1/d_{\mathrm{w}}} . \]

By \eqref{eq:markov_tail_lower} and the Markov property, we have
\[ \bbP( A_k \mid \cF_{e^{-(k+1)}} ) \ge c_3 \exp\left( -c_4 a_0^{d_{\mathrm{w}}/(d_{\mathrm{w}}-1)}(1-e^{-1})^{1/(d_{\mathrm{w}}-1)} (\log k) \right) = c_3 k^{-1} \]
for a suitable choice of $a_0$. Applying this estimate iteratively yields
\[ \bbP( A_{k}^c \cap ... \cap A_{k'}^c ) \le (1-c_3 k^{-1})\dots(1-c_3 (k')^{-1}) \le \exp(-c_3(k^{-1}+...+(k')^{-1})) \to 0 \quad \text{as }k' \to \infty \]
and hence
\[ \bbP\left( \bigcup_{k' > k} A_{k'} \right) = 1 . \]
Since this holds for any $k$, the claim follows.

\ref{it:markov_var}: This follows from \ref{it:markov_lil} by a general result which we will state as \cref{pr:var_lil} in the next section.
\end{proof}

\section{Lower bounds for SLE}
\label{sec:sle-lower-bounds}

In this section we conclude the lower bounds in our main results (\cref{thm:main_lil,thm:main_moc,thm:main_var}). We begin in \cref{sec:variation-lower-general} by reviewing a general argument saying that the lower bound for $\psi$-variation follows from the lower bound in the law of the iterated logarithm. In  \cref{sec:diam-lower-nonspf,sec:diam-lower-spf}, which constitute the main part of this section, we prove the lower bound in \cref{eq:mink-tail} along with some conditional variants of this estimate. Finally, we use these in \cref{se:lowerbounds_pf} to conclude the lower bounds for the modulus of continuity and the law of the iterated logarithm.

\subsection{Law of the iterated logarithm implies variation lower bound}
\label{sec:variation-lower-general}

We review an argument for general processes that says that a ``lower'' law of the iterated logarithm implies a lower bound on the variation regularity. We follow \cite[Section 13.9]{fv-rp-book} where the argument is spelled out for Brownian motion (implying Taylor's variation \cite{tay-bm-variation}).

\begin{proposition}\label{pr:var_lil}
Let $(X_t)_{t \in [0,1]}$ be a separable process such that for every fixed $t \in(0,1)$ we almost surely have
\begin{equation}\label{eq:lil_lower}
\limsup_{s \to 0}\frac{\abs{X_{t+s}-X_t}}{\sigma(\abs{s})} > 1
\end{equation}
where $\sigma$ is a (deterministic) self-homeomorphism of $[0,\infty)$.
Then, almost surely, for any $\varepsilon > 0$ there exist disjoint intervals $[t_i,u_i]$ of length at most $\varepsilon$ such that 
\[ \sum_i \sigma^{-1}(\abs{X_{t_i}-X_{u_i}}) > 1-\varepsilon . \]
\end{proposition}

The proposition proves in particular that $X$ cannot have better $\psi$-variation regularity than $\psi = \sigma^{-1}$, i.e., $X$ has infinite $\wt\psi$-variation if $\wt\psi(x)/\psi(x)\rta \infty$ as $x\downarrow 0$.

\begin{proof}
Let $E$ be the set of $t \in [0,1]$ where \eqref{eq:lil_lower} holds. By Fubini's theorem, we almost surely have $\abs{E} = 1$. By definition, for each $t \in E$, there exist arbitrarily small $s$ such that $\abs{X_{t+s}-X_t} > \sigma(\abs{s})$. The collection of all such intervals $[t,t+s]$ form a Vitali cover of $E$ in the sense of \cite[Lemma 13.68]{fv-rp-book}. In particular, there exist disjoint invervals $[t_i,u_i]$ with $\abs{E \setminus \bigcup_i [t_i,u_i]} < \varepsilon$ and $\abs{X_{t_i}-X_{u_i}} > \sigma(\abs{t_i-u_i})$, so
\[ \sum_i \sigma^{-1}(\abs{X_{t_i}-X_{u_i}}) > \sum_i \abs{t_i-u_i} >1-\varepsilon . \]
By picking in the Vitali cover only intervals of length at most $\varepsilon$, we get intervals $[t_i,u_i]$ of length at most $\varepsilon$.
\end{proof}

\subsection{Diameter lower bound for non-space-filling SLE}
\label{sec:diam-lower-nonspf}

In this section we prove the matching lower bound to the result in \cref{pr:diam_tail_upper}. Our main result is the following proposition, together with a ``conditional'' variant of it, see \cref{pr:cont_ltail_cond} below. As before we let $\tau_r=\inf\{t\geq 0\,:\, |\eta(t)|=r \}$ denote the hitting time of radius $r$.

\begin{proposition}\label{pr:cont_tail_lower}
Let $\eta$ be a whole-plane \slek{}$(2)$ from $0$ to $\infty$, $\kappa \le 8$. For some $\wt c_2 > 0$ we have
\[ \bbP( \Cont(\eta[0,\tau_r]) < \ell ) \ge \exp(-\wt c_2 \ell^{-1/(d-1)} r^{d/(d-1)}) \]
for any $r,\ell>0$ with $\ell \le r^d$.
\end{proposition}

This is the matching lower bound to the upper bound in \cref{pr:diam_tail_upper}. Note that our upper and lower bounds match except for a different constant $\wt c_2$.

We remark that the proposition can be equivalently stated as a tail lower bound on the increment of $\eta$ when parametrised by Minkowski content. Namely, we have
\[ \bbP( \diam(\eta[0,1]) > r ) \ge \exp(-c r^{d/(d-1)}) \]
for $r \ge 1$.

Similarly as in \cref{se:sle_upper}, we can use the scaling property to reduce the statement of \cref{pr:cont_tail_lower} to the following special case.

\begin{lemma}\label{le:cont_tail_lower1}
Let $\eta$ be a whole-plane \slek{}$(2)$ from $0$ to $\infty$. Then for some $c_1,c_2 > 0$ we have
\[ \bbP( \Cont(\eta[0,\tau_r]) < c_1 r ) \ge \exp(-c_2 r) \]
for any $r \ge 1$.
\end{lemma}

\begin{proof}[Proof of \cref{pr:cont_tail_lower} given \cref{le:cont_tail_lower1}]
Let $\lambda = c_1^{1/(d-1)}\ell^{-1/(d-1)}r^{1/(d-1)}$. By the scaling property (cf.\@ \cref{se:prelim_sle}), $\wt\eta = \lambda\eta$ is also a whole-plane \slek{}$(2)$, and $\Cont(\wt\eta[0,\wt\tau_{\lambda r}]) = \lambda^d \Cont(\eta[0,\tau_r])$. Hence the desired probability is equal to
\[ \bbP( \Cont(\wt\eta[0,\wt\tau_{\lambda r}]) < \lambda^d \ell ) . \]
We have chosen $\lambda$ such that $\wt r \defeq \lambda r = c_1^{1/(d-1)}\ell^{-1/(d-1)}r^{d/(d-1)}$ and $\lambda^d \ell = c_1 \wt r$. Therefore, by \cref{le:cont_tail_lower1}, the probability is at least
\[ \exp(-c_2 \wt r) = \exp(-\wt c_2 \ell^{-1/(d-1)}r^{d/(d-1)}) . \]
\end{proof}

The heuristic idea of \cref{le:cont_tail_lower1} is straightforward, but requires some care to implement. We would like to show that when $\eta$ crosses from radius $r$ to $r+1$, it has some chance $p>0$ to do so while creating at most $c_1$ Minkowski content. If $p>0$ is uniform conditionally on $\eta[0,\tau_r]$, the lemma would follow. However, it is difficult to control the Minkowski content near the tip of $\eta[0,\tau_r]$ due to its fractal nature, therefore we want to keep only those curves that have a sufficiently ``nice'' tip when hitting radius $r$. Below we will implement this idea.

Let $(D,a)$ be as in \eqref{eq:Da}, and $u\in\partial D\setminus\{a\}$. The point $u$ will play the role of a force point with weight $2$. (We can allow more general force points but to keep notation as simple as possible, we restrict to this case.) 
For $p_{\mathrm{N}},r_{\mathrm{N}},c_{\mathrm{N}} > 0$  
we say that $(D,a,u)$ is $(p_{\mathrm{N}},r_{\mathrm{N}},c_{\mathrm{N}})$-nice if
\[ \abs{f(u)-1} \ge 2r_{\mathrm{N}} \]
and
\[ \nu_{\bbD;1\to 0}(\Cont(f^{-1}(\eta[0,\sigma_{r_{\mathrm{N}}}])) > c_{\mathrm{N}}) < p_{\mathrm{N}} \]
where $\nu_{\bbD;1\to 0}$ denotes radial \slek{} (without force point) and $\sigma_{r_{\mathrm{N}}}$ the exit time of $B(1,r_{\mathrm{N}}) \cap \barD$.

\begin{proposition}\label{pr:nice_tip_iter}
There exist finite and positive constants $p_{\mathrm{N}},r_{\mathrm{N}},c_{\mathrm{N}},p,c$ such that the following is true. Let $(D,a,u)$ with $\abs{a}\ge 1$ be $(p_{\mathrm{N}},r_{\mathrm{N}},c_{\mathrm{N}})$-nice. Then
\[ 
\nu_{D;a\to \infty;u} \left( \begin{array}{cc}
\Cont(\eta[0,\tau_{\abs{a}+1}]) < c \quad\text{and} \\
D \setminus \eta[0,\tau_{\abs{a}+1}] \text{ is } (p_{\mathrm{N}},r_{\mathrm{N}},c_{\mathrm{N}})\text{-nice} 
\end{array} \right)
\ge p .
\]
\end{proposition}

\begin{lemma}\label{le:nice_tip}
There exist finite positive constants $p_{\mathrm{N}},r_{\mathrm{N}},c_{\mathrm{N}},p,\varepsilon,c_1$ with $\varepsilon < r_{\mathrm{N}}/2$ such that the following is true. 
Let $(D,a,u)$ with $\abs{a}\ge 1$ be $(p_{\mathrm{N}},r_{\mathrm{N}},c_{\mathrm{N}})$-nice, and let $f\colon D \to \bbD$ the corresponding conformal map. Let $A = f(\hatC \setminus B(0,\abs{a}+1))$, and $\tau_A$ the hitting time, and $\eta$ a radial \slek{}$(2)$ in $\bbD$ from $1$ to $f(\infty)$ with force point $f(u)$. Then the following event has probability at least $p$.
\begin{enumerate}[label=(\roman*)]
    \item\label{it:follow} $\norm{\gamma-\eta}_{\infty;[0,\tau_A]} < \varepsilon$ where $\gamma$ denotes the straight line from $1$ to $1/64$, and $\gamma$ and $\eta$ are parametrised by capacity relative to $-1$ 
    (see \cref{se:prelim} for the definition of relative capacity).
    \item\label{it:cont} $\Cont(\eta[0,\tau_A]) \le c_1$.
    \item\label{it:nice_base} $\Cont(f^{-1}(\eta[0,\sigma_{r_{\mathrm{N}}}])) \le c_{\mathrm{N}}$.
    \item\label{it:nice_tip} $D \setminus f^{-1}(\eta[0,\tau_A])$ is $(p_{\mathrm{N}},r_{\mathrm{N}},c_{\mathrm{N}})$-nice.
\end{enumerate}
\end{lemma}

This lemma will be proved in several steps where we successively pick the constants that we look for. First, we recall that for any $\varepsilon > 0$, the probability of \ref{it:follow} can be bounded from below by some $p_\varepsilon > 0$ (cf.\@ \cref{le:support} below). The lemma then follows if we can guarantee that the probability of any of the other conditions failing is at most $p_\varepsilon/2$. We will pick all the required constants in a way such that this is true.

We begin by making a few general comments about the conformal maps $f\colon D \to \bbD$ corresponding to domains as in \eqref{eq:Da}. Consider $D^* = \{ 1/z \mid z \in D\} \subseteq \bbC$ and $g(z) = f(1/z)$. This is the conformal map from $D^*$ to $\bbD$ with $g(1/z_D) = 0$ and $g(1/a) = 1$. Note that $\dist(1/z_D, \partial D^*) = \abs{1/z_D-1/a} = \frac{1}{\abs{a}}-\frac{1}{\abs{a}+2}$. Hence, the conformal radius of $1/z_D$ in $D^*$ is between $\frac{2}{\abs{a}(\abs{a}+2)}$ and $\frac{8}{\abs{a}(\abs{a}+2)}$ (cf.\@ \cref{se:prelim_conformal}). It follows that $\abs{f'(z_D)} \in [\frac{\abs{a}}{8(\abs{a}+2)}, \frac{\abs{a}}{2(\abs{a}+2)}]$.

\begin{lemma}\label{le:dist_1}
There exists $r_0 > 0$ such that for any $(D,a)$ as in \eqref{eq:Da} we have $f^{-1}(B(1,r_0) \cap \bbD) \subseteq B(0,\abs{a}+1)$.
\end{lemma}

\begin{proof}
When $\abs{a}$ is not too large, we can use Koebe's distortion theorem to argue that even $f^{-1}(\bbD \setminus B(0,1-r_0)) \subseteq B(0,\abs{a}+1)$. Indeed, considering $D^* = \{ 1/z \mid z \in D\} \subseteq \bbC$ and the conformal map $z \mapsto f(1/z)$ from $D^*$ to $\bbD$, we see that its derivative is comparable everywhere on $B(0,\frac{1}{\abs{a}+1})$. In particular, it cannot map any of these points anywhere close to $\partial\bbD$.

When $\abs{a}$ is large, let $\wt z \in \partial B(0,\abs{a}+1/3)$ be the closest point to $a$. The argument above shows that $f(\wt z)$ cannot be too close to $\partial\bbD$. Let $V$ be the union of $B(\wt z,2/3) \cap D$ with all points that it separates from $\infty$ in $D$. By \cref{le:ghm} there exists a universal constant $r_0$ (independent of $D$) and a path from $f(\wt z)$ to $1$ (dependent of $D$) whose $r_0$-neighbourhood is contained in $f(V)$. Since $V \subseteq B(0,\abs{a}+1)$, the claim follows.
\end{proof}

\begin{lemma}\label{le:inf_loc}
Given $r_0 > 0$, there exists $R>0$ such that for any $(D,a)$ as in \eqref{eq:Da} with $\abs{a} > R$ we have $f(\infty) \notin B(0,1-r_0/2)$.
\end{lemma}

\begin{proof}
Consider the conformal map $z \mapsto 1/f^{-1}(z)$ from $\bbD$ to $D^* = \{ 1/z \mid z \in D\} \subseteq \bbC$. We saw right above the statement of Lemma \ref{le:dist_1} that the conformal radius of $1/z_D$ in $D^*$ is comparable to $\abs{a}^{-2}$. 
Koebe's distortion theorem implies that the derivative of $z \mapsto 1/f^{-1}(z)$ is comparable to $\abs{a}^{-2}$ on $B(0,1-r_0/2)$ (up to some factor depending on $r_0$). But since the distance from $1/z_D = 1/f^{-1}(0)$ to $0 = 1/\infty$ is comparable to $\abs{a}^{-1} \gg \abs{a}^{-2}$, we cannot have $\infty \in f^{-1}(B(0,1-r_0/2))$.
\end{proof}

\begin{lemma}\label{le:inf_dist}
For any $R>0$ there exists $\delta>0$ such that the following is true. Let $(D,a)$ be as in \eqref{eq:Da} with $\abs{a} \in [1,R]$, and $A = f(\hatC \setminus B(0,\abs{a}+1))$. Then $\dist(f(\infty),\partial A) \ge \delta$.
\end{lemma}

\begin{proof}
Consider the conformal map $z \mapsto f(1/z)$ from $D^* = \{ 1/z \mid z \in D\}$ to $\bbD$. 
By the discussion right above the statement of Lemma \ref{le:dist_1}, its derivative at $1/z_D$ is comparable to $\abs{a}(\abs{a}+2) \ge 1$. Since $\abs{1/z_D} \asymp \dist(1/z_D,\partial D^*)\asymp 1$ (where the implicit constants depend on $R$),
Koebe's distortion theorem gives that also the derivative at $0$ is bounded from below by a constant depending on $R$. Since $\partial A = \{ f(1/z) \mid z \in \partial B(0,1/(\abs{a}+1)) \}$, the claim follows from Koebe's $1/4$-theorem.
\end{proof}

\begin{corollary}\label{co:target_change}
There exist finite constants $c' > 0$ and $\varepsilon > 0$ such that the following is true.
Let $(D,a)$ be as in \eqref{eq:Da} with $\abs{a} \ge 1$, and $A = f(\hatC \setminus B(0,\abs{a}+1))$. Let $\gamma$ be the straight line from $1$ to $0$. Consider either the two \slek{} measures $\nu_{\bbD;1\to 0}$ and $\nu_{\bbD;1\to f(\infty)}$, or the two \slek{}$(2)$ measures $\nu_{\bbD;1\to 0;u}$ and $\nu_{\bbD;1\to f(\infty);u}$ with the same force point $u \in \barD$. Then, on the event $\{ \norm{\gamma-\eta}_{\infty;[0,\tau_A]} < \varepsilon \}$, the law of $\eta\big|_{[0,\tau_A]}$ under the two measures are absolutely continuous with density bounded between $[1/c',c']$.
\end{corollary}

\begin{proof}
By \cref{le:inf_loc,le:inf_dist}, if $\varepsilon$ is sufficiently small, the point $f(\infty)$ has distance at least $2\varepsilon$ from $\gamma[0,\tau_A]$. Moreover, we have $B(0,1/64) \subseteq A$ due to Koebe's $1/4$-theorem, so $0$ is also bounded away from $\gamma[0,\tau_A]$. Therefore, the claim follows from \cref{le:sle_abs_cont}.
\end{proof}

\begin{lemma}\label{le:no_return}
Given $r_0 > 0$, there exists $r_1 > 0$ such that the following is true. Let $\alpha\colon [0,1] \to \overline{\bbD}$ be a curve with $\alpha(0) = 1$ and $\alpha \subseteq \{ \abs{z} \ge 1/64, \Re(z)>0 \}$. Suppose also that $\alpha$ does not leave $B(0,1-r_0/4)$ after entering $B(0,1-r_0/2)$, and that $\alpha(1)$ is connected to $0$ in $B(0,1-r_0) \setminus \alpha$. Let $D_\alpha$ denote the connected component of $\bbD\setminus\alpha$ containing $0$ (so in particular $D_\alpha=\bbD\setminus\alpha$ if $\bbD\setminus\alpha$ is connected), and let $f_\alpha\colon D_\alpha \to \bbD$ denote the conformal map with $f_\alpha(0) = 0$ and $f_\alpha(\alpha(1)) = 1$. 
Then $f_\alpha^{-1}(B(1,r_1) \cap \bbD) \subseteq B(0,1-r_0/4)$.
\end{lemma}

\begin{figure}[h]
	\centering
	\includegraphics[width=0.8\textwidth]{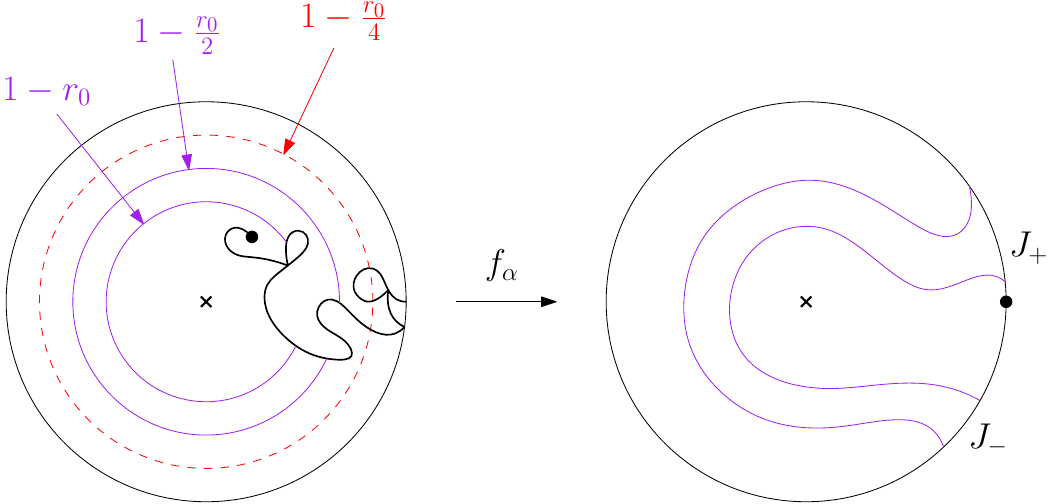}
	\caption{The setup and proof of \cref{le:no_return}.}
\end{figure}

\begin{proof}
Notice that $\partial B(0,1-r_0) \setminus \alpha$ may have several connected components, each of which is an arc of $\partial B(0,1-r_0)$.
Let $C_{r_0}$ denote the longest arc of $\partial B(0,1-r_0) \setminus \alpha$ (i.e. the one to the left in the figure, or, equivalently, the unique arc crossing the negative real axis). By our assumption, it does not separate $\alpha(1)$ from $0$ in $D_\alpha$. Therefore $f_\alpha(C_{r_0})$ does not separate $1$ from $0$.

Next, let $C_{r_0/2}$ denote the longest arc of $\partial B(0,1-r_0/2) \setminus \alpha$. Its image $f_\alpha(C_{r_0/2})$ lies in the component of $\bbD \setminus f_\alpha(C_{r_0})$ that does not contain $0$. Let $J_+$, $J_-$ denote the two arcs of $\partial\bbD$ that lie between $f_\alpha(C_{r_0})$ and $f_\alpha(C_{r_0/2})$. By considering a Brownian motion in the domain $D_\alpha$ starting from $0$, we see that the harmonic measures of both $J_\pm$ seen from $0$, and therefore their lengths, are at least some constant depending on $r_0$. (More precisely, these harmonic measures are at least the probability of Brownian motion staying inside $B(0,1/64) \cup \{\Re(z)<0\}$ before entering the annulus $\{ 1-r_0 < \abs{z} < 1-r_0/2 \}$ and then making a clockwise (resp., counterclockwise) turn inside the annulus.)

We claim that $f_\alpha^{-1}(J_\pm) \subseteq B(0,1-r_0/4)$. The result will then follow from the fact that $f_\alpha^{-1}(z)$ is obtained from integrating $f_\alpha^{-1}$ on $\partial\bbD$ against the harmonic measure seen from $z$.

By symmetry, it suffices to show the claim for $J_+$. Let $z_1,z_2$ denote the endpoints of $f_\alpha^{-1}(J_+)$. By our assumption, any sub-curve of $\alpha$ from $z_1$ to $z_2$ stays inside $B(0,1-r_0/4)$. Suppose $f_\alpha^{-1}(J_+)$ contains some point $z' \notin B(0,1-r_0/4)$. Pick a simple path $P'$ in $D_\alpha$ that begins with a straight line from $-1$ to $-(1-\frac{3}{4}r_0)$, then stays between $C_{r_0}$ and $C_{r_0/2}$, and ends at $z'$. As a consequence of the Jordan curve theorem, $P'$ separates $B(0,1-r_0/4)$ into at least two components, and the two halfs of $C_{r_0/2}$ lie in different components. Moreover, $C_{r_0}$ is in the same component as the lower half of $C_{r_0/2}$. Therefore $z_1$ and $z_2$ (which are the upper endpoints of $C_{r_0}$ and $C_{r_0/2}$ respectively) lie in different components of $B(0,1-r_0/4) \setminus P'$. But since $P' \subseteq D_\alpha$, any sub-curve of $\alpha$ from $z_1$ to $z_2$ must avoid $P'$ which is impossible while staying inside $B(0,1-r_0/4)$.
\end{proof}

\begin{lemma}\label{le:nice_tip_in_D}
Given $r_0,r_1 > 0$ there exist $r_{\mathrm{N}} > 0$ and $\wt c > 0$ such that the following is true.

Let $(D,a,u)$ be as in \eqref{eq:Da} with $\abs{a} \ge 1$, and $f\colon D \to \bbD$ the corresponding conformal map. Let $\alpha\colon [0,1] \to \overline{\bbD}$ be a curve with $\alpha(0) = 1$, $\alpha(1) \in B(0,1-r_0)$, and $\abs{f^{-1}(\alpha(1))} = \abs{a}+1 > \abs{f^{-1}(\alpha(t))}$ for $t < 1$. Let $D_\alpha$ be the connected component of $\bbD\setminus\alpha$ containing $0$ and suppose $\alpha(1)\in\partial D_\alpha$. Let $f_\alpha\colon D_\alpha \to \bbD$ denote the conformal map with $f_\alpha(0) = 0$ and $f_\alpha(\alpha(1)) = 1$. 
Suppose additionally that $f_\alpha^{-1}(B(1,r_1)\cap \bbD) \subseteq B(0,1-r_0/4)$. If for some $p_{\mathrm{N}}, c_{\mathrm{N}} > 0$
\[ \nu_{\bbD;1\to 0}(\Cont(f_\alpha^{-1}(\eta[0,\sigma_{r_1}])) > \wt c\,c_{\mathrm{N}}) < \wt c\,p_{\mathrm{N}} , \]
then $(D \setminus f^{-1}(\alpha),\ f^{-1}(\alpha(1)),\ u)$ is $(p_{\mathrm{N}},r_{\mathrm{N}},c_{\mathrm{N}})$-nice. 
\end{lemma}

\begin{figure}[h]
	\centering
	\includegraphics[width=0.8\textwidth]{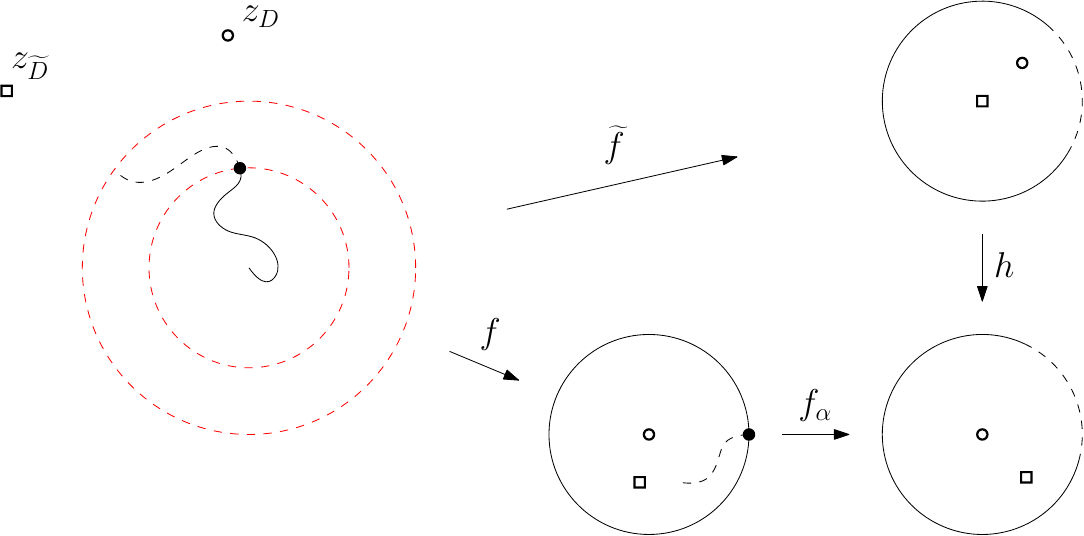}
	\caption{The setup and proof of \cref{le:nice_tip_in_D}.}
\end{figure}

\begin{proof}
Write $\wt D = D \setminus f^{-1}(\alpha)$ and denote by $\wt f\colon \wt D \to \bbD$ the conformal map corresponding to $\wt D$. Let $z_D$ and $z_{\wt D}$ denote the points on $\partial B(0,\abs{a}+2)$ and $\partial B(0,\abs{a}+3)$ closest to $a$ and $f^{-1}(\alpha(1))$, respectively. 

Observe that $\wt f$ and $f$ are related via $\wt f = h^{-1} \circ f_\alpha \circ f$ where $h\colon \bbD \to \bbD$ is the conformal map with $h(\wt f(z_D)) = 0$ and $h(1) = 1$. We claim that the distance $\abs{z_D-z_{\wt D}}$ is bounded from above by a constant depending on $r_0$. Indeed, if we let $R$ be as in \cref{le:inf_loc}, then if $\abs{a} \ge R$, the distance $\abs{z_D-f^{-1}(\alpha(1))}$ is bounded via Koebe's distortion theorem, and hence also $\abs{z_D-z_{\wt D}}$. If $\abs{a} < R$, then the distance is trivially bounded by $2R+5$. We conclude that $\abs{z_D-z_{\wt D}}$ is bounded.

Since $\abs{z_D-z_{\wt D}}$ is bounded from above (by a constant depending on $r_0$), we get that $\wt f(z_D)$ is bounded away from $\partial\bbD$, and $\abs{h'} \asymp 1$ on $\bbD$, with both bounds depending only on $r_0$. 
Consequently there exists $r_{\mathrm{N}} > 0$ such that $h(B(1,2r_{\mathrm{N}}) \cap \bbD) \subseteq B(1,r_1) \cap \bbD$.

To show that $\wt D$ is nice, we need to consider an \slek{} in $\bbD$ from $1$ to $0$, stopped at hitting $\partial B(1,r_{\mathrm{N}})$. However since $\wt f(z_D)$ is bounded away from $\partial\bbD$, this law is absolutely continuous with respect to \slek{} from $1$ to $\wt f(z_D)$, with Radon-Nikodym derivative bounded in some interval $[\wt c, \wt c^{-1}]$ with $\wt c > 0$ depending on $r_0$ and $r_{\mathrm{N}}$ 
(after possibly further decreasing $r_{\mathrm{N}}$; cf.\@ \cref{le:sle_abs_cont}).

The first condition for niceness, i.e. $\abs{\wt f(u)-1} \ge 2r_{\mathrm{N}}$ is guaranteed by $f_{\alpha}^{-1}(h(B(1,2r_{\mathrm{N}}) \cap \bbD)) \subseteq f_\alpha^{-1}(B(1,r_1)\cap \bbD) \subseteq B(0,1-r_0/4)$, which is away from the boundary. The second condition for niceness is satisfied if
\[ \nu_{\bbD;1\to \wt f(z_D)}(\Cont(\wt f^{-1}(\wt\eta[0,\sigma_{r_{\mathrm{N}}}])) > c_{\mathrm{N}}) < \wt c\,p_{\mathrm{N}} . \]
Suppose now that $\wt\eta$ is an \slek{} from $1$ to $\wt f(z_D)$, so that $h\circ\wt\eta$ has the law of an \slek{} $\eta$ from $1$ to $0$. 
By our assumption, with probability at least $1-\wt c\,p_{\mathrm{N}}$ we have
\[ \Cont(f_\alpha^{-1}(\eta[0,\sigma_{r_1}])) \le \wt c\,c_{\mathrm{N}} . \]
Applying Koebe's distortion theorem to the map $z \mapsto 1/f^{-1}(z)$ we see that on the set $B(0,1-r_0/4) \cap f(D \cap B(0,\abs{a}+2))$ the derivative  $\abs{(f^{-1})'}$ is bounded from above by a constant depending on $r_0$.
Therefore, using the assumption $f_\alpha^{-1}(\eta[0,\sigma_{r_1}]) \subseteq f_\alpha^{-1}(B(1,r_1)\cap \bbD) \subseteq B(0,1-r_0/4)$ and the transformation rule for Minkowski content \eqref{eq:cont_transf}, such $\eta$ satisfies
\[ 
\Cont(\wt f^{-1}(\wt\eta[0,\sigma_{r_{\mathrm{N}}}]))
\le \Cont(f^{-1} \circ f_\alpha^{-1}(\eta[0,\sigma_{r_1}])) 
\lesssim \wt c\,c_{\mathrm{N}} 
\]
with a factor depending only on $r_0$. 
The claim follows if $\wt c$ has been picked small enough.
\end{proof}

\begin{lemma}\label{le:support}
Let $\gamma$ be a simple curve in $\bbD \setminus \{0\}$ with $\gamma(0) = 1$, and $T \ge 0$ the capacity of $\gamma$ relative to $-1$ (i.e. the half-plane capacity of the curve after mapping to $\bbH$ as described in \cref{se:prelim}). For any $\varepsilon > 0$ there exists $p_\varepsilon > 0$ such that
\[ \nu_{\bbD;1\to 0}( \norm{\gamma-\eta}_{\infty;[0,T]} < \varepsilon ) \ge p_\varepsilon \]
where $\gamma$ and $\eta$ are parametrised by capacity.
\end{lemma}

\begin{proof}
For chordal SLE from $1$ to $-1$ in $\bbD$, this is \cite[Proposition 1.4]{ty-support}. The result transfers to radial SLE by absolute continuity (cf.\@ \cite{sw-sle-coordinate-change}). 
\end{proof}

\begin{lemma}\label{le:next_tip}
Let $\gamma$ be the straight line from $1$ to $0$. Let $r_0,r_1,r_{\mathrm{N}} > 0$ be as in \cref{le:dist_1,le:no_return,le:nice_tip_in_D}. For any $\varepsilon \in {(0,r_0/4]}$ and $p_{\mathrm{N}} > 0$ there exists $c_{\mathrm{N}} > 0$ such that the following is true.

Let $(D,a,u)$ be as in \eqref{eq:Da} with $\abs{a} \ge 1$, and $A = f(\hatC \setminus B(0,\abs{a}+1))$. Then
\[
\nu_{\bbD;1\to 0}\left(\begin{array}{cc}
\norm{\gamma-\eta}_{\infty;[0,\tau_A]} < \varepsilon
\quad\text{but the domain} \\
D \setminus f^{-1}(\eta[0,\tau_A]) \text{ is not } (p_{\mathrm{N}}, r_{\mathrm{N}}, c_{\mathrm{N}})\text{-nice}
\end{array}\right) 
< p_{\mathrm{N}} 
\]
where $\tau_A$ denotes the hitting time of $A$.
\end{lemma}

\begin{proof}
We have $B(0,1/64) \subseteq A$ and $A \cap B(1,r_0) = \varnothing$ due to Koebe's $1/4$-theorem and \cref{le:dist_1}. Since the Minkowski content of radial \slek{} (stopped before entering $B(0,1/128)$) is almost surely finite, we can find for any $p_2 > 0$ some $c_2 > 0$ such that
\[ \nu_{\bbD;1\to 0}(\Cont(\eta[0,\tau_{\partial B(0,1/128)}]) > c_2) < p_2 . \]
This gives
\[
\nu_{\bbD;1\to 0}\left( \nu_{\bbD;1\to 0}(\Cont(\eta[0,\tau_{\partial B(0,1/128)}]) > c_2 \mid \eta[0,\tau_A] ) > p_2^{1/2} \right) < p_2^{1/2}
\]
by Markov's inequality since
\begin{multline*}
\nu_{\bbD;1\to 0}\left( \nu_{\bbD;1\to 0}(\Cont(\eta[0,\tau_{\partial B(0,1/128)}]) > c_2 \mid \eta[0,\tau_A] ) \right) \\
= \nu_{\bbD;1\to 0}(\Cont(\eta[0,\tau_{\partial B(0,1/128)}]) > c_2) 
< p_2 .
\end{multline*} 

Now, suppose $\norm{\gamma-\eta}_{\infty;[0,\tau_A]} < \varepsilon$. Apply the above with $p_2 = \wt c^2 p_{\mathrm{N}}^2$, and let $c_{\mathrm{N}} = \wt c^{-1}c_2$. We claim that if
\[ \nu_{\bbD;1\to 0}(\Cont(\eta[\tau_A,\tau_{\partial B(0,1/128)}]) > c_2 \mid \eta[0,\tau_A] ) < p_2^{1/2} , \]
then $D \setminus f^{-1}(\eta[0,\tau_A])$ is $(p_{\mathrm{N}},r_{\mathrm{N}},c_{\mathrm{N}})$-nice with our choices of $p_{\mathrm{N}},c_{\mathrm{N}}$.

Indeed, conditionally on $\eta[0,\tau_A]$, the curve $\wt\eta = f_{\tau_A} \circ \eta\big|_{[\tau_A,\infty)}$ is an independent \slek{} in $\bbD$. Since the capacity of $\eta[0,\tau_{\partial B(0,1/128)}]$ is much larger than of $\eta[0,\tau_A]$, we must have $f_{\tau_A}^{-1}(\wt\eta[0,\sigma_{r_1}]) \subseteq \eta[\tau_A,\tau_{\partial B(0,1/128)}]$ (there is no loss of generality assuming $r_1$ is small). Moreover, by \cref{le:no_return} (the conditions are satisfied due to $\norm{\gamma-\eta}_{\infty;[0,\tau_A]} < \varepsilon < r_0/4$), 
we have $f_{\tau_A}^{-1}(B(1,r_1)\cap \bbD) \subseteq B(0,1-r_0/4)$. Hence, by \cref{le:nice_tip_in_D}, $\wt D$ is $(p_{\mathrm{N}},r_{\mathrm{N}},c_{\mathrm{N}})$-nice.
\end{proof}

\begin{lemma}\label{le:initial_nice_tip}
Let $r_{\mathrm{N}}>0$ be as in \cref{le:next_tip}. For any $p_{\mathrm{N}}>0$ there exist $c_{\mathrm{N}} > 0$ and $p_1 > 0$ such that the following is true. Let $(D,a,u)$ be as in \eqref{eq:Da} with $\abs{a} \ge 1$. Then
\[ 
\nu_{D;a\to\infty;u}\left(
\text{the domain } 
\hatC \setminus \eta[0,\tau_{\abs{a}+1}] \text{ is } (p_{\mathrm{N}}, r_{\mathrm{N}}, c_{\mathrm{N}})\text{-nice}
\right)
\ge p_1 .
\]
In particular,
\[
\nu_{\hatC;0\to\infty;0}\left(
\text{the domain } 
\hatC \setminus \eta[0,\tau_R] \text{ is } (p_{\mathrm{N}}, r_{\mathrm{N}}, c_{\mathrm{N}})\text{-nice}
\right)
\ge p_1
\]
for $R \ge 2$.
\end{lemma}

\begin{proof}
The second assertion follows from the first assertion by considering $D = \hatC\setminus\eta[0,\tau_{R-1}]$, so in the remainder of the proof we will focus only on proving the first assertion.

Let us suppose first that $\abs{f(u)-1} \ge \delta$ for some $\delta > 0$. In that case, the claim follows from \cref{le:support,le:next_tip} and the absolute continuity between SLE variants, cf.\@ \cref{le:sle_abs_cont,co:target_change} (note that $\wt\eta = f\circ\eta$ is an \slek{}$(2)$ from $1$ to $f(\infty)$ with force point at $f(u)$). 

In case $f(u)$ is close to $1$, we do not have uniform control over the density between the SLE variants. But picking a small time $t>0$ there exists some $\wt p_1 > 0$ and $\wt\delta > 0$ such that $\abs{f_t(f(u))-1} \ge \wt\delta$ with probability at least $\wt p_1$ (independent of $f(u)$). This is because $t \mapsto \arg(f_t(f(u)))$ is a radial Bessel process of positive index started at $\arg(f(u))$ (cf.\@ \cite[Section 2.1]{zhan-hoelder}), and by the monotonicity of Bessel processes in the starting point it suffices to compare to the case when it starts at $0$. But a Bessel process stopped at a deterministic time is almost surely positive.

This allows us to consider $\wt\eta^{(t)} = f_t \circ \wt\eta$, stopped at hitting $A^{(t)} = f_t(A)$. On the event $\norm{\gamma-\wt\eta^{(t)}}_{\infty;[0,\tau_{A^{(t)}}]} < \wt\delta/2$, we now do have bounded density between the SLE variants (with a bound depending on $\wt\delta$). In order to show $\hatC \setminus \eta[0,\tau_{\abs{a}+1}] = D \setminus f^{-1}(\wt\eta[0,\tau_A]) = D \setminus f^{-1}(f_t^{-1}(\wt\eta^{(t)}_{[0,\tau_{A^{(t)}}]}))$ is nice with sufficiently positive probability, we can follow the proof of \cref{le:next_tip} with minor modifications. Instead of $\wt\eta$ and $A$, we consider $\wt\eta^{(t)}$ and $A^{(t)}$. We write $f^{(t)}_s = f_{t+s} \circ f_t^{-1}$ for the mapping-out function of $\wt\eta^{(t)}$.

Following the proof of \cref{le:next_tip}, we get that conditionally on $\wt\eta[0,t]$ and $\abs{f_t(f(u))-1} \ge \wt\delta$, with probability at least $p_{\varepsilon,\delta} > 0$ we have
\[ \norm{\gamma-\wt\eta^{(t)}}_{\infty;[0,\tau_{A^{(t)}}]} < \wt\delta/2 \wedge \varepsilon \]
and
\[ \nu_{\bbD;1\to 0}(\Cont(\wt\eta^{(t)}[\tau_{A^{(t)}}, \tau_{\partial B(0,1/128)}]) > \wt c\,c_{\mathrm{N}} \mid \wt\eta^{(t)}[0, \tau_{A^{(t)}}] ) < \wt c\,p_{\mathrm{N}} . \]
We claim that on this event, $D \setminus f^{-1}(\wt\eta[0,\tau_A])$ is $(p_{\mathrm{N}},r_{\mathrm{N}},2c_{\mathrm{N}})$-nice.

Conditionally on $\wt\eta^{(t)}[0, \tau_{A^{(t)}}]$, consider an independent \slek{} $\dbtilde\eta$ from $1$ to $0$, stopped at $\sigma_{r_1}$. As in the proof of \cref{le:next_tip}, $(f^{(t)}_{\tau_{A^{(t)}}})^{-1} \circ \dbtilde\eta$ has the same law as a subsegment of $\wt\eta^{(t)}[\tau_{A^{(t)}}, \tau_{\partial B(0,1/128)}]$, and therefore
\[ \nu_{\bbD;1\to 0}(\Cont((f^{(t)}_{\tau_{A^{(t)}}})^{-1}(\dbtilde\eta[0,\sigma_{r_1}])) > \wt c\,c_{\mathrm{N}} ) < \wt c\,p_{\mathrm{N}} . \]
We need to map $\dbtilde\eta$ back by $f_{\tau_A}^{-1} = f_t^{-1} \circ (f^{(t)}_{\tau_{A^{(t)}}})^{-1}$. Since $(f^{(t)}_{\tau_{A^{(t)}}})^{-1}(\dbtilde\eta[0,\sigma_{r_1}]) \subseteq B(0,1-r_0/4)$ by \cref{le:no_return} and $t$ is small, the derivative $\abs{(f_t^{-1})'}$ is bounded on $(f^{(t)}_{\tau_{A^{(t)}}})^{-1}(\dbtilde\eta[0,\sigma_{r_1}])$ by $2$, say. By the transformation rule for Minkowski content, this implies
\[ \nu_{\bbD;1\to 0}(\Cont(f_{\tau_A}^{-1}(\dbtilde\eta[0,\sigma_{r_1}])) > 2\wt c\,c_{\mathrm{N}}) < \wt c\,p_{\mathrm{N}} , \]
and \cref{le:nice_tip_in_D} implies the claim.
\end{proof}

\begin{proof}[Proof of \cref{le:nice_tip}]
We proceed as outlined below the statement of the lemma. First observe that it suffices to show the statement for radial \slek{} from $1$ to $0$. Indeed, on the event $\{\norm{\gamma-\eta}_{\infty;[0,\tau_A]}\} < \varepsilon$, the laws of the corresponding SLEs are absolutely continuous with density bounded by a constant depending on $\varepsilon$. This is because $f(\infty)$ and $f(u)$ are both at distance at least $2\varepsilon$ from $\gamma[0,\tau_A]$, due to \cref{le:inf_loc,le:inf_dist} and the definition of niceness of $D$, and we may apply \cref{le:sle_abs_cont} to get the desired absolute continuity.

Pick $r_0,r_1,r_{\mathrm{N}} > 0$ as in \cref{le:dist_1,le:no_return,le:nice_tip_in_D}. Then pick $\varepsilon < r_{\mathrm{N}}/2$, and $T$ the capacity of the straight line from $1$ to $1/64$. For this choice of $\varepsilon$, let $p_\varepsilon$ as in \cref{le:support}. Then \cref{it:follow} occurs with probability at least $p_\varepsilon$.

Since the Minkowski content of radial \slek{} (stopped before entering $B(0,1/64)$) is almost surely finite, we can find $c_1 > 0$ such that \ref{it:cont} fails with probability at most $p_\varepsilon/4$.

Next, pick $p_{\mathrm{N}} < p_\varepsilon/4$. Supposing $D$ is $(p_{\mathrm{N}},r_{\mathrm{N}},c_{\mathrm{N}})$-nice, the probability of \ref{it:nice_base} failing is at most $p_{\mathrm{N}} < p_\varepsilon/4$.

Finally, given $r_{\mathrm{N}}$, $\varepsilon$, and $p_{\mathrm{N}}$, we pick $c_{\mathrm{N}}$ according to \cref{le:next_tip}. This will imply \ref{it:nice_tip} fails with probability at most $p_{\mathrm{N}} < p_\varepsilon/4$.
\end{proof}

\begin{proof}[Proof of \cref{pr:nice_tip_iter}]
Pick the constants as in \cref{le:nice_tip}. As already observed in the proof of \cref{le:next_tip}, we have $B(0,1/64) \subseteq A$ and $A \cap B(1,r_0) = \varnothing$. Note that $f \circ \eta$ is an \slek{}$(2)$ in $\bbD$ from $1$ to $f(\infty)$ with force point $f(u)$, with probability at least $p$ all the \cref{it:follow,it:cont,it:nice_base,it:nice_tip} occur for $f \circ \eta$. In particular, \ref{it:nice_tip} means $D\setminus \eta[0,\tau_{\abs{a}+1}]$ is $(p_{\mathrm{N}},r_{\mathrm{N}},c_{\mathrm{N}})$-nice.

To bound $\Cont(\eta[0,\tau_{\abs{a}+1}])$, note that due to \ref{it:follow} and $B(0,1/64) \subseteq A$, $f \circ \eta$ hits $A$ before reaching the capacity of $\gamma$. Moreover, also due to \ref{it:follow}, it does not leave $B(0,1-r_{\mathrm{N}}/4)$ after $\sigma_{r_{\mathrm{N}}}$. Recalling from the proof of \cref{le:nice_tip} that $f(\infty)$ is at distance at least $2\varepsilon$ from $\gamma[0,\tau_A]$, we see that $\abs{(f^{-1})'}$ is bounded on the points of $(f \circ \eta)[\sigma_{r_{\mathrm{N}}},\tau_A]$ by a constant depending on $r_{\mathrm{N}}$ and $\varepsilon$ (by applying Koebe's distortion theorem on a subdomain avoiding $f(\infty)$).
Therefore, combining \cref{it:nice_base,it:cont} and the transformation rule for Minkowski content \eqref{eq:cont_transf}, $\Cont(\eta[0,\tau_{\abs{a}+1}])$ is bounded.
\end{proof}

\begin{proof}[Proof of \cref{le:cont_tail_lower1}]
It suffices to show this for large integers $r$. By \cref{le:initial_nice_tip}, $\hatC \setminus \eta[0,\tau_2]$ is $(p_{\mathrm{N}},r_{\mathrm{N}},c_{\mathrm{N}})$-nice with positive probability. Then, by \cref{pr:nice_tip_iter}, with probability at least $p^r$ we have
\[ \Cont(\eta[\tau_2,\tau_{2+r}]) < rc . \]
Since the Minkowski content of two-sided whole-plane \slek{} and therefore also of whole-plane \slek{}$(2)$ is almost surely finite, $\Cont(\eta[0,\tau_2])$ is also bounded above by some large constant with probability close to $1$.
This finishes the proof.
\end{proof}

The proof shows also the following statement.

\begin{proposition}\label{pr:cont_ltail_cond}
There exist finite and positive constants $p_{\mathrm{N}},r_{\mathrm{N}},c_{\mathrm{N}},c,\wt c_2$ such that the following is true. Let $(D,a,u)$ be as in \eqref{eq:Da}, and $r,\ell>0$ such that $\ell \le r^d$. Let $\lambda = \ell^{-1/(d-1)}r^{1/(d-1)}$. If $\lambda D$ is $(p_{\mathrm{N}},r_{\mathrm{N}},c_{\mathrm{N}})$-nice and $\lambda\abs{a} \ge 1$, then
\[ \nu_{D;a\to\infty;u}( \Cont(\eta[0,\tau_{\abs{a}+r}]) < c\ell ) \ge \exp(-\wt c_2 \ell^{-1/(d-1)} r^{d/(d-1)}) . \]
\end{proposition}

\begin{proof}
By iterating \cref{pr:nice_tip_iter} as in the proof of \cref{le:cont_tail_lower1}, if $D \subset \hatC$ is $(p_{\mathrm{N}},r_{\mathrm{N}},c_{\mathrm{N}})$-nice and $\abs{a} \ge 1$, then
\[ \nu_{D;a\to\infty;u}( \Cont(\eta[0,\tau_{\abs{a}+r}]) < rc ) \ge p^r \]
for any $r \ge 1$. For the general statement, we use the same scaling argument as in the proof of \cref{pr:cont_tail_lower}.
\end{proof}

\subsection{Diameter lower bound for space-filling SLE}
\label{sec:diam-lower-spf}

In this section we will prove the counterparts of Propositions \ref{pr:cont_tail_lower} and \ref{pr:cont_ltail_cond} in the setting of space-filling SLE, which are stated as Propositions \ref{pr:area_tail_lower} and \ref{pr:area_ltail_cond} below. The proofs follow the exact same structure as in the non-space-filling case and we will therefore be brief and only highlight the differences. 

The following is the counterpart of Proposition \ref{pr:cont_tail_lower}. Note in particular that the estimate takes exactly the same form as in the non-space-filling case with $d=2$.
\begin{proposition}\label{pr:area_tail_lower}
	Let $\eta$ be a whole-plane space-filling \slek{} from $-\infty$ to $\infty$, $\kappa > 4$, satisfying $\eta(0)=0$. For some $\tilde c_2 > 0$ we have
	\[ \bbP( \Cont(\eta[0,\tau_r]) < \ell ) \ge \exp(-\tilde c_2 \ell^{-1} r^{2}) \]
	for any $r,\ell>0$ with $\ell \le r^2$.
\end{proposition}
One difference between the space-filling and non-space-filling settings is that in the space-filling case, $\eta|_{[0,\infty]}$ (viewed as a curve in $\C$) is not an SLE$_\kappa(\rho)$ for any vector $\rho$, and therefore the curve does not satisfy the domain Markov property, which plays a crucial role in the argument in the previous section. However, via the theory of imaginary geometry the curve does satisfy a counterpart of the domain Markov property if we also condition on a particular triple of marked points on the boundary of the trace, see Section \ref{se:prelim_spf}. While most of the argument in the non-space-filling case carries through using this observation, we need to modify some parts of the argument, e.g.\ the various absolute continuity arguments comparing different variants of SLE and the proof of Lemma \ref{le:initial_nice_tip}. There are also some parts of the proof that simplify since Cont$(\cdot)$ is equal to Lebesgue area measure, so the natural measure of the curve while staying in a domain is deterministically bounded by the Lebesgue area measure of the domain.

Following the proof in Section \ref{sec:diam-lower-nonspf}, we start by giving the definition of nice for space-filling SLE. This property is now defined for tuples $(D,a,\u)$, $\u=(u_1,u_2,u_3)\in(\partial D)^3$, satisfying
\eqb
\text{$(D,a)$ is as in \eqref{eq:Da}; $a,u_1,u_2,u_3$ are distinct and ordered counterclockwise}.
\label{eq:Daubold}
\eqe
For $r_{\mathrm{N}},c_{\mathrm{N}} > 0$ we say that $(D,a,\u)$ is $(r_{\mathrm{N}},c_{\mathrm{N}})$-nice if
\[ \abs{f(u_1)-1}\wedge\abs{f(u_3)-1} \ge 2r_{\mathrm{N}}, 
\]
and
\eqb
\Cont(f^{-1}( B(1,r_{\mathrm{N}}) \cap \ol{\bbD} )) \le c_{\mathrm{N}} .
\label{eq:nice-spfill2}
\eqe
The second condition is simplified as compared to the non-space-filling case since Cont$(\cdot)$ is equal to Lebesgue measure.

The space-filling counterpart of Proposition \ref{pr:cont_ltail_cond} is the following.
\begin{proposition}\label{pr:area_ltail_cond}
	There exist finite and positive constants $r_{\mathrm{N}},c_{\mathrm{N}},c,\tilde c_2$ such that the following is true. Let $(D,a,\u)$ be as in \eqref{eq:Daubold}, and $r,\ell>0$ such that $\ell \le r^2$. Let $\lambda = \ell^{-1}r$. If $\lambda D$ is $(r_{\mathrm{N}},c_{\mathrm{N}})$-nice and $\lambda\abs{a} \ge 1$, then
	\[ \wh\nu_{D;a\to\infty;\u}( \Cont(\eta[0,\tau_{\abs{a}+r}]) < c\ell ) \ge \exp(-\tilde c_2 \ell^{-1} r^{2}) . \]
\end{proposition}

We now go through Section \ref{sec:diam-lower-nonspf} in chronological order and point out what changes we need to make for the case of space-filling SLE.
The counterparts of Lemma \ref{le:cont_tail_lower1}, Proposition \ref{pr:nice_tip_iter}, and Lemma \ref{le:nice_tip} in the space-filling case are identical as before, except that we consider whole-plane space-filling SLE$_\kappa$ (restricted to the time interval $[0,\infty)$) and $(D,a,\u)$, respectively, instead of SLE$_\kappa(2)$ and $(D,a,u)$. Proposition \ref{pr:cont_tail_lower} is deduced from Lemma \ref{le:cont_tail_lower1} as before via scaling. 
Lemmas \ref{le:dist_1}, \ref{le:inf_loc}, \ref{le:inf_dist}, and \ref{le:no_return} are used in precisely the same form in the space-filling case as in the non-space-filling case; note that these results only concern conformal maps and not SLE. For the counterpart of Corollary \ref{co:target_change}, on the other hand, we modify the statement and the proof as follows. In particular, we do not prove a uniform bound on the Radon-Nikodym derivative in this case. 
\begin{lemma}\label{le:target_change_ig}
	There exist finite constants $c',\varepsilon > 0$ such that the following is true.
	Let $(D,a,\u)$ be as in \eqref{eq:Daubold} with $\abs{a} \ge 1$ and $|f(u_1)-1|\wedge|f(u_3)-1|>2\eps$, and set $\u_0 \defeq (-i,-1,i)$. Let $A = f(\hatC \setminus B(0,\abs{a}+1))$ and let $\gamma$ be the straight line from $1$ to $0$. Then,  
	for any event $E\subset \{ \norm{\gamma-\eta}_{\infty;[0,\tau_A]} < \varepsilon \}$ measurable with respect to the curve until time $\tau_A$ and for 
	$\nu_1,\nu_2\in\{ \wh\nu_{\bbD;1\to 0;\u_0}, 
	\wh\nu_{\bbD;1\to f(\infty);\u_0},
	\wh\nu_{\bbD;1\to 0;f(\u)}, 
	\wh\nu_{\bbD;1\to f(\infty);f(\u)} \}$ 
	it holds that $\nu_1(E)\leq c' \nu_2(E)^{1/2}$.
\end{lemma}

\begin{proof}
	Let $h_1,h_2$ be the variants of the GFF associated with the measures $\nu_1,\nu_2$. Then $h_1,h_2$ can be coupled together such that $h_1=h_2+g$ for $g$ a harmonic function that is zero along the $2\varepsilon$-neighbourhood of $1$ on $\partial\bbD$, and bounded by a constant depending only on $\eps$ on $\{ z\not\in A\,:\,\op{dist}(z,\gamma)<1.5\eps \}$. The lemma is now immediate by the argument in \cite[Lemma 2.1]{hs-mating-eucl}.
\end{proof}

The space-filling counterpart of Lemma \ref{le:nice_tip_in_D} is the following, with the same proof as before. 

\begin{lemma}\label{le:nice_tip_in_D_spf}
Suppose we have the setup of \cref{le:nice_tip_in_D}. If
$\Cont(f_\alpha^{-1}(B(1,r_{1}) \cap \ol{\bbD})) \le \wt c\,c_{\mathrm{N}}, $
then $(D \setminus f^{-1}(\alpha),\ f^{-1}(\alpha(1)),\ \u)$ is $(r_{\mathrm{N}},c_{\mathrm{N}})$-nice.
\end{lemma}

Lemma \ref{le:support} still holds in the space-filling setting, with the only modification being that we replace $\nu_{\bbD;1\to 0}$ by $\wh\nu_{\bbD;1\to 0;\u}$ for fixed $\u=(u_1,u_2,u_3)$ such that $1,u_1,u_2,u_3$ are ordered counterclockwise. 
The proof of the lemma in the space-filling setting will follow by iterative applications of Lemma \ref{le:support-spfill-nbh} right below. 
Notice that in this lemma we do not rule out scenarios where $\eta$ oscillates back and forth while staying close to $\gamma$, while do not want such behavior in Lemma \ref{le:support} as we consider the $L^\infty$ distance.

\begin{lemma}\label{le:support-spfill-nbh}
Let $\u=(u_1,u_2,u_3)$ be distinct points of $(\partial\bbD)\setminus\{1 \}$ such that $1,u_1,u_2,u_3$ are ordered clockwise and let $\gamma\colon[0,T]\to \bbD \setminus \{0\}$ be a simple curve with $\gamma(0) = 1$ for some $T>0$. For $\delta>0$ let $A(\delta)$ denote the $\delta$-neighborhood of $\gamma([0,T])$. 
Define stopping times
\[
\sigma_1=\inf\{t\geq 0\,:\,|\eta(t)-\gamma(T)|<\delta \},\qquad
\sigma_2 = \inf\{ t\geq 0\,:\,\eta(t)\not\in A(\delta) \}.
\]
Then for any $\delta>0$  we have $\wh\nu_{\bbD;1\to 0;\u}[\sigma_1<\sigma_2]>0$.
\end{lemma}

\begin{proof}
The proof is identical to the proof of \cite[Lemma 2.3]{mw-cut-double}. The only difference is that the lemma treats chordal curves and we consider a radial curve, but the argument is the same. More precisely, we need the argument in the lemma for $\kappa>4$, corresponding to the case where the curve in question is a counterflow line. It is enough to prove the claim for counterflow lines since by \cite[Theorem 1.13]{ms-ig4}, for every rational $z$, the set $\eta[0,\tau_z]$ (with $\tau_z$ the first hitting time of $z$) is the union of branches of a branching counterflow line. (Strictly speaking, \cite[Theorem 1.13]{ms-ig4} is stated for GFF with fully branchable boundary condition, but we can perform an absolutely continuous change of measure such that the law of $h\big|_{A(\delta)}$ becomes the law of the restriction of a GFF with fully branchable boundary condition.)
\end{proof}

\begin{proof}[Proof of Lemma \ref{le:support} with $\wh\nu_{\bbD;1\to 0;\u}$ instead of $\nu_{\bbD;1\to 0}$]
Note that Lemma \ref{le:support-spfill-nbh} can also be applied to simply connected domains not equal to $\bbD$ by applying an appropriate conformal change of coordinates to the GFF and the curves $\eta$ and $\gamma$. Set $j_0=10/\eps+1$, assuming without loss of generality that $j_0$ is an integer. We apply Lemma \ref{le:support-spfill-nbh} iteratively for $j=1,2,\dots,j_0$ with domain $D=\bbD\setminus \eta[0,\tau_{j-1}]$, $\gamma$ the straight line from $\eta(\tau_{j-1})$ to $1-j\eps/10$, $\delta\ll\eps$, and $\tau_{j}$ defined as the stopping time $\sigma_1$ in the lemma (with $\tau_0=0$). 
\end{proof}

The statement and proof of \cref{le:next_tip} simplify as follows, since for suitable $c_{\mathrm{N}}$ the condition of \cref{le:nice_tip_in_D_spf} is trivially satisfied as $f_{\alpha}^{-1}$ maps into $\bbD$.
\begin{lemma}
Suppose that we have the setup of \cref{le:next_tip}. If $\norm{\gamma-\eta}_{\infty;[0,\tau_A]} < \varepsilon$, then $(D \setminus f^{-1}(\eta[0,\tau_A]),\ f^{-1}(\eta(\tau_A)),\ \u)$ is $(r_{\mathrm{N}}, c_{\mathrm{N}})$-nice.
\end{lemma}

The statement and proof of Lemma \ref{le:initial_nice_tip} are identical to the non-space-filling case, except that we consider $\wh\nu_{D;a\to\infty;\u}$ instead of $\nu_{D;a\to\infty;u}$ and that the proof relies on the following lemma (proved at the very end of the proof of \cite[Lemma 3.1]{ghm-kpz-sle}) to treat the case where both $|f(u_1)-1|$ and $|f(u_3)-1|$ are small.\footnote{Note that there is a typo in the published version of \cite{ghm-kpz-sle} which interchanges the role of ``left'' and ``right'' in the second-to-last sentence of the below lemma.} 
A similar lemma is used when only one of $|f(u_1)-1|$ and $|f(u_3)-1|$ is small, with the only difference being that we only grow a flow line from the point in $\{f(u_1),f(u_3)\}$ that is close to 1.

\begin{lemma}
There exist positive constants $\delta, q > 0$ such that the following is true. Let $\u=(u_1,u_2,u_3)$ be such that $1,u_1,u_2,u_3\in\partial\BB D$ are distinct and ordered clockwise. Let $\wh h$ denote the variant of the GFF in $\D$ that is used when defining $\wh\nu_{\D;1\to 0;\u}$ at the very end of \cref{se:prelim_spf}. Suppose that $|u_1-1| \vee|u_3-1| \leq\delta$.
	Let $\eta^{\mathrm{L}}_{u_1}$ (resp.\ $\eta^{\mathrm{R}}_{u_3}$) denote the flow line of $\wh h$ started from $u_1$ (resp.\ $u_3$) with angle $\pi/2$ (resp.\ $-\pi/2$) targeted at $0$, and let $S_1$ (resp.\ $S_3$) be its exit time from $B_{2\delta}(1)$. 
	Let $U$ be the connected component containing $0$ of $\BB D\setminus ( \eta^{\mathrm{L}}_{u_1}([0,S_1]) \cup \eta^{\mathrm{R}}_{u_3}([0,S_3]) )$. Let $E$ be the event that the harmonic measure from 0 in $U$ of each of the right side of $\eta^{\mathrm{L}}_{u_1}([0,S_1])$ and the left side of $\eta^{\mathrm{R}}_{u_3}([0,S_3])$ is at least $q$. 
	Then $\BB P(E) \geq q$. 
\end{lemma}
Finally, the proofs of the space-filling counterparts of 
Lemma \ref{le:cont_tail_lower1}, 
Proposition \ref{pr:nice_tip_iter}, 
Lemma \ref{le:nice_tip}, and Proposition \ref{pr:cont_ltail_cond} go through precisely as before. The only minor change is when justifying the space-filling counterpart of Lemma \ref{le:nice_tip}\ref{it:nice_base}, where we now use
\[
\Cont(f^{-1}(\eta[0,\sigma_{r_{\mathrm{N}}}])) \le \Cont(f^{-1}(  B(1,r_{\mathrm{N}})\cap\ol{\bbD}  ))\le c_{\mathrm{N}}.
\]

\subsection{Lower bounds on the regularity of SLE}
\label{se:lowerbounds_pf}

In this section we conclude the proofs of \cref{thm:main_lil,thm:main_moc,thm:main_var}. Given Proposition \ref{pr:upperbounds}, it remains to prove that the constants $c_0,c_1$ in the theorems are positive.

A natural approach for proving lower bounds is to find disjoint intervals $[s_k,t_k]$ on which the increment of $\eta$ is exceptionally big with a certain (small) probability. If the sum of these probabilities is infinite and there is sufficient decorrelation of these events, then a variant of the second Borel-Cantelli lemma will imply that infinitely many of these events will occur. In our case, however, the correlations are not easy to control. The probabilities of having exceptionally large increments in the non-space-filling case are given in \cref{pr:cont_tail_lower} or its ``conditional'' version \cref{pr:cont_ltail_cond}. But when conditioning on $\eta[0,s_k]$, \cref{pr:cont_ltail_cond} does not apply to every realisation of $\eta[0,s_k]$ but only the ones that are nice (as defined in \cref{sec:diam-lower-nonspf}). Another attempt would be to use the corresponding upper bound of the probability given by \cref{pr:diam_tail_upper}. But the upper and lower bounds differ not just by a factor but by a power, which is too weak to guarantee sufficient decorrelation. The exact same issues arise in the case of space-filling SLE.

Our idea is to introduce another sequence of events $B_k$ on which the conditional lower bound on the interval $[s_k,t_k]$ is valid again (an example is the event that $\eta[0,s_k]$ is nice). We formulate the argument as an abstract lemma.

\begin{lemma}\label{le:conditional_prob}
Let $A_1,...,A_k$ and $B_1,...,B_k$ be events. Suppose that $\bbP(B_j) \ge p$ for some $p>0$ and all $j$. Let
\[
p_j = \bbP( A_j \mid B_j \cap (A_1 \cap B_1)^c \cap ... \cap (A_{j-1} \cap B_{j-1})^c ) .
\]
Then, if $q\in[0,1]$ satisfies
\begin{equation}\label{eq:cond_p_sum}
\exp\left( -(1-\frac{1-p}{q})(p_1+...+p_k) \right) < q ,
\end{equation}
we have
\[ \bbP( (A_1 \cap B_1) \cup ... \cup (A_k \cap B_k) ) > 1-q . \]
\end{lemma}

\begin{proof}
Suppose the contrary, i.e.
\begin{equation}\label{eq:cond_p_contra}
\bbP( (A_1 \cap B_1)^c \cap ... \cap (A_k \cap B_k)^c ) \ge q .
\end{equation}
Observe this in that case
\[ \bbP( B_j^c \mid (A_1 \cap B_1)^c \cap ... \cap (A_{j-1} \cap B_{j-1})^c ) \le \frac{\bbP(B_j^c)}{\bbP( (A_1 \cap B_1)^c \cap ... \cap (A_{j-1} \cap B_{j-1})^c )} \le \frac{1-p}{q} \]
by our assumptions.
It follows that
\[ \bbP( A_j \cap B_j \mid (A_1 \cap B_1)^c \cap ... \cap (A_{j-1} \cap B_{j-1})^c ) \ge (1-\frac{1-p}{q})p_j \]
and therefore inductively
\[ \bbP( (A_1 \cap B_1)^c \cap ... \cap (A_k \cap B_k)^c ) \le \prod_{j \le k} \left(1-(1-\frac{1-p}{q})p_j\right) \le \exp\left( -(1-\frac{1-p}{q}) \sum_{j \le k} p_j \right) \]
which by \eqref{eq:cond_p_sum} contradicts \eqref{eq:cond_p_contra}.
\end{proof}

We now prove the lower bound in \cref{thm:main_moc}, namely that the constant $c_0$ in the theorem statement must be positive; see Proposition \ref{pr:upperbounds} for the other assertions of the theorem. 
We use \cref{le:conditional_prob} to show that the lower bound is satisfied with positive probability. This will imply the claim.

As in the previous sections, we define $\tau_r=\inf\{t\geq 0\,:\, |\eta(t)|=r \}$ to be the hitting time of radius $r$. 

\begin{lemma}\label{le:mod_lower_p}
There exist $p_0 > 0$ and $b_0 > 0$ such that the following is true. For any $(D,a)$ as in \eqref{eq:Da} with $\abs{a} > 0$ and $u\in\partial D\setminus\{a\}$, we have 
\eqbn
\nu_{D;a \to \infty;u}\left( \sup_{s,t \in [0,\tau_{\abs{a}+r}]} \frac{\abs{\eta(t)-\eta(s)}}{\abs{t-s}^{1/d}(\log \abs{t-s}^{-1})^{1-1/d}} \ge b_0 \right)
\ge p_0
\eqen
for any $r > 0$.

The same holds for space-filling \slek{}, except that we consider $(D,a,\u)$ as in \eqref{eq:Daubold} and $\wh\nu_{D;a\to\infty;\u}$ instead of $(D,a,u)$ and $\nu_{D;a\to\infty;u}$.
\end{lemma}

\begin{proof}[Proof of \cref{thm:main_moc} given \cref{le:mod_lower_p}]
By \cref{pr:upperbounds} it remains to show $c_0 > 0$ in \cref{pr:01law_mv}. Applying \cref{le:mod_lower_p} with arbitrarily small $r>0$, we see that with probability at least $p_0$ we have
\[ \lim_{r\downarrow 0} \sup_{s,t \in [0,\tau_{\abs{a}+r}]} \frac{\abs{\eta(t)-\eta(s)}}{\abs{t-s}^{1/d}(\log \abs{t-s}^{-1})^{1-1/d}} \ge b_0 . \]
This implies $S_I \ge b_0$ for some random interval $I$ (with $S_I$ defined in the proof of \cref{pr:01law_mv}). But since $S_I$ is a deterministic constant independent of $I$, the result follows.
\end{proof}

\begin{proof}[Proof of \cref{le:mod_lower_p}]
We will do the proof for \slek$(2)$ and explain at the end what modifications are necessary for the case of space-filling \slek{}.

For $\varepsilon > 0$ and $k=1,...,\lfloor\varepsilon^{-1}r\rfloor$ we define the event
\[ A_{\varepsilon,k} = \{ \Cont(\eta[\tau_{\abs{a}+(k-1)\varepsilon},\tau_{\abs{a}+k\varepsilon}]) \le a_0 \varepsilon^d (\log \varepsilon^{-1})^{-(d-1)} \} \]
where $a_0 > 0$ is a constant whose value will be decided upon later. Note that on the event $A_{\varepsilon,k}$ we have 
\[ \abs{\eta(\tau_{\abs{a}+k\varepsilon})-\eta(\tau_{\abs{a}+(k-1)\varepsilon})} \geq \varepsilon 
\gtrsim \abs{\tau_{\abs{a}+k\varepsilon}-\tau_{\abs{a}+(k-1)\varepsilon}}^{1/d}(\log \abs{\tau_{\abs{a}+k\varepsilon}-\tau_{\abs{a}+(k-1)\varepsilon}}^{-1})^{(d-1)/d} \]
which is the lower bound we want to prove.

Let $p_{\mathrm{N}},r_{\mathrm{N}},c_{\mathrm{N}},\wt c_2$ be as in \cref{pr:cont_ltail_cond}. Let $B_{\varepsilon,k}$ denote the event that $(D_k,a_k,u_k)$ is $(p_{\mathrm{N}},r_{\mathrm{N}},c_{\mathrm{N}})$-nice where
\[
(D_k,a_k,u_k) \defeq 
\left(
\varepsilon^{-1}(\log\varepsilon^{-1})\,(D \setminus \eta[0,\tau_{\abs{a}+(k-1)\varepsilon}]), \ 
\varepsilon^{-1}(\log\varepsilon^{-1})\,\eta(\tau_{\abs{a}+(k-1)\varepsilon}), \ 
\varepsilon^{-1}(\log\varepsilon^{-1}) u
\right).
\]
On the event $B_{\varepsilon,k}$ we have by \cref{pr:cont_ltail_cond}
\begin{equation*}
\nu_{D;a\to\infty;u}(A_{\varepsilon,k} \mid \eta[0,\tau_{\abs{a}+(k-1)\varepsilon}]) \ge \exp\left(-\wt c_2 a_0^{-1/(d-1)} (\log \varepsilon^{-1})\right) = \varepsilon^{\wt c_2 a_0^{-1/(d-1)}} = \varepsilon^{1/2} 
\end{equation*} 
for a suitable choice of $a_0$.

Note that $\varepsilon^{-1}(\log\varepsilon^{-1})\,\eta[0,\tau_{\abs{a}+(k-1)\varepsilon}]$ is a radial \slek{}$(2)$ in $\varepsilon^{-1}(\log\varepsilon^{-1})D$ stopped at hitting radius $\varepsilon^{-1}(\log\varepsilon^{-1})\abs{a}+(k-1)(\log\varepsilon^{-1})$. By \cref{le:initial_nice_tip}, we have $\nu_{D;a\to\infty;u}(B_{\varepsilon,k}) \ge p_1$ where $p_1>0$ does not depend on $D,\varepsilon,k$.

Applying \cref{le:conditional_prob} with any choice of $q \in {(1-p_1,1)}$ and noting that
\[ \exp\left( -(1-\frac{1-p_1}{q})r\varepsilon^{-1}\varepsilon^{1/2} \right) < q \]
for small $\varepsilon > 0$, this implies the result with $p_0 = 1-q$.

In the case of space-filling \slek{}, we instead apply \cref{pr:area_ltail_cond} and the space-filling analogue of \cref{le:initial_nice_tip} as explained at the end of \cref{sec:diam-lower-spf}.
\end{proof}

We finally show the lower bounds in \cref{thm:main_lil}. By \cref{pr:var_lil}, this implies also the lower bound in \cref{thm:main_var}. Given \cref{pr:upperbounds}, it remains to prove that the constants $c_0,c_1$ are positive. By the stationarity of $\eta$, it suffices to show the claim for $t_0=0$.

\begin{proof}[Proof of \cref{thm:main_lil}]
The proof is almost identical for whole-plane \slek$(2)$ ($\kappa\in(0,8]$) and for whole-plane space-filling \slek{} ($\kappa>4$). We begin with the former.

We first prove the $t \downarrow 0$ statement.
Define the sequences
\[ 
r_k = \exp(-k^2),
\quad m_k = a_0 r_k^d (\log\log r_k^{-1})^{-(d-1)} 
\]
where the exact value of $a_0 > 0$ will be decided upon later. Note that if
\[ \Cont(\eta[0,\tau_{r_k}]) < m_k , \]
then for some $t < m_k$ we have
\[ \abs{\eta(t)} = r_k \asymp m_k^{1/d}(\log\log m_k^{-1})^{(d-1)/d} \]
which is exactly the lower bound in the law of the iterated logarithm.

By \cref{pr:cont_tail_lower} we have
\[ \bbP( \Cont(\eta[0,\tau_{r_k}]) < m_k ) \ge \exp(-\wt c_2 m_k^{-1/(d-1)}r_k^{d/(d-1)}) \asymp k^{-\wt c a_0^{-1/(d-1)}} , \]
so $a_0$ can be chosen such that the sum of the probabilities diverges. We would like to argue by \cref{le:conditional_prob} that with positive probability this happens for infinitely many $k$. Then \cref{pr:01law} implies that the probability must actually be $1$.

To show this, we introduce another sequence $r'_k$ with $r_{k+1} < r'_{k+1} < r_k$. Let
\[ r'_{k+1} = 2 r_k (\log k)^{-1} \]
and define the events
\[ 
A_k = \{ \Cont(\eta[\tau_{r'_{k+1}},\tau_{r_k}]) \le m_k \},
\quad A'_k = \{ \Cont(\eta[0,\tau_{r'_{k+1}}]) \le m_k \},
\quad \bar A_k = A_k \cap A'_k .
\]
As noted above, on the event $\limsup_k \bar A_k$ we have
\[ \limsup_{t \downarrow 0} \frac{\abs{\eta(t)}}{t^{1/d}(\log\log t^{-1})^{1-1/d}} \ge b_0 \]
where $b_0$ depends on the choice of $a_0$.

We are left to show $\bbP(\limsup_k \bar A_k) > 0$. Let $B'_k$ denote the event that $\hatC \setminus r_k^{-1}(\log k)\,\eta[0,\tau_{r'_{k+1}}]$ is $(p_{\mathrm{N}},r_{\mathrm{N}},c_{\mathrm{N}})$-nice. On the event $B'_k$ we have by \cref{pr:cont_ltail_cond}
\begin{equation*}
\bbP( \Cont(\eta[\tau_{r'_{k+1}},\tau_{r_k}]) < a_1 m_k \mid \eta[0,\tau_{r'_{k+1}}] ) \ge \exp(-\wt c_2 m_k^{-1/(d-1)}r_k^{d/(d-1)}) \asymp k^{-\wt c a_0^{-1/(d-1)}} = k^{-1} 
\end{equation*}
for suitable choice of $a_0,a_1$.

Since $r_k^{-1}(\log k)\,\eta[0,\tau_{r'_{k+1}}]$ is again a whole-plane \slek{}$(2)$ stopped at hitting radius $2$, we have $\bbP(B'_k) \ge p_1$ by \cref{le:initial_nice_tip} where $p_1 > 0$ does not depend on $k$.

The content $\Cont(\eta[0,\tau_{r'_{k+1}}])$ can be controlled by the negative moments in \cite[Lemma 1.7]{zhan-hoelder}, implying
\eqb \begin{split}
    \bbP( \Cont(\eta[0,\tau_{r'_{k+1}}]) > m_k ) 
    &\le \bbP( \abs{\eta(m_k)} < r'_{k+1} ) \le (r'_{k+1})^{d(1-\varepsilon)} \bbE \abs{\eta(m_k)}^{-d(1-\varepsilon)}\\
    &\lesssim (r'_{k+1})^{d(1-\varepsilon)} m_k^{-(1-\varepsilon)}
    \asymp (\log k)^{-1+\varepsilon} .
\end{split} 
\label{eq:neg-moment}
\eqe
It follows that $\bbP(A'_k \cap B'_k) > p_1/2$.

Write $B_k = A'_k \cap B'_k$. We apply \cref{le:conditional_prob} with the events $A_{k'},...,A_k$ and $B_{k'},...,B_k$. Pick any $q \in {(1-p_1/2,1)}$ and note that for any $k \in \bbN$ we can find $k' > k$ such that
\[ \exp\left( -(1-\frac{1-p_1/2}{q})((k')^{-1}+...+k^{-1}) \right) < q . \]
This implies $\bbP\left( \bigcup_{k' \ge k} \bar A_{k'} \right) \ge \bbP\left( \bigcup_{k' \ge k} (A_{k'} \cap B_{k'}) \right) > 1-q$ for all $k$, and therefore we have $\bbP(\limsup_k \bar A_k) \ge 1-q > 0$.

The proof is identical for space-filling \slek{}, except that the proof of \eqref{eq:neg-moment} is even easier  since $\Cont(\eta[0,\tau_{r'_{k+1}}])\leq \pi(r'_{k+1})^2$ deterministically.

The proof of the $t \to \infty$ statement is identical when we set
\[ 
r_k = \exp(k^2),
\quad r'_{k-1} = 2r_k(\log k)^{-1},
\quad m_k = a_0 r_k^d (\log\log r_k)^{-(d-1)}. 
\]
\end{proof}

\section{$\psi$-variation of chordal SLE}
\label{se:chordal}

In this section, we prove finite $\psi$-variation for chordal SLE. We expect that an analogue of \cref{pr:diam_tail_upper} holds also for chordal SLE, but is is challenging to establish a sharp enough tail due to the lack of exact symmetries when working in a domain. Moreover, there may be an arbitrarily large height difference of the GFF between interior flow lines and the domain boundary for which we have not succeeded in gaining a sufficient control. For this reason, we chose a different approach.

Throughout this section, we let $\eta$ be a chordal \slek{}, $\kappa>0$ (either space-filling or non-space-filling). We will now give a brief proof outline. Let $\wt\eta$ denote a curve as in \cref{it:non-spfill,it:spfill}. The idea of our approach is to use our regularity result for $\wt\eta$ (\cref{thm:main_var}) along with the fact that conditioned on $\wt\eta|_{(-\infty,0]}$, the rest of the curve has the law of $\varphi\circ\eta$, where $\varphi$ is an appropriately chosen conformal map from $\bbH$ to (a component of) $\bbC\setminus\fill(\wt\eta(-\infty,0])$. If we consider segments of $\eta$ away from the boundary of $\bbH$, then the regularity result for $\eta$ follows immediately from the regularity result for $\wt\eta$ (see Lemma \ref{le:var_chordal_int} below). Treating points near the boundary, however, is challenging since $\varphi$ is highly irregular near the domain boundary. The key to our approach for points near the boundary is Proposition \ref{pr:f_incr_dist} which shows that if we consider a conformal map to the complement of an SLE curve, then with positive probability this conformal map does not decrease distances too much. Once we have this estimate, we can bound the increments $|\eta(t_{i+1})-\eta(t_i)|$ in terms of $|\varphi\circ\eta(t_{i+1})-\varphi\circ\eta(t_i)|$, and then we can use our bound for the latter; see the proof of \cref{pr:bdy-fin-var}. The main technical challenge of the section is to establish Proposition \ref{pr:f_incr_dist} and we do this via the reverse Loewner differential equation and It\^o calculus; see \cref{sec:f_incr_dist} for a more detailed outline. 

Throughout the section, we let $\psi(x) = x^d(\logp\logp\frac{1}{x})^{-(d-1)}$. Note that for $u\ge 1$
\begin{equation}\label{eq:psi_scaling}
u^d\psi(x) \le \psi(ux) \le u^{d}\log\log(ue^e)\psi(x) .
\end{equation}

We have the following corollary of \cref{thm:main_var}.
\begin{lemma}\label{le:var_chordal_int}
For any $R>0$, there exists $c>0$ such that
\begin{multline*}
\bbP\left( [\eta]_{\psi\text{-var},[s,t]} > u(1\vee\Cont(\eta[s,t])^{1/d}) \text{ for some $s,t$ with $\eta[s,t]\subseteq[-R,R]\times[R^{-1},R]$} \right) \\
\le c^{-1}\exp(-c u^{d/(d-1)}) .
\end{multline*}
\end{lemma}

\begin{proof}
Let $\wt\eta$ be a two-sided whole-plane \slek{}, $\kappa\le 8$, or whole-plane space-filling \slek{}, $\kappa>4$, as in \cref{it:non-spfill,it:spfill}. Pick a conformal map $\varphi$ from $\bbH$ to (a component of) $\bbC\setminus\fill(\wt\eta(-\infty,0])$ such that $0$ and $\infty$ are sent to the points of $\varphi(\bbH)$ visited first and last, respectively, by $\wt\eta|_{[0,\infty)}$. We can couple $\wt\eta$ and $\eta$ such that $\varphi$ is independent of $\eta$ and $\varphi(\eta)$ is (a segment of) $\wt\eta\big|_{[0,\infty)}$. The scaling of $\varphi$ can be chosen arbitrarily as long as the independence requirement is satisfied. There is a $c_1>0$ such that the following event $E$ occurs with positive probability
\[
E= \{\abs{\varphi'(z)} \in [c_1,c_1^{-1}] \text{ for } z\in[-R,R]\times[R^{-1},R] \}
\cap
\{\varphi([-R,R]\times[R^{-1},R]) \subseteq B(0,c_1^{-1}) \}.
\]
By the independence of $\varphi$ and $\eta|_{[0,\infty)}$,
\[ 
\begin{split}
&\P[E]\cdot \bbP\left( \sup_{\substack{t_0<\dots<t_r\\ \eta[t_0,t_r]\subseteq[-R,R]\times[R^{-1},R]}} \sum_i \psi\left(\frac{\abs{\eta(t_{i+1})-\eta(t_i)}}{u(1\vee\Cont(\eta[t_0,t_r])^{1/d})}\right) > 1 \right) \\
&\qquad \le 
\bbP\left( \sup_{\substack{t_0<\dots<t_r\\ \wt\eta[t_0,t_r]\subseteq B(0,c_1^{-1})}} \sum_i \psi\left(\frac{\abs{\wt\eta(t_{i+1})-\wt\eta(t_i)}}{c u(1\vee\Cont(\wt\eta[t_0,t_r])^{1/d})}\right) > 1 \right) \\
&\qquad = \bbP\left( [\wt\eta]_{\psi\text{-var},[s,t]} > c u(1\vee\abs{t-s}^{1/d}) \text{ for some } s<t \text{ with } \wt\eta[s,t]\subseteq B(0,c_1^{-1}) \right) .
\end{split} 
\]
To bound the latter probability, pick $T=\exp(\delta u^{d/(d-1)})$ with $\delta>0$ sufficiently small. Then by \cref{le:sle_transience}
\[
\bbP( \wt\eta[T,\infty) \cap B(0,c_1^{-1}) \neq\varnothing ) \lesssim T^{-\beta} = \exp(-\beta \delta u^{d/(d-1)}) .
\]
By \cref{thm:main_var}, for any $s<t$,
\[
\bbP\left( [\wt\eta]_{\psi\text{-var},[s,t]} > cu\abs{t-s}^{1/d} \right) 
\lesssim \exp(-\wt{c}u^{d/(d-1)}) .
\]
Summing this up for $s,t \in \bbN$, $s<t \le T$, yields the bound
\[ \begin{split}
\bbP\left( [\wt\eta]_{\psi\text{-var},[s,t]} > cu(1\vee\abs{t-s}^{1/d}) \text{ for some } s<t\le T \right) 
&\lesssim T^2 \exp(-\wt{c}u^{d/(d-1)}) \\
&= \exp(-(\wt{c}-2\delta)u^{d/(d-1)}) .
\end{split} \]
\end{proof}

To transfer the results to segments close to the boundary, we use the following estimate that roughly says that the conformal map into the complement of an SLE does not decrease distances too much with positive probability. 
We will prove the proposition in the next subsection.

\begin{proposition}\label{pr:f_incr_dist}
Let $\eta$ be a chordal \slek{} in $\bbH$ from $0$ to $\infty$ with the capacity parametrization and let $f\colon \bbH \to \bbH\setminus\fill(\eta[0,1])$ be the conformal map with $f(0)=\eta(1)$, $f(z) = z+O(1)$ as $z\to\infty$. 
For any $\lambda>0$ and $R>0$, there exists $c=c(\kappa,\lambda,R)>1$ such that
\[ 
\bbE\left[ \abs{f(z_2)-f(z_1)}^\lambda \,1_{E_{R,c}} \right] \ge c^{-1}\abs{z_2-z_1}^{\lambda}
\]
for any $z_1,z_2 \in [R^{-1},R]\times[0,R]$ where $E_{R,c}$ denotes the event that $f(z) \in [-c,c]\times[c^{-1},c]$ for all $z \in [-2R,2R]\times[0,R]$.
\end{proposition}

\begin{proposition}\label{pr:bdy-fin-var}
Let $R>0$. If $\eta$ is a non-space-filling chordal SLE$_\kappa$ in $(\bbH,0,\infty)$, $\kappa\in(0,8)$, then
\[
\bbP\left( [\eta]_{\psi\text{-var},[s,t]} > u \text{ for some $s,t$ with $\eta[s,t]\subseteq[R^{-1},R]\times[0,R]$} \right) 
= o^\infty(u^{-1})
\]
where $o^\infty(u^{-1})$ denotes a function that decays faster than any power of $u^{-1}$. If $\eta$ is a space-filling chordal SLE$_\kappa$ in $(\bbH,0,\infty)$, $\kappa>4$, then 
\[
\bbP\left( [\eta]_{\psi\text{-var},[s,t]} > u \text{ for some $s,t$ with $\eta[s,t]\subseteq[R^{-1},R]\times[0,R]$} \right) 
\le c^{-1}\exp(-c u^{d/(d-1)}).
\]
for some $c=c(\kappa,R)>0$.
\end{proposition}

We expect, but do not prove, that also in the non-space-filling case $\eta$ satisfies the stronger estimate
\begin{multline*}
\bbP\left( [\eta]_{\psi\text{-var},[s,t]} > u(1\vee\Cont(\eta[s,t])^{1/d}) \text{ for some $s,t$ with $\eta[s,t]\subseteq[R^{-1},R]\times[0,R]$} \right) \\
\lesssim \exp(-c u^{d/(d-1)}) .
\end{multline*}

\begin{proof}
Pick $f\colon \bbH \to \bbH\setminus\fill(\wt\eta[0,1])$ as in \cref{pr:f_incr_dist} independently of $\eta$, and let $\wt\eta\big|_{[1,\infty)} = f(\eta)$.

Let $t_1<t_2<...<t_r$ be such that $\eta[t_0,t_r]\subseteq[R^{-1},R]\times[0,R]$. By the independence of $f$ and $\eta$, \cref{pr:f_incr_dist} with $\lambda=1$ (in fact, any $\lambda<d$ will do) and Jensen's inequality with the function $\psi$,\footnote{Strictly speaking, $\psi$ is not convex everywhere but it is convex near $0$ and at large $x$. Therefore it is comparable to a convex function in the sense that there is a convex function $\wt\psi$ and a constant $c>0$ such that  $c\wt\psi(x)\leq\psi(x)\leq c^{-1}\wt\psi(x)$ for all $x>0$.}
\[ \begin{split}
\sum_i \psi\left(\frac{\abs{\eta(t_{i+1})-\eta(t_i)}}{u}\right) 
&\lesssim \sum_i \bbE\left[ \psi\left(\frac{\abs{f(\eta(t_{i+1}))-f(\eta(t_i))}}{u}\right) 1_{E_{R,c}} \mmiddle| \eta \right] \\
&\le \bbE\left[ \sup_{\substack{s_0<\dots<s_r\\ \wt\eta[s_0,s_r]\subseteq[-c,c]\times[c^{-1},c]}} \sum_i \psi\left(\frac{\abs{\wt\eta(s_{i+1})-\wt\eta(s_i)}}{u}\right) \mmiddle| \eta \right] .
\end{split} \]
Using that
\[
\bbE[X \mid \eta] = \int_0^\infty \bbP(X>y \mid \eta)\,dy
\le 1+\sum_{n\in\bbN_0} 2^n \bbP(X>2^n \mid \eta) ,
\]
we get that for some small constant $\delta>0$,
\[
\{ \bbE[X \mid \eta] > 3 \} \subseteq \bigcup_{n\in\bbN} \left\{ \bbP(X>2^n \mid \eta) > \delta 2^{-n(1+\delta)} \right\}.
\]
It follows that for some $\wt{c}>0$,
\begin{align}
&\bbP\left( \sup_{\substack{t_0<\dots<t_r\\ \eta[t_0,t_r]\subseteq[R^{-1},R]\times[0,R]}} \sum_i \psi\left(\frac{\abs{\eta(t_{i+1})-\eta(t_i)}}{u}\right) > \wt{c} \right) \nonumber\\
&\quad \le \bbP\left( \bbE\left[ \sup_{\substack{t_0<\dots<t_r\\ \wt\eta[t_0,t_r]\subseteq[-c,c]\times[c^{-1},c]}} \sum_i \psi\left(\frac{\abs{\wt\eta(s_{i+1})-\wt\eta(s_i)}}{u}\right) \mmiddle| \eta \right] > 3 \right) \nonumber\\
&\quad \le \sum_{n\in\bbN} \bbP\left( \bbP\left( \sup_{\substack{t_0<\dots<t_r\\ \wt\eta[t_0,t_r]\subseteq[-c,c]\times[c^{-1},c]}} \sum_i \psi\left(\frac{\abs{\wt\eta(s_{i+1})-\wt\eta(s_i)}}{u}\right) > 2^n \mmiddle| \eta\right) > \delta 2^{-n(1+\delta)} \right) \nonumber\\
&\quad \le \sum_{n\in\bbN} \delta^{-1}2^{n(1+\delta)}\bbP\left( \sup_{\substack{t_0<\dots<t_r\\ \wt\eta[t_0,t_r]\subseteq[-c,c]\times[c^{-1},c]}} \sum_i \psi\left(\frac{\abs{\wt\eta(s_{i+1})-\wt\eta(s_i)}}{u}\right) > 2^n \right) . \label{eq:var_chordal_est}
\end{align}
To apply \cref{le:var_chordal_int}, we need to control the Minkowski content of $\wt\eta$. In case of space-filling SLE, since $\Cont_{d=2} = \operatorname{area}$, it is bounded deterministically. For non-space-filling SLE, by \cite[Theorem~1.2]{rz-cont-moments},
\eqb
\bbP\left( \Cont(\wt\eta \cap [-c,c]\times[c^{-1},c]) > T \right) = o^\infty(T^{-1}) .
\label{eq:cont-upper-tail}
\eqe
By \eqref{eq:psi_scaling} and \cref{le:var_chordal_int},
\[ \begin{split}
&\bbP\left( \sup_{\substack{t_0<\dots<t_r\\ \wt\eta[t_0,t_r]\subseteq[-c,c]\times[c^{-1},c]}} \sum_i \psi\left(\frac{\abs{\wt\eta(s_{i+1})-\wt\eta(s_i)}}{u}\right) > 2^n,\ \Cont(\wt\eta \cap [-c,c]\times[c^{-1},c]) \le T \right) \\
&\quad \lesssim \bbP\left( \sup_{\substack{t_0<\dots<t_r\\ \wt\eta[t_0,t_r]\subseteq[-c,c]\times[c^{-1},c]}} \sum_i \psi\left(\frac{\abs{\wt\eta(s_{i+1})-\wt\eta(s_i)}}{u\,2^{n(1-\delta)/d}\,T^{-1/d}(1\vee\Cont(\wt\eta[t_0,t_r])^{1/d})}\right) > 1\right) \\
&\quad \lesssim \exp\left(-cu^{d/(d-1)}T^{-1/(d-1)}2^{n(1-\delta)/(d-1)}\right) .
\end{split} \]
For non-space-filling SLE, applying this and \eqref{eq:cont-upper-tail} with $T=u^\delta 2^{n\delta}$, we see that the sum \eqref{eq:var_chordal_est} is bounded by
\[
o^\infty(u^{-1}) + \exp\left(-cu^{(d-\delta)/(d-1)}\right) .
\]
In case of space-filling SLE, we proceed similarly, but can now use a deterministic $T$. By \eqref{eq:psi_scaling}, the same bound holds with $\wt{c}$ replaced by $1$.
\end{proof}

In the last part of the subsection, we prove \cref{thm:var_cont}. Let us denote
\[ V_\eta(s,t) = \lim_{\delta\downarrow 0} \sup_{\substack{s \le t_0 < ... < t_r \le t\\ \abs{t_{i+1}-t_i}<\delta}} \sum_i \psi(\abs{\eta(t_{i+1})-\eta(t_i)}) . \]

\begin{lemma}\label{le:var_cont_interior}
Let $\eta$ be a chordal \slek{}, $\kappa>0$ (either space-filling or non-space-filling) in a domain $D$. Then almost surely the following holds: For any $s<t$, if $\eta[s,t] \subseteq D$, then $V_\eta(s,t) = c_0 \Cont_d(\eta[s,t])$ where $c_0>0$ is the constant in \cref{thm:main_var}.
\end{lemma}

\begin{proof}
We have seen in \cref{thm:main_var} that this property holds for the SLE variants \ref{it:non-spfill}, \ref{it:spfill}. So we only need to show that this property is preserved under conformal maps. Let $f\colon D \to D'$, and $\eta$ any curve in (the interior of) $D$ such that $V_\eta(s,t) = \Cont_d(\eta[s,t])$ for all $s<t$. Since $V_{f(\eta)}$ is finitely additive, we have
\[ V_{f(\eta)}(s_0,s_r) = \sum_i V_{f(\eta)}(s_i,s_{i+1}) . \]
We can pick the mesh fine enough so that $\abs{f'} = \abs{f'(\eta(s_i))}+o(1)$ near each segment $\eta[s_i,s_{i+1}]$. Then
\[ V_{f(\eta)}(s_i,s_{i+1}) = (\abs{f'(\eta(s_i))}+o(1))^d\, V_{\eta}(s_i,s_{i+1}) = (\abs{f'(\eta(s_i))}+o(1))^d \Cont_d(\eta[s_i,s_{i+1}]) . \]
By the transformation rule for Minkowski content \eqref{eq:cont_transf} we have
\[ \Cont_d(f(\eta[s_i,s_{i+1}])) = (\abs{f'(\eta(s_i))}+o(1))^d \Cont_d(\eta[s_i,s_{i+1}]) . \]
Altogether this implies
\[ V_{f(\eta)}(s_0,s_r) = \sum_i (1+o(1)) \Cont_d(f(\eta[s_i,s_{i+1}])) = (1+o(1)) \Cont_d(f(\eta[s_0,s_r])) . \]
\end{proof}

\begin{proposition}
Let $\eta$ be a chordal \slek{}, $\kappa>0$ (either space-filling or non-space-filling) in a domain $D$ with analytic boundary. Then almost surely $V_\eta(s,t) = c_0 \Cont_d(\eta[s,t])$ for any $0<s<t<\infty$ where $c_0>0$ is the constant in \cref{thm:main_var}.
\end{proposition}

\begin{proof}
It suffices to assume $D=\bbH$ since conformal maps extend smoothly to analytic boundary parts (cf.\@ \cite[Chapter~3]{pom-boundary-book}).

We have seen in \cref{le:var_cont_interior} that $V_\eta(s,t) = c_0 \Cont(\eta[s,t])$ whenever $\eta[s,t] \subseteq \bbH$. It remains to show that the contribution from the segments close to the boundary is negligible. Since $V_\eta$ is finitely additive, we can assume $\eta[s,t] \subseteq [R^{-1},R]\times[0,R]$ for some $R>0$ (the case when $\eta[s,t] \subseteq [-R,-R^{-1}]\times[0,R]$ is the same due to symmetry).

Let $\varepsilon > 0$. Given a segment $\eta[s,t]$, divide it into finitely many subsegments $\eta[\sigma_1,\tau_1]$, $\eta[\tau_1,\sigma_2]$, $\eta[\sigma_2,\tau_2]$, $\eta[\tau_2,\sigma_3]$, ..., so that $\eta[\sigma_j,\tau_j] \subseteq \{\Im z < 2\varepsilon\}$ and $\eta[\tau_j,\sigma_{j+1}] \subseteq \{\Im z \ge \varepsilon\}$. We already know $V_\eta(\tau_j,\sigma_{j+1}) = c_0\Cont(\eta[\tau_j,\sigma_{j+1}])$. The lower bound $V_\eta(s,t) \ge c_0 \Cont(\eta[s,t])$ follows since
\[
\sum_j \Cont(\eta[\sigma_j,\tau_j]) \le \Cont(\eta \cap [-R,R]\times[0,2\varepsilon]) \searrow 0 \]
as $\varepsilon \searrow 0$ because
\[
\bbE[ \Cont(\eta \cap [-R,R]\times[0,2\varepsilon]) ] = \int_{[-R,R]\times[0,2\varepsilon]} G(z)\,dz 
\to 0 
\]
where $G$ denotes the SLE Green's function (cf.\@ \cite{lr-minkowski-content}).

To see the upper bound $V_\eta(s,t) \le c_0 \Cont(\eta[s,t])$, it remains to show that $\sum_j V_\eta(\sigma_j,\tau_j) \to 0$ as $\varepsilon\searrow 0$.

Consider $f\colon \bbH \to \bbH\setminus\fill(\wt\eta[0,1])$ and let $\wt\eta\big|_{[1,\infty)} = f(\eta)$ as in the previous proof. By \cite{lr-minkowski-content}, for almost every instance of $\wt\eta[0,1]$,
\[ 
\bbE\left[ \Cont(f(\eta\cap[-R,R]\times[0,2\varepsilon])) \mmiddle| \wt\eta[0,1] \right] 
= \int_{[-R,R]\times[0,2\varepsilon]} \abs{f'(z)}^d G(z)\,dz 
\to 0 \quad\text{as } \varepsilon\searrow 0 .
\]
Since the integrand is monotone, this convergence actually holds almost surely, i.e.
\begin{equation}\label{eq:bdry_cont_small}
\Cont(f(\eta\cap[-R,R]\times[0,2\varepsilon])) \searrow 0
\end{equation}
for almost every $\wt\eta[0,1]$ and $\eta$.

On the event that $f(\eta[\sigma_j,\tau_j])$ are away from the boundary, we know already that
\[
\sum_j V_{f(\eta)}(\sigma_j,\tau_j) = \sum_j c_0\Cont(f(\eta[\sigma_j,\tau_j])) \le c_0\Cont(f(\eta\cap[-R,R]\times[0,2\varepsilon])) .
\]
Moreover, for every $j$,
\[
V_{f(\eta)}(\sigma_j,\tau_j) = \lim_{\delta\downarrow 0} \sup_{\substack{\sigma_j \le t_0 < ... < t_r \le \tau_j\\ \abs{t_{i+1}-t_i}<\delta}} \sum_i \psi(\abs{f(\eta(t_{i+1}))-f(\eta(t_i))}) .
\]
Denoting the supremum on the right-hand side by $V^\delta_{f(\eta)}$, we have
\[
V^\delta_{f(\eta)}(\sigma_j,\tau_j) - V_{f(\eta)}(\sigma_j,\tau_j) \searrow 0 
\quad\text{as } \delta\searrow 0 .
\]
By \cref{pr:f_incr_dist} with $\lambda=1$ (in fact, any $\lambda<d$ will do) and Jensen's inequality with the function $\psi$,
\[ \begin{split}
\sum_i \psi(\abs{\eta(t_{i+1})-\eta(t_i)}) 
&\lesssim \sum_i \bbE\left[ \psi(\abs{f(\eta(t_{i+1}))-f(\eta(t_i))}) \,1_{E_{R,c}} \mmiddle| \eta \right] \\
&\le \bbE\left[ V^\delta_{f(\eta)}(\sigma_j,\tau_j) \,1_{E_{R,c}} \mmiddle| \eta \right] 
\end{split} \]
and hence
\[ V^\delta_{\eta}(\sigma_j,\tau_j) \lesssim \bbE\left[ V^\delta_{f(\eta)}(\sigma_j,\tau_j) \,1_{E_{R,c}} \mmiddle| \eta \right] . \]
Taking the $\delta\searrow 0$ limit yields
\[
V_{\eta}(\sigma_j,\tau_j) 
\lesssim \bbE\left[ V_{f(\eta)}(\sigma_j,\tau_j) \,1_{E_{R,c}} \mmiddle| \eta \right] 
= c_0\bbE\left[ \Cont(f(\eta[\sigma_j,\tau_j])) \,1_{E_{R,c}} \mmiddle| \eta \right] .
\]
Finally,
\[
\sum_j V_{\eta}(\sigma_j,\tau_j) 
\lesssim \bbE\left[ \Cont(f(\eta\cap[-R,R]\times[0,2\varepsilon])) \mmiddle| \eta \right]
\searrow 0 \quad\text{as } \varepsilon\searrow 0 
\]
where we concluded from \cref{eq:bdry_cont_small}.
\end{proof}

\subsection{Proof of \cref{pr:f_incr_dist}}
\label{sec:f_incr_dist}

In this subsection we prove \cref{pr:f_incr_dist}. Our proof uses the backward Loewner differential equation which provides an explicit description of the conformal map. We distinguish the case when the points are close to the real line ($y_1,y_2\lesssim |x_2-x_1|$ for $z_j=x_j+iy_j$, $j=1,2$) and the case when they are sufficiently away from the real line ($y_1 \vee y_2\gtrsim |x_2-x_1|$). The latter case amounts to lower bounding the derivative $\abs{f'(z)}$. This has been done in \cite{law-reverse-sle,vl-sle-hoelder} for points on the imaginary axis, and we extend the result to arbitrary points $z$ (\cref{le:hp_x_lower}). In the former case the main work is to consider points on the real line (\cref{le:bdry_dist}). We prove this by considering the times when $x_1$ resp.\ $x_2$ are swallowed, and concatenating the conformal maps. The primary tools we are using are martingale arguments, scaling invariance, and distortion estimates for conformal maps.

Let $(h_t)_{t\ge 0}$ denote the reverse \slek{} flow, i.e.
\[ dh_t(z) = \frac{-2}{h_t(z)}-\sqrt{\kappa}\,dB_t , \quad h_0(z)=z . \]
Recall that the law of the map $h_1$ agrees with the law of $f$ from \cref{pr:f_incr_dist}.

For $s\le t$, we write $h_{s,t} = h_t\circ h_s^{-1}$. We will also write $h_t$ resp.\@ $h_{s,t}$ when referring to their extension to the boundary or to their Schwarz reflection. 
For $x \in \bbR$, write $T_x = \inf\{ t\ge 0 \mid h_t(x)=0 \}$. Further, write $I_t = \{ x\in\bbR \mid T_x \le t\}$, and recall that the Schwarz reflection of $h_t$ is a conformal map defined on $\hatC\setminus I_t$.

For $r\in\bbR$, let
\begin{align*}
    \lambda &= r\left(1+\frac{\kappa}{4}\right)-r^2\frac{\kappa}{8} , \\
    \zeta &= r-r^2\frac{\kappa}{8} .
\end{align*}

The following is proved in \cite[Section~5]{vl-sle-hoelder}.
\begin{lemma}\label{le:hp_moment}
For $r < \frac{1}{2}+\frac{4}{\kappa}$ and $\lambda,\zeta$ as above, there exists $c=c(\kappa,r) > 1$ such that
\[ \bbE\left[ \abs{h_t'(iy)}^\lambda \,1_{h_t(iy) \in [-ct^{1/2},ct^{1/2}]\times[c^{-1}t^{1/2},3t^{1/2}]} \right] \asymp t^{-\zeta/2} y^\zeta \]
for $0 < y \le t^{1/2}$.
\end{lemma}

The next lemma is a variant of the previous one for points which are not on the positive imaginary axis. We only give the upper bound here; the matching lower bound will be given below in \cref{le:hp_x_lower}.
\begin{lemma}\label{le:hp_x_upper}
For $r\in\bbR$ and $R>1$, there exists $c=c(\kappa,r,R)$ such that
\[ \bbE\left[ \abs{h_t'(x+iy)}^\lambda \,1_{h_t(x+iy) \in [-Rt^{1/2},Rt^{1/2}]\times[R^{-1}t^{1/2},Rt^{1/2}]} \right] \le c\, t^{-\zeta/2} y^\zeta \left(1+\frac{x^2}{y^2}\right)^{r/2} \]
for $x+iy\in\bbH$ and $t\ge 0$.
\end{lemma}

\begin{proof}
This follows from the fact that
\[ M_t(z) = \abs{h_t'(z)}^\lambda (\Im h_t(z))^\zeta (\sin\arg h_t(z))^{-r} \]
is a martingale (cf.\@ \cite[Proposition~2.1]{law-reverse-sle}).
\end{proof}

For $x>0$ and $r\in\bbR$ the process\begin{equation}\label{eq:martingale_bessel}
M_t(x) = \abs{h_t'(x)}^{\lambda+\zeta-r}h_t(x)^r
\end{equation}
is a martingale when stopped before $\dist(x,I_t) \le \delta$ for any fixed $\delta>0$; see e.g.\ \cite[Proposition~2.1]{law-reverse-sle} and \cite[Remark 7]{sw-sle-coordinate-change} for the $z\in\bbR$-variant of the martingale. By Girsanov's theorem (\cite[Chapter~VIII]{ry-stochastic-book}), there is a probability measure $\bbP^*$ such that for every $\delta>0$ and finite stopping time $\tau$ with $\dist(x,I_\tau)\ge\delta$, we have $d\bbP^*\big|_{\cF_\tau} = \frac{M_\tau(x)}{M_0(x)}\,d\bbP\big|_{\cF_\tau}$. The law of the process $t \mapsto h_t(x)$ under $d\bbP^*$ is
\[ dh_t(x) = \frac{-2+r\kappa}{h_t(x)}-\sqrt{\kappa}\,dB^*_t  \]
for $B^*$ a standard Brownian motion, i.e.\ $(h_t(x))_{t\geq 0}$ has the law of a (time-changed) Bessel process of index $r-\frac{2}{\kappa}-\frac{1}{2}$. 
Note also that 
\eqb
\abs{h_t'(x)} \asymp \frac{\dist(h_t(x),h_t(I_t))}{\dist(x,I_t)} \le \frac{h_t(x)}{\dist(x,I_t)}
\label{eq-derivative-It-dist}
\eqe
for $x \notin I_t$ by Koebe's $1/4$-theorem and the Schwarz reflection principle.

We begin by showing \cref{pr:f_incr_dist} for points on the boundary.
\begin{lemma}\label{le:bdry_dist}
For $0 \le r < \frac{1}{2}+\frac{2}{\kappa}$ and $R>0$, there exists $c=c(\kappa,r,R)$ such that
\[ 
\bbE\left[ \abs{h_1(x_2)-h_1(x_1)}^\lambda \,1_{h_1(x_1) \in [-c,c]\times[c^{-1},2]} \right] \gtrsim x_2^r \abs{x_2-x_1}^{\lambda+\zeta-r} \ge x_2^\zeta \abs{x_2-x_1}^{\lambda}
\]
for any $0 \le x_1 < x_2 \le R$.
\end{lemma}

\begin{proof}
Recall the notation $T_{x_i} = \inf\{t \mid h_t(x_i) = 0\}$. By \eqref{eq:koebe_hp}, we have $\abs{h_{T_{x_2},1}'(h_{T_{x_2}}(x_1))} \gtrsim \abs{h_{T_{x_2},1}'(i\Im h_{T_{x_2}}(x_1))} (\sin\arg h_{T_{x_2}}(x_1))^4$, hence \cref{le:hp_moment} implies that on the event $\{ 1-T_{x_2} \ge 1/2 \ge (\Im h_{T_{x_2}}(x_1))^2 \}$ we have
\[
\bbE\left[ \abs{h_{T_{x_2},1}'(h_{T_{x_2}}(x_1))}^\lambda \,1_{h_1(x_1) \in [-c,c]\times[c^{-1},2]} \mmiddle| \cF_{T_{x_2}} \right] \gtrsim (\Im h_{T_{x_2}}(x_1))^\zeta (\sin\arg h_{T_{x_2}}(x_1))^{4\lambda} .
\]
Moreover, note that on the event $T_{x_2}\leq 1$, by the Koebe $1/4$ theorem,
\[ 
\abs{h_1(x_2)-h_1(x_1)} \gtrsim \Im h_{T_{x_2}}(x_1) \abs{h_{T_{x_2},1}'(h_{T_{x_2}}(x_1))} ,
\]
and thus
\[
\bbE\left[ \abs{h_1(x_2)-h_1(x_1)}^\lambda \mmiddle| \cF_{T_{x_2}} \right] 
\gtrsim (\Im h_{T_{x_2}}(x_1))^{\lambda+\zeta} (\sin\arg h_{T_{x_2}}(x_1))^{4\lambda} \,1_{1-T_{x_2} \ge 1/2 \ge (\Im h_{T_{x_2}}(x_1))^2} .
\]
Next, by scale invariance (and noting that $\lambda+\zeta \ge 0$ which guarantees the expectation being finite),
\[
\bbE\left[ (\Im h_{T_{x_1},T_{x_2}}(0))^{\lambda+\zeta} (\sin\arg h_{T_{x_1},T_{x_2}}(0))^{4\lambda} \,1_{T_{x_2}-T_{x_1} \le \frac{1}{32R^2}h_{T_{x_1}}(x_2)^2} \mmiddle| \cF_{T_{x_1}} \right] \asymp h_{T_{x_1}}(x_2)^{\lambda+\zeta} .
\]
Using the martingale \eqref{eq:martingale_bessel} along with \eqref{eq-derivative-It-dist} we get
\[ \begin{split}
\bbE\left[ h_{T_{x_1}}(x_2)^{\lambda+\zeta} \,1_{T_{x_1} \le 1/4}\,1_{h_{T_{x_1}(x_2)} \le 2R} \right] 
&\gtrsim \abs{x_2-x_1}^{\lambda+\zeta-r} \,\bbE\left[ M_{T_{x_1}}(x_2) \,1_{T_{x_1} \le 1/4}\,1_{h_{T_{x_1}(x_2)} \le 2R} \right] \\
&= \abs{x_2-x_1}^{\lambda+\zeta-r} x_2^r \,\bbP^*\left(T_{x_1} \le 1/4,\, h_{T_{x_1}(x_2)} \le 2R\right) .
\end{split} \]
When $r < \frac{1}{2}+\frac{2}{\kappa}$, the $\bbP^*$-law of $h_t(x_2)$ is a Bessel process of negative index which hits $0$ in finite time, hence $\bbP^*(T_{x_1} \le 1/4,\, h_{T_{x_1}(x_2)} \le 2R) \ge \bbP^*(T_{x_2} \le 1/4,\, \max_t h_t(x_2) \le 2R) \gtrsim 1$ since we are assuming $x_1 < x_2 \le R$. Furthermore,
\[
\left\{ T_{x_1} \le 1/4;\ 
h_{T_{x_1}(x_2)} \le 2R;\
T_{x_2}-T_{x_1} \le \frac{1}{32R^2}h_{T_{x_1}}(x_2)^2\right\} 
\subseteq \{1-T_{x_2} \ge 1/2 \ge (\Im h_{T_{x_2}}(x_1))^2\},
\]
where we use in particular that on the event on the left side we have $T_{x_2}-T_{x_1}\leq 1/8$. 
Combining everything, we get
\[ \begin{split}
&\bbE\left[ \abs{h_1(x_2)-h_1(x_1)}^\lambda \,1_{h_1(x_1) \in [-c,c]\times[c^{-1},2]} \right] \\
&\quad \gtrsim \bbE\left[ (\Im h_{T_{x_2}}(x_1))^{\lambda+\zeta} (\sin\arg h_{T_{x_2}}(x_1))^{4\lambda} \,1_{1-T_{x_2} \ge 1/2 \ge (\Im h_{T_{x_2}}(x_1))^2} \right]\\
&\quad \gtrsim \bbE\left[ h_{T_{x_1}}(x_2)^{\lambda+\zeta} \,1_{T_{x_1} \le 1/4}\,1_{h_{T_{x_1}(x_2)} \le 2R} \right] \\
&\quad \gtrsim x_2^r \abs{x_2-x_1}^{\lambda+\zeta-r} .
\end{split} \]
\end{proof}

\begin{lemma}\label{le:bdry_dist_hp}
For $0 \le r < \frac{1}{2}+\frac{2}{\kappa}$ and $R>0$, there exists $c=c(\kappa,r,R)$ such that
\[ 
\bbE\left[ \abs{h_2(x_2)-h_2(x_1)}^\lambda \,1_{E_{R,c}} \right] \gtrsim x_2^r \abs{x_2-x_1}^{\lambda+\zeta-r} \ge x_2^\zeta \abs{x_2-x_1}^{\lambda}
\]
for any $0 \le x_1 < x_2 \le R$ where $E_{R,c}$ denotes the event that $h_2(z) \in [-c,c]\times[c^{-1},c]$ for all $z \in [-2R,2R]\times[0,1]$.
\end{lemma}

\begin{proof}
We first apply \cref{le:bdry_dist}. By using a harmonic measure estimate, we see that $\diam(h_1([-2R,2R])) \le b$ for some deterministic $b>0$, and therefore we have $h_1([-2R,2R]) \subseteq [-2b,2b]\times[0,1]$ on the event from \cref{le:bdry_dist}. With positive probability, the map $h_{1,2}$ maps $[-2b,2b]$ strictly into $\bbH$, hence also into $[-c,c]\times [c^{-1},1]$ for some large $c$. On the intersection of both events, we see that $E_{R,c}$ holds with some possibly different $c$, and $\abs{h_{1,2}'(h_1(x_1))} \asymp 1$.
\end{proof}

The following lemma extends the previous one to points that are not on the real line but satisfy $y_1,y_2\leq c^{-1}|x_2-x_1|$. It concludes the proof of \cref{pr:f_incr_dist} for such points.
\begin{lemma}
For $0 \le r < \frac{1}{2}+\frac{2}{\kappa}$ and $R>0$, there exists $c=c(\kappa,r,R)>1$ such that
\[ 
\bbE\left[ \abs{h_2(x_2+iy_2)-h_2(x_1+iy_1)}^\lambda \,1_{E_{R,c}} \right] \gtrsim x_2^r \abs{x_2-x_1}^{\lambda+\zeta-r} \ge x_2^\zeta \abs{x_2-x_1}^{\lambda}
\]
for any $0 \le x_1 < x_2 \le R$ and $y_1,y_2 \le c^{-1}\abs{x_2-x_1}$ where $E_{R,c}$ denotes the event that $h_2(z) \in [-c,c]\times[c^{-1},c]$ for all $z \in [-2R,2R]\times[0,1]$.
\end{lemma}

\begin{proof}
We have
\[
\abs{h_2(x_2)-h_2(x_1)} \le \abs{h_2(x_2+iy_2)-h_2(x_1+iy_1)}+\abs{h_2(x_2+iy_2)-h_2(x_2)}+\abs{h_2(x_1+iy_1)-h_2(x_1)} .
\]
By Koebe's distortion theorem,
\[ \begin{split}
\abs{h_2(x_1+iy_1)-h_2(x_1)}^\lambda 
&\le \left( \int_0^{y_1} \abs{h_2'(x_1+iv)}\,dv \right)^\lambda \\
&\asymp \left( \sum_{n\in\bbN} y_1 2^{-n}\abs{h_2'(x_1+iy_1 2^{-n})} \right)^\lambda \\
&\lesssim y_1^\lambda \sum_{n\in\bbN} n^{2(\lambda-1)\vee 0} 2^{-n\lambda}\abs{h_2'(x_1+iy_1 2^{-n})}^\lambda ,
\end{split} \]
where we use H\"older's inequality in the last step. 
With \cref{le:hp_x_upper} (and using that $\lambda+\zeta-r>0$), we get
\[
\bbE\left[ \abs{h_2(x_1+iy_1)-h_2(x_1)}^\lambda \,1_{E_{R,c}} \right] \lesssim y_1^{\lambda+\zeta} \left(1+\frac{x_1^2}{y_1^2}\right)^{r/2} 
= y_1^{\lambda+\zeta-r}(x_1^2+y_1^2)^{r/2} ,
\]
and the same holds for $x_2+iy_2$. 
Hence, if $y_1,y_2 \le c^{-1}\abs{x_2-x_1}$, we get with \cref{le:bdry_dist_hp}
\[ \begin{split}
x_2^r \abs{x_2-x_1}^{\lambda+\zeta-r} 
&\lesssim \bbE\left[ \abs{h_2(x_2)-h_2(x_1)}^\lambda \,1_{E_{R,c}} \right] \\
&\lesssim \bbE\left[ \abs{h_2(x_2+iy_2)-h_2(x_1+ix_1)}^\lambda \,1_{E_{R,c}} \right] +  y_1^{\lambda+\zeta-r}(x_1^2+y_1^2)^{r/2} \\
&\lesssim \bbE\left[ \abs{h_2(x_2+iy_2)-h_2(x_1+ix_1)}^\lambda \,1_{E_{R,c}} \right] + x_2^r (c^{-1}\abs{x_2-x_1})^{\lambda+\zeta-r} .
\end{split} \]
The claim follows for sufficiently large $c$.
\end{proof}

As mentioned above, the previous lemma shows \cref{pr:f_incr_dist} in the case $y_1,y_2 \le c^{-1}\abs{x_2-x_1}$. To treat the case when $y_1$ or $y_2 \gtrsim \abs{x_2-x_1}$, we prove the following which essentially follows from the arguments in \cite{law-reverse-sle}.

\begin{lemma}\label{le:hp_x_lower}
For $r < \frac{1}{2}+\frac{2}{\kappa}$ and $R>0$, there exists $c=c(\kappa,r,R)$ such that
\[ 
\bbE\left[ \abs{h_1'(x+iy)}^\lambda \,1_{h_1(x+iy) \in [-c,c]\times[c^{-1},c]} \right] \gtrsim y^\zeta \left(1+\frac{x^2}{y^2}\right)^{r/2} 
\]
for any $x+iy \in [-R,R]\times{]0,1]}$.
\end{lemma}

The high-level idea of the proof is as follows. As in \cite[Section~8]{law-reverse-sle}, let
\[ M_t(z) = \abs{h_t'(z)}^\lambda (\Im h_t(z))^\zeta (\sin\arg h_t(z))^{-r} \]
and define the probability measure $\bbP^*$ by reweighting $\bbP$ according to $M$, in a similar fashion as in the proof of \cref{le:hp_x_upper}. 
Then
\[ \begin{split}
\bbE\left[ \abs{h_1'(z)}^\lambda \,1_{h_1(z) \in [-c,c]\times[c^{-1},c]} \right] 
&\asymp \bbE\left[ M_1(z) \,1_{h_1(z) \in [-c,c]\times[c^{-1},c]} \right] \\
&= y^\zeta \left(1+\frac{x^2}{y^2}\right)^{r/2} \bbP^*\left( h_1(z) \in [-c,c]\times[c^{-1},c] \right) .
\end{split} \]
Therefore, in order to conclude the proof it is sufficient to bound the $\bbP^*$-probability on the right-hand side from below.

\begin{proof}
By symmetry, we can assume $x \ge 0$. We follow the setup in \cite{law-reverse-sle} and let
\[
\sigma(s) = \inf\{ t \ge 0 \mid \Im h_t(z) = e^{2s} \} ,
\quad s \ge s_0 = \frac{1}{2}\log y .
\]
(Note that \cite{law-reverse-sle} uses a different time parametrisation $t \mapsto t/\kappa$ and $s \mapsto s/\kappa$.) 
Under the law $\bbP^*$, the process
\[ 
J_s = \sinh^{-1}\left(\frac{\Re h_{\sigma(s)}(z)}{\Im h_{\sigma(s)}(z)}\right) = \cosh^{-1}\left(\frac{1}{\sin\arg h_{\sigma(s)}(z)}\right)
\]
satisfies (cf.\@ \cite[eq.~(27)]{law-reverse-sle})
\eqb
dJ_s = -\left(4+\frac{\kappa}{2}-r\kappa\right) \tanh J_s \,ds + \sqrt{\kappa}\,dW^*_s
\label{eq:J-sde}
\eqe
where $W^*$ is a standard Brownian motion under the measure $\bbP^*$. 
Further (cf.\@ \cite[Lemma~5.1 and eq.~(24)]{law-reverse-sle}),
\begin{equation}\label{eq:sigma_int}
\sigma(s) = \int_{s_0}^s e^{4s'}\cosh^2 J_{s'} \,ds' .
\end{equation}
Fix $\delta>0$ small enough so that $(1-\delta)(4+\frac{\kappa}{2}-r\kappa) > 2$ (which is possible since $r < \frac{1}{2}+\frac{2}{\kappa}$). Let 
\[ S_\delta = \inf\{ s\ge s_0 \mid \tanh J_s \le 1-\delta \} . \]
Our next goal is to upper bound $\sigma(S_\delta)$. For this, define the event
\[
E_b \defeq \left\{ W^*_s-W^*_{s_0} \le 2\sqrt{(s-s_0)\log(b+s-s_0)} \text{ for all } s\ge s_0 \right\} ,
\]
so that, since $\sup_{s\in [0,1]}W^*_{s+k}-W^*_{k}$ has a Gaussian tail,
\[ \bbP^*(E_b) = 1-O(b^{-1}) > \frac{1}{2} \]
for a suitable choice of $b$. 
By \eqref{eq:J-sde} and $J_{s_0}=\sinh^{-1}(x/y)$, on the event $E_b$ we must have
\[
J_s \le \sinh^{-1}(x/y)-\left(4+\frac{\kappa}{2}-r\kappa\right)(1-\delta)(s-s_0) + \sqrt{\kappa}2\sqrt{(s-s_0)\log(b+s-s_0)}
\]
for $s \in [s_0,S_\delta]$, and hence by \eqref{eq:sigma_int} (and using that $4-2(4+\frac{\kappa}{2}-r\kappa)(1-\delta) < 0$ due to our choice of $\delta$)
\[ \begin{split}
\sigma(S_\delta) 
&\le \int_{s_0}^{S_\delta} e^{4s'}\exp\bigg( 2\sinh^{-1}(x/y) - 2\left(4+\frac{\kappa}{2}-r\kappa\right)(1-\delta)(s'-s_0) \\
&\hphantom{\le \int_{s_0}^{S_\delta} e^{4s'}\exp\bigg( 2\sinh^{-1}(x/y)} 
+\sqrt{\kappa}2\sqrt{(s'-s_0)\log(b+s'-s_0)} \bigg) \,ds' \\
&= \exp\left(2\sinh^{-1}(x/y)\right) e^{4s_0} \int_{s_0}^{S_\delta} \exp\bigg( 4(s'-s_0) - 2\left(4+\frac{\kappa}{2}-r\kappa\right)(1-\delta)(s'-s_0) + \dots \bigg) \,ds' \\
&\lesssim \left(1+\frac{x^2}{y^2}\right) e^{4s_0} \\
&\le \wt{c}\abs{z}^2 
\end{split} \]
for some $\wt{c} > 0$.
Given this, it follows from \cref{le:hp_moment} (applied to $h'_{\sigma(S_\delta),\wt{c}\abs{z}^2+1}$ conditionally on $\cF_{\sigma(S_\delta)}$) and \eqref{eq:koebe_hp} that 
\[ \begin{split}
\bbE\left[ \abs{h_{\tilde{c}\abs{z}^2+1}'(z)}^\lambda \,1_{h_{\tilde{c}\abs{z}^2+1}(z) \in [-c,c]\times[c^{-1},c]} \right] 
&\gtrsim \bbE\left[ \abs{h_{\sigma(S_\delta)}'(z)}^\lambda (\Im h_{\sigma(S_\delta)}(z))^\zeta \,1_{\sigma(S_\delta) \le \tilde{c}\abs{z}^2} \right] \\
&\asymp \bbE\left[ M_{\sigma(S_\delta)} \,1_{\sigma(S_\delta) \le \tilde{c}\abs{z}^2} \right] \\
&= y^\zeta \left(1+\frac{x^2}{y^2}\right)^{r/2} \bbP^*(\sigma(S_\delta) \le \tilde{c}\abs{z}^2) \\
&\gtrsim y^\zeta \left(1+\frac{x^2}{y^2}\right)^{r/2} .
\end{split} \]
Since $\abs{z} \le R+1$, the statement of the lemma (with time $1$ in place of $\wt{c}\abs{z}^2+1$) follows by scaling.
\end{proof}

\begin{lemma}
For $r < \frac{1}{2}+\frac{2}{\kappa}$ and $R>0$, there exists $c=c(\kappa,r,R)$ such that
\[ 
\bbE\left[ \abs{h_2'(x+iy)}^\lambda \,1_{E_{R,c}} \right] \gtrsim y^\zeta \left(1+\frac{x^2}{y^2}\right)^{r/2} 
\]
for any $x+iy \in [-R,R]\times{]0,1]}$ where $E_{R,c}$ denotes the event that $h_2(z) \in [-c,c]\times[c^{-1},c]$ for all $z \in [-2R,2R]\times[0,1]$.
\end{lemma}

\begin{proof}
This follows from \cref{le:hp_x_lower} in the same way as \cref{le:bdry_dist_hp} follows from \cref{le:bdry_dist}.
\end{proof}

This completes the proof of \cref{pr:f_incr_dist} in the case when $y_1$ or $y_2 \gtrsim \abs{x_2-x_1}$. Indeed, write $z_j=x_j+iy_j$ and note that we must have $y_j \gtrsim \abs{z_2-z_1}$ for $j=1$ or $j=2$ as well (since in case $\abs{x_2-x_1} < \abs{z_2-z_1}/2$ we can pick the larger $y_j$). Koebe's $1/4$-theorem (applied to a ball of radius $\abs{z_2-z_1}\wedge y_j$) then implies $\abs{h_2(z_1)-h_2(z_2)} \gtrsim \abs{h_2'(z_j)}\abs{z_1-z_2}$, hence
\[ \begin{split}
\bbE\left[ \abs{h_2(z_1)-h_2(z_2)}^\lambda \,1_{E_{R,c}} \right] 
&\gtrsim \abs{z_1-z_2}^\lambda \bbE\left[ \abs{h_2'(z_j)}^\lambda \,1_{E_{R,c}} \right] \\
&\gtrsim \abs{z_1-z_2}^\lambda y_j^{\zeta-r}(x_j^2+y_j^2)^{r/2} \\
&\gtrsim \abs{z_1-z_2}^\lambda
\end{split} \]
since we are assuming $\abs{z_j} \ge R^{-1}$.

\bibliographystyle{alpha}

\end{document}